\documentclass[12pt]{amsart}
\usepackage[margin=1in]{geometry}
\usepackage{amssymb,amsfonts,amsmath,mathtools}
\usepackage{color}
\usepackage{soul}

\usepackage{enumerate}
\usepackage{mathrsfs}
\usepackage[breaklinks=true,colorlinks=true,linkcolor=blue,citecolor=blue,filecolor=blue,urlcolor=blue]{hyperref}
\hypersetup{linktocpage}
\usepackage[hiresbb]{graphicx,xcolor}
\usepackage{subcaption}
\usepackage[capitalise]{cleveref}
\let\C\undefined
\usepackage{constants}
\newconstantfamily{abcon}{symbol=c}
\newconstantfamily{abcon2}{symbol=C}
\usepackage{bbm}
\usepackage{tikz-cd}
\usepackage{tikz}
\usetikzlibrary{shapes.geometric}
\usetikzlibrary {shapes.misc}
\usetikzlibrary{positioning}
\allowdisplaybreaks[4]

\newtheorem{theorem}{Theorem}[section]
\newtheorem{corollary}[theorem]{Corollary}

\newtheorem{lemma}[theorem]{Lemma}

\newtheorem{conjecture}[theorem]{Conjecture}

\newtheorem{question*}{Question}
\newtheorem{problem*}{Problem}

\theoremstyle{definition}

\theoremstyle{remark}
\newtheorem*{remark}{Remark}

\numberwithin{equation}{section}

\crefname{figure}{Figure}{Figures}
\theoremstyle{plain}
\newtheorem*{theorem*}{Theorem}
\crefname{theorems}{Theorem}{Theorems}
\crefname{corollaries}{Corollary}{Corollaries}
\newtheorem*{corollary*}{Corollary}
\crefname{corollaries*}{Corollary}{Corollaries}
\crefname{lemma}{Lemma}{Lemmata}
\crefname{proposition}{Proposition}{Propositions}
\crefname{conjectures}{Conjecture}{Conjectures}
\newtheorem*{conjonjecture*}{Conjecture}
\crefname{conjonjectures*}{Conjecture}{Conjectures}
\crefname{definitions}{Definition}{Definitions}
\crefname{hypotheses}{Hypothesis}{Hypotheses}

\renewcommand{\hat}{\widehat}

\newcommand{\Z}{\mathbb{Z}}
\renewcommand{\C}{\mathbb{C}}
\newcommand{\R}{\mathbb{R}}
\newcommand{\Q}{\mathbb{Q}}

\newcommand{\GL}{\mathrm{GL}}
\newcommand{\PSL}{\mathrm{PSL}}

\newcommand{\SU}{\mathrm{SU}}

\newcommand{\SL}{\mathrm{SL}}

\newcommand{\A}{\mathbb{A}}

\renewcommand{\tilde}{\widetilde}

\renewcommand{\epsilon}{\varepsilon}

\newcommand{\Li}{\mathrm{Li}}

\renewcommand{\Im}{\mathrm{Im}}

\renewcommand{\pmod}[1]{\, (\mathrm{mod} {\, #1})}

\renewcommand{\Re}{\mathrm{Re}}

\def\Res{\mathop{\mathrm{res}}}

\newcommand{\supp}{\mathrm{supp}\,}

\newcommand{\tr}{\mathrm{tr}}

\newcommand{\vol}{\text{\rm vol}}

\makeatletter\newcommand{\tpmod}[1]{{\@displayfalse\pmod{#1}}}

\begin{document}

\title[PGT for the Picard Manifold and Spectral Exponential Sums]
{The Prime Geodesic Theorem for $\PSL_{2}(\Z[i])$ and Spectral Exponential Sums}

\author{Ikuya Kaneko}
\address{The Division of Physics, Mathematics and Astronomy, California Institute of Technology, 1200 E. California Blvd., Pasadena, CA 91125, USA}
\email{ikuyak@icloud.com}
\urladdr{\href{https://sites.google.com/view/ikuyakaneko/}{https://sites.google.com/view/ikuyakaneko/}}
\thanks{The author acknowledges the support of the Masason Foundation.}

\subjclass[2010]{11F72 (primary); 11M36, 11L05 (secondary)}

\keywords{Prime Geodesic Theorem, Picard manifold, second moment, $L$-functions, Selberg trace formula, Kuznetsov formula, Kloosterman sums, spectral exponential sums, subconvexity}

\date{\today}

\dedicatory{}

\begin{abstract}
This work addresses the Prime Geodesic Theorem for the Picard manifold~$\mathcal{M} = \PSL_{2}(\Z[i]) \backslash \mathfrak{h}^{3}$, which asks for the asymptotic evaluation of a counting function for the closed geodesics on $\mathcal{M}$. Let $E_{\Gamma}(X)$ be the error term in the Prime Geodesic Theorem. We~establish that $E_{\Gamma}(X) = O_{\epsilon}(X^{3/2+\epsilon})$ on average as well as many pointwise bounds.~The~second moment bound parallels an analogous result for $\Gamma = \PSL_{2}(\Z)$ due to Balog et al. and our innovation features the delicate analysis of sums of Kloosterman sums with an explicit manipulation~of oscillatory integrals. The proof of the pointwise bounds requires Weyl-strength subconvexity for quadratic Dirichlet $L$-functions over $\Q(i)$. Moreover, an asymptotic formula for a spectral exponential sum in the spectral aspect for a cofinite Kleinian group $\Gamma$ is given. Our numerical experiments visualise in particular that $E_{\Gamma}(X)$ obeys a conjectural bound of size $O_{\epsilon}(X^{1+\epsilon})$.
\end{abstract}

\maketitle
\tableofcontents

\section{Introduction}\label{introduction}

\subsection{Historical Prelude}\label{historical-prelude}
The Prime Geodesic Theorem concerns the asymptotic behaviour as $X \to \infty$ for the number $\pi_{\Gamma}(X)$ of primitive closed geodesics on hyperbolic manifolds.~The most classical case is when $\Gamma \subset \PSL_{2}(\R)$ is a cofinite Fuchsian group acting on the upper~half-plane $\mathfrak{h}^{2}$. In retrospect, this problem was intensively studied in~\cite{Hejhal1976-2,Hejhal1976,Hejhal1983,Huber1961,Huber1961-2,Kuznetsov1978,Venkov1990,Selberg1989}. Let $\Lambda_{\Gamma}(P) = \log N(P_{0})$ if $\{P \}$ is a power of the underlying primitive hyperbolic~class~$\{P_{0} \}$ and $\Lambda_{\Gamma}(P) = 0$ otherwise. Then the summatory functions \`{a} la Chebyshev are defined by
\begin{equation*}
\Theta_{\Gamma}(X) = \sum_{N(P_{0}) \leq X} \log N(P_{0}), \qquad 
\Psi_{\Gamma}(X) = \sum_{N(P) \leq X} \Lambda_{\Gamma}(P),
\end{equation*}
where $N(P)$ stands for the norm of $P$ so that $\log N(P_{0})$ is the length of the primitive closed geodesic $P_{0}$. The Prime Geodesic Theorem are viewed as a geometric analogue of the Prime Number Theorem; whence the norms are sometimes called \textit{pseudoprimes}. Selberg~\cite{Selberg1989} showed via the analysis of his trace formula that
\begin{equation*}
\Psi_{\Gamma}(X) = \sum_{1/2 < s_{j} \leq 1} \frac{X^{s_{j}}}{s_{j}}+E_{\Gamma}(X),
\end{equation*}
where the main term comes from the small eigenvalues of the hyperbolic Laplacian acting on $L^{2}(\Gamma \backslash \mathfrak{h}^{2})$ and $E_{\Gamma}(X)$ is an error term. It is well known that $E_{\Gamma}(X) = O(X^{3/4})$ for~any~cofinite $\Gamma$. This barrier is called the trivial bound. Given an analogue of the Riemann Hypothesis~for the Selberg zeta function apart from a finite number of exceptional zeroes,~it is reasonable to expect that $E_{\Gamma}(X) = O_{\epsilon}(X^{1/2+\epsilon})$ for $\Gamma = \PSL_{2}(\Z)$; see the work of~Kaneko--Koyama~\cite{KanekoKoyama2022} for evidence. This remains an open problem since being posed by Iwaniec~\cite{Iwaniec1984} in 1984 and is out of reach with current technology due to the abundance of eigenvalues.

As a precursor, Iwaniec \cite{Iwaniec1984} showed via the Kuznetsov formula that $E_{\Gamma}(X) = O_{\epsilon}(X^{35/48+\epsilon})$ for $\Gamma = \PSL_{2}(\Z)$. He also remarked in~\cite{Iwaniec1984-2} that the exponent $2/3$ follows from the~Generalised Lindel\"{o}f Hypothesis (GLH) for quadratic Dirichlet $L$-functions. Moreover, if an analogue of GLH for Rankin--Selberg $L$-functions is assumed, then the Weil bound for Kloosterman~sums yields the same exponent. Following the ideas of Iwaniec~\cite{Iwaniec1984}, one can also improve upon his exponent; see~\cite{Cai2002,LuoSarnak1995}. The crucial gist in all of these pursuits was to show a nontrivial~bound for the spectral exponential sum by applying various versions of spectral summation formul{\ae}. On the other hand, Soundararajan--Young~\cite{SoundararajanYoung2013} have subsequently proven the bound
\begin{equation}\label{Soundararajan-Young}
E_{\Gamma}(X) \ll_{\epsilon} X^{\frac{2}{3}+\frac{\vartheta}{6}+\epsilon},
\end{equation}
where $\vartheta$ signifies a subconvex exponent for quadratic Dirichlet $L$-functions. This justifies~the exponent $25/36$ in view of the work of Conrey--Iwaniec~\cite{ConreyIwaniec2000}. The proof of~\eqref{Soundararajan-Young} leverages~the Kuznetsov--Bykovski\u{\i} formula (see~\cite{Bykovskii1994,Kuznetsov1978,SoundararajanYoung2013}), leaving the theory of the Selberg zeta function aside. This result marks the current record amongst many hitherto established.

\subsection{Statement of Results}\label{statement-of-results}
We now proceed to the rigorous statement of our results,~which necessitates a bit of notation. This work unveils an approach to the Prime Geodesic Theorem associated to the Picard manifold $\mathcal{M} = \PSL_{2}(\mathcal{O}) \backslash \mathfrak{h}^{3}$ with $\mathcal{O} = \Z[i]$ and the upper half-plane $\mathfrak{h}^{3} \cong \SL_{2}(\C)/\SU_{2}(\C) \cong \mathrm{Z}(\C) \backslash \GL_{2}(\C)/\mathrm{U}_{2}(\C)$, viewed as a subset of Hamiltonian quaternions with vanishing fourth coordinate. For an explanation in the general setting, let $\Gamma \subset \PSL_{2}(\C)$ be a cofinite Kleinian group and let $\Psi_{\Gamma}(X)$ denote the analogous counting function associated to $\Gamma$, which counts hyperbolic and loxodromic (not necessarily primitive) conjugacy classes~of $\Gamma$. In our scenario, the small eigenvalues $\lambda_{j} = s_{j}(2-s_{j}) = 1+t_{j}^{2} < 1$ provide a finite number of terms that form the main term in $\Psi_{\Gamma}(X)$, namely
\begin{equation*}
E_{\Gamma}(X) = \Psi_{\Gamma}(X)-\sum_{1 < s_{j} \leq 2} \frac{X^{s_{j}}}{s_{j}}.
\end{equation*}
In a major breakthrough, Sarnak~\cite{Sarnak1983} showed that $E_{\Gamma}(X) = O_{\epsilon}(X^{5/3+\epsilon})$ for cofinite Kleinian groups. In the case of the Picard group, the Kuznetsov formula (see~\S\ref{Kuznetsov-formula}) enables one to lower the exponent $5/3$. Balkanova et al.~\cite{BalkanovaChatzakosCherubiniFrolenkovLaaksonen2019} have obtained the exponent $13/8$ via the utilisation~of the classical method of Luo--Sarnak~\cite{LuoSarnak1995}. A heuristic deduction of the expected exponent~$3/2$ was offered by Balkanova--Frolenkov~\cite{BalkanovaFrolenkov2019-2}. Balog et al.~\cite{BalogBiroCherubiniLaaksonen2022} have refined their work, deriving
\begin{equation}\label{Balog-Biro-Cherubini-Laaksonen}
E_{\Gamma}(X) \ll_{\epsilon} X^{\frac{3}{2}+\frac{4\theta}{7}+\epsilon},
\end{equation}
where $\theta$ is a subconvex exponent satisfying
\begin{equation}\label{subconvexity-for-Dirichlet-L-functions}
L \left(\frac{1}{2}+it, \chi_{D} \right) \ll_{\epsilon} (1+|t|)^{A} N(D)^{\theta+\epsilon}
\end{equation}
for all primitive quadratic characters $\chi_{D}$ over $\Q(i)$ with $D$ the discriminant $\in \mathcal{O} \setminus \{0 \}$. In what follows, the letter $\delta$ means that $E_{\Gamma}(X) = O_{\epsilon}(X^{\delta+\epsilon})$. If we insert the convex exponent~$\theta = 1/4$ in the bound~\eqref{Balog-Biro-Cherubini-Laaksonen}, then we obtain $\delta = 23/14$. In order to make $\theta$ smaller, we need to prove subconvex bounds as in the antecedent work in the literature. Michel--Venkatesh~\cite{MichelVenkatesh2010} solved this problem for $L$-functions on $\GL_{1}$ and $\GL_{2}$ over number fields with a subconvex~exponent unspecified. In his tour de force work, Nelson~\cite{Nelson2020} has recently shown spectral reciprocity~for a cubic moment of central values of automorphic $L$-functions on $\GL_{2}$, obtaining Weyl-strength subconvexity that reinforces the work of Conrey--Iwaniec~\cite{ConreyIwaniec2000} and Petrow--Young~\cite{PetrowYoung2020,PetrowYoung2022}.

Koyama~\cite{Koyama2001} improved Sarnak's exponent to $\delta = 11/7$ conditionally on the \textit{mean-Lindel\"{o}f hypothesis} for symmetric square $L$-functions associated to Hecke--Maa{\ss} cusp forms on $\Gamma \backslash \mathfrak{h}^{3}$. Our work seeks an unconditional proof of $\delta = 3/2$ on average in emulation of previous work.
\begin{theorem}\label{second-moment}
Let $\Gamma = \PSL_{2}(\mathcal{O})$ and $0 \leq \eta \leq 1$ be an additional exponent in the bound~\eqref{mean-value-estimate} towards the mean-Lindel\"{o}f hypothesis. For $1 \ll \Delta \leq V^{(3+2\eta)/4}$ and any $\epsilon > 0$, we have that
\begin{equation}\label{main}
\frac{1}{\Delta} \int_{V}^{V+\Delta} |E_{\Gamma}(X)|^{2} dX 
\ll_{\epsilon} V^{4-\frac{2}{5+2\eta}+\epsilon} \Delta^{-\frac{4}{5+2\eta}}.
\end{equation}
If we assume $\theta = 0$ and $V^{3/10} \leq \Delta \leq V^{19/20}$, then we have that
\begin{equation}\label{main-smooth-explicit-formula}
\frac{1}{\Delta} \int_{V}^{V+\Delta} |E_{\Gamma}(X)|^{2} dX 
\ll_{\epsilon} V^{\frac{44}{13}+\epsilon} \Delta^{-\frac{8}{13}}.
\end{equation}
\end{theorem}

The first assertion~\eqref{main} was announced in the author's report~\cite{Kaneko2020}. On the other hand,~the estimate~\eqref{main-smooth-explicit-formula} is shown via the smooth explicit formula in~\cite{BalogBiroCherubiniLaaksonen2022}, thus resulting in better bounds in a bulk range of $\Delta$. A generalised version of~\eqref{main-smooth-explicit-formula}~is given in Theorem~\ref{second-moment-2} that contains both the parameters $\eta$ and $\theta$. For any cofinite $\Gamma$, Balkanova et al.~\cite{BalkanovaChatzakosCherubiniFrolenkovLaaksonen2019} have shown that the second moment is $O(V^{18/5} \Delta^{-2/5} (\log V)^{2/5})$, which is weaker than Theorem~\ref{second-moment} when $\Gamma = \PSL_{2}(\mathcal{O})$. Their argument relies on the Selberg trace formula, whereas we use the Kuznetsov formula over Gaussian integers instead --- the latter is known to be more powerful for estimations. Direct corollaries of~\eqref{main} include the following.
\begin{corollary}\label{corollary-of-Theorem-1.1}
Keep the notation as above. For $1 \ll \Delta \leq V^{11/12}$ and any $\epsilon > 0$, we~have~that
\begin{equation}\label{3/2-average}
\frac{1}{\Delta} \int_{V}^{V+\Delta} |E_{\Gamma}(X)|^{2} dX 
\ll_{\epsilon} V^{\frac{62}{17}+\epsilon} \Delta^{-\frac{12}{17}}.
\end{equation}
\end{corollary}

Corollary~\ref{corollary-of-Theorem-1.1} improves upon~\cite[Theorem 1.1]{ChatzakosCherubiniLaaksonen2022}. The mean-Lindel\"{o}f hypothesis asserts~that $\eta = 0$, which would yield a better bound. However, we are unable to establish this conjecture with current technology. Balkanova et al.~\cite[Remark 1.4]{BalkanovaChatzakosCherubiniFrolenkovLaaksonen2019} have also verified that the second moment bound in short intervals of the type $O_{\epsilon}(V^{\beta+\epsilon} \Delta^{-\gamma})$ gives the pointwise bound
\begin{equation}\label{mean-to-max}
E_{\Gamma}(X) \ll_{\epsilon} X^{\alpha+\epsilon}, \qquad \alpha = \frac{\beta+\gamma}{2+\gamma}.
\end{equation}
The bound~\eqref{main} with $\eta = 1/2$ yields $\delta = 13/8$, which is in accordance with~\cite[Theorem~1.1]{BalkanovaChatzakosCherubiniFrolenkovLaaksonen2019}. The current record $\eta = 1/3$ due to Balkanova--Frolenkov~\cite{BalkanovaFrolenkov2022} gives the stronger exponent~$\delta = 37/23$. In general, an application of the second moment bound~\eqref{main} to $V^{\beta+\epsilon} \Delta^{-\gamma}$ conduces~to
\begin{equation}\label{restatement}
E_{\Gamma}(X) \ll_{\epsilon} X^{\frac{11+4\eta}{7+2\eta}+\epsilon}.
\end{equation}
This is a restatement of the bound in~\cite[p.5363]{BalkanovaChatzakosCherubiniFrolenkovLaaksonen2019} with a different proof and we obtain~$\delta = 11/7$ conditionally on the mean-Lindel\"{o}f hypothesis. Theorem~\ref{second-moment} therefore involves a refinement of the work of Koyama~\cite{Koyama2001} as well. For more discussions, we refer the reader to~\S\ref{sums-of-Kloosterman-sums}.
\begin{remark}
It would be interesting to delve into second moment bounds in terms of subconvex exponents $\eta$ and $\theta$ in the full range of $\Delta$. At the moment, we only know bounds with a~certain restriction on $\Delta$ as in~\eqref{main-smooth-explicit-formula}. We relegate this moot issue to future work.
\end{remark}

The following pointwise bounds surpass the aforementioned exponent $\delta = 37/23$.
\begin{theorem}\label{main-2}
Keep the notation as above. For any $\epsilon > 0$, we have that
\begin{equation}\label{eta-theta}
E_{\Gamma}(X) \ll_{\epsilon} X^{\frac{3}{2}+\frac{24\theta+2\eta(1+8\theta)-1}{46+20\eta}+\epsilon}.
\end{equation}
In particular, the current records $\theta = 1/6$ and $\eta = 2\theta = 1/3$ give
\begin{equation}\label{best-known}
E_{\Gamma}(X) \ll_{\epsilon} X^{\frac{376}{237}+\epsilon}.
\end{equation}
If we assume $\theta = 0$, then we have that
\begin{equation}\label{bound-under-MLH}
E_{\Gamma}(X) \ll_{\epsilon} X^{\frac{34}{23}+\epsilon}.
\end{equation}
\end{theorem}

This serves as a counterpart of~\cite[Theorem 1.1]{SoundararajanYoung2013} and a generalisation of~\cite[Corollaries~1.4 \& 1.5]{BalogBiroCherubiniLaaksonen2022}. The exponent $\eta = 1/3$ was proven in the recent work of Balkanova--Frolenkov~\cite{BalkanovaFrolenkov2022}. The method of proof of~\eqref{best-known} is inspired by Nelson~\cite{Nelson2020}, and we could make do with Weyl-strength subconvex bounds for quadratic Dirichlet $L$-functions over $\Q(i)$ in view of~\eqref{Balog-Biro-Cherubini-Laaksonen}.~Theorem~\ref{main-2} gives the state-of-the-art unconditional pointwise bound for $E_{\Gamma}(X)$.

For $X \geq 1$, the spectral exponential sum is defined by
\begin{equation}\label{spectral-exponential-sums}
S(T, X) \coloneqq \sum_{t_{j} \leq T} X^{it_{j}}.
\end{equation}
We wander back and forth between second moment bounds for $E_{\Gamma}(X)$ and $S(T, X)$ in short intervals via an explicit formula. The trivial bound becomes $S(T, X) = O(T^{3})$ and we~cannot bound $S(T, X)$ by summing up the terms with absolute values. The next result demonstrates that the spectral exponential sum $S(T, X)$ obeys a conjectural bound in the \textit{spectral aspect}\footnote{This type of asymptotic formula was shown for a cofinite Fuchsian group $\Gamma \subset \PSL_{2}(\R)$ by the author~\cite{Kaneko2022-2}.}.
\begin{theorem}\label{Laaksonen-analogue}
Define
\begin{align*}
\mathcal{S}(T) &\coloneqq \frac{1}{\pi} \arg Z(1+iT),\\
G(T) &\coloneqq \int_{1}^{2+\xi} \log |Z(\sigma+iT)| d\sigma, \qquad \xi = (\log X)^{-1},\\
\widehat{\Lambda_{\Gamma}}(X) &\coloneqq \sum_{N(T) = X} \frac{N(T) \Lambda_{\Gamma}(T)}{m(T)|a(T)-a(T)^{-1}|^{2}},
\end{align*}
where $Z(s)$ is the Selberg zeta function of $\mathcal{M}$. Moreover, we define~$\Lambda_{K}(X) = (1+\chi_{4}(X)) \log p$ when $X$ equals a power of a prime $p$ (where $\chi_{4}$ is the primitive Dirichlet character modulo~$4$) and $\Lambda_{K}(X) = 0$ otherwise. Then we have for fixed $X \geq 1$ that
\begin{equation}\label{Fujii-analogue}
S(T, X) = \frac{\vol(\mathcal{M})}{2\pi^{2} i \log X} X^{iT} T^{2}
 + \frac{\vol(\mathcal{M})}{2\pi(\log X)^{2}} X^{iT} T
 + X^{iT} \mathcal{S}(T)+\frac{T}{2\pi X}(\widehat{\Lambda_{\Gamma}}(X)+\Lambda_{K}(X))+O(G(T)),
\end{equation}
where the implicit constant is independent on $T$.
\end{theorem}

Theorem~\ref{Laaksonen-analogue} can be interpreted as a three-dimensional counterpart of the work of Fujii~\cite{Fujii1984} and Laaksonen~\cite[Appendix]{PetridisRisager2017}\footnote{An astute reader might ask which result gives asymptotic formul{\ae} of better quality. The work of Fujii~\cite{Fujii1984} entirely covers the work of Laaksonen~\cite{PetridisRisager2017} since both the sine and cosine kernels are handled in the former.} and reinforces the adequacy of the expected bound~$E_{\Gamma}(X) = O_{\epsilon}(X^{3/2+\epsilon})$. The first term in the identity~\eqref{Fujii-analogue} stems from the main term in the Weyl law. Theorem~\ref{Laaksonen-analogue} implies that $S(T, X)$ may be used as an indirect way of detecting pseudoprimes. Furthermore, the spectral exponential sum attains to a \textit{peak} of order $T$ when $X$ is a power of a norm of a primitive hyperbolic or loxodromic element in $\Gamma$, or a power of a prime~number $p \equiv 1 \tpmod 4$. Therefore, there exists a hidden connection between the spectral parameters $t_{j}$ and the length spectrum in $\mathfrak{h}^{3}$.

\subsection{Structure of the Article}\label{structure-of-the-article}
In~\S\ref{automorphic-machinery}, we configure our automorphic toolbox that deserves a quick perusal by the reader who is not familiar with the subject of this article. The~proof~of Theorem~\ref{second-moment} is executed in~\S\ref{proof-of-Theorem-1.1} and our overall scheme differs from that in two dimensions.~In practice, we apply the Kuznetsov formula for $\PSL_{2}(\mathcal{O})$ to reduce the problem to an analysis of sums of Kloosterman sums over Gaussian integers. This is followed by explicit manipulations of certain integrals involving special functions. In~\S\ref{corollaries-and-Theorem-1.3}, we give a proof of Theorem~\ref{main-2} as well~as applications of Theorem~\ref{second-moment}. The proof of Theorem~\ref{main-2} follows the process developed in~\cite{BalogBiroCherubiniLaaksonen2022} except that we keep track of $\eta$. In~\S\ref{conclusions-and-conjectures}, Theorem~\ref{Laaksonen-analogue} is shown and the spectral exponential~sum in both $X$- and spectral aspects is studied. All implements in~\S\ref{conclusions-and-conjectures} are old-fashioned in the sense that a refined version of the Weyl law plays a crucial r\^{o}le. One may catch more~information about the size of $S(T, X)$ in both aspects from our plots of its scaled versions.

\subsection{Notational Conventions}\label{notational-conventions}
All implicit constants throughout the article are allowed to depend on $\epsilon > 0$ that represents an arbitrarily small positive quantity, but not necessarily the same at each occurrence. The Vinogradov symbol $A \ll B$ or the big $O$ notation $A = O(B)$ signifies that $|A| \leq C|B|$ for some constant $C$. We use the standard symbol $\asymp$ to mean both $\ll$ and $\gg$. We define the Kronecker symbol $\delta_{S}$ according as whether the statement $S$ is true or not. For example, $\delta_{a = b}$ equals 1 if $a = b$ and $0$ otherwise. We write $e(x) \coloneqq \exp(2\pi ix)$ for $x \in \C$. Henceforth, we exploit the sign convention $t_{j} > 0$ for spectral parameters along with the notation $\sum_{|t_{j}| \leq T}$ to indicate that the sum is symmetrised by including both $t_{j}$ and $-t_{j}$.

\section{Automorphic Machinery}\label{automorphic-machinery}

\subsection{Underlying Knowledge}\label{underlying-knowledge}
For the basis $\{1, i, j, k \}$ of Hamiltonian quaternions, a point $v \in \mathfrak{h}^{3}$ is represented by $v = z+rj$ with $r > 0$ and $z = x+yi \in \C$; thus $\Im(v) \coloneqq r$, $\Im(z) \coloneqq y$ and $\Re(z) \coloneqq x$. The space $\mathfrak{h}^{3}$ is endowed with the hyperbolic metric $r^{-1} \sqrt{dx^{2}+dy^{2}+dr^{2}}$ and the corresponding volume element $d\mu(v) = r^{-3} dx dy dr$. In the following lines, for brevity, we define $K = \Q(i)$ (whose ring of integers is $\mathcal{O}_{K} = \mathcal{O} = \Z[i]$ so that $\Gamma = \PSL_{2}(\mathcal{O})$ unless otherwise stated). We regard $\mathcal{O}$ as a lattice in $\R^{2}$, which is denoted by $L$ with the fundamental domain $F_{L}$. The group $\PSL_{2}(\C)$ acts on $\mathfrak{h}^{3}$ by the orientation-preserving isometric action
\begin{equation*}
gv = (av+b)(cv+d)^{-1}, \qquad g = \begin{pmatrix} a & b \\ c & d \end{pmatrix} \in \PSL_{2}(\C).
\end{equation*}
Then the quotient $\mathcal{M} = \Gamma \backslash \mathfrak{h}^{3}$ amounts to a three-dimensional arithmetic hyperbolic~manifold called the Picard manifold. The hyperbolic Laplacian on $\mathcal{M}$ is defined by
\begin{equation*}
\Delta \coloneqq 
 - y^{2} \left(\frac{\partial^{2}}{\partial x^{2}}+\frac{\partial^{2}}{\partial y^{2}}+\frac{\partial^{2}}{\partial r^{2}} \right)
 + r \frac{\partial}{\partial r},
\end{equation*}
which possesses a self-adjoint extension on $L^{2}(\mathcal{M})$. The spectrum on $\Delta$ consists of~discrete and continuous spectra. We denote the family of Maa{\ss} cusp forms by $\{u_{j}(v) \colon j = 1, 2, 3, \dots \}$ associated to the eigenvalue $\lambda_{j} = 1+t_{j}^{2}$ with the sign convention $t_{j} > 0$. We shall assume~$u_{j}$ to be chosen so that they are simultaneous eigenfunctions of the ring of Hecke operators and $L^{2}$-normalised. The Fourier expansion of $u_{j}(v)$ reads (see~\cite[Equation (2.20)]{Sarnak1983})
\begin{equation}\label{Fourier-development}
u_{j}(v) = r \sum_{n \in \mathcal{O}^{\ast}} \rho_{j}(n) K_{it_{j}}(2\pi |n|r) e(\Re(n \overline{z})),
\end{equation}
where $K_{\nu}$ denotes the modified Bessel function of the second kind. This is based on the non-degenerate $\R$-bilinear form $(w, z) \mapsto \Re(w \overline{z})$. The Fourier coefficients $\rho_{j}(n)$ are proportional to the Hecke eigenvalues $\lambda_{j}(n)$, namely
\begin{equation}\label{proportionality}
\rho_{j}(n) = \rho_{j}(1) \lambda_{j}(n), \qquad n \in \mathcal{O}^{\ast}.
\end{equation}
The multiplicativity relation becomes
\begin{equation}\label{multiplicativity}
T_{m} T_{n} = \sum_{(d) \mid (m, n)} T_{\frac{mn}{d^{2}}}, \qquad m, n \in \mathcal{O}^{\ast},
\end{equation}
where the sum is taken over non-zero ideals in $\mathcal{O}$ that divide $m$ and $n$, and the Hecke operator $T_{n}$ on functions $f: \Gamma \backslash \mathfrak{h}^{3} \to \C$ is given by (see~\cite{Shimura1994})
\begin{equation*}
(T_{n} f)(v) = \frac{1}{|n|} \sum_{g \in R(1) \backslash R(n)} f(gv)
 = \frac{1}{4|n|} \sum_{ab = n} \sum_{b \tpmod d} f \left(\begin{pmatrix} a & b \\ 0 & d \end{pmatrix} v \right)
\end{equation*}
with
\begin{equation*}
R(n) \coloneqq \left\{\begin{pmatrix} a & b \\ c & d \end{pmatrix} \in \mathrm{M}_{2}(\mathcal{O}): 
ad-bc = n \right\}, \qquad \Gamma = R(1).
\end{equation*}

For a Hecke--Maa{\ss} cusp form $u_{j}$ with its Fourier expansion in~\eqref{Fourier-development}, we define the Rankin--Selberg convolution
\begin{equation}\label{Rankin--Selberg}
L(s, u_{j} \otimes u_{j}) \coloneqq \sum_{n \in \mathcal{O}^{\ast}} \frac{|\rho_{j}(n)|^{2}}{N(n)^{s}}
\end{equation}
with which we associate the symmetric square $L$-function
\begin{equation*}
L(s, \mathrm{sym}^{2} u_{j}) \coloneqq \sum_{n \in \mathcal{O}^{\ast}} \frac{c_{j}(n)}{N(n)^{s}} 
 = \frac{\zeta_{K}(2s)}{\zeta_{K}(s)} \sum_{n \in \mathcal{O}^{\ast}} \frac{|\lambda_{j}(n)|^{2}}{N(n)^{s}}.
\end{equation*}
This is defined in the region of absolute convergence of both the Dirichlet series, where~$c_{j}(n) = \sum_{\ell^{2} k = n} \lambda_{j}(k^{2})$. Given the convention
\begin{equation*}
\rho_{j}(n) = \sqrt{\frac{\sinh(\pi t_{j})}{t_{j}}} \nu_{j}(n), \qquad \nu_{j}(n) = \nu_{j}(1) \lambda_{j}(n),
\end{equation*}
it is convenient to recast $L(s, \mathrm{sym}^{2} u_{j})$ as
\begin{equation*}
L(s, \mathrm{sym}^{2} u_{j}) = \frac{\zeta_{K}(2s)}{\zeta_{K}(s)} 
L(s, u_{j} \otimes u_{j}) \frac{t_{j}}{\sinh(\pi t_{j})} |\nu_{j}(1)|^{-2}.
\end{equation*}
Moreover, there is an integral representation
\begin{equation*}
L(s, u_{j} \otimes u_{j}) = \frac{8\pi^{2s} \Gamma(2s)}{\Gamma(s+it_{j}) \Gamma(s-it_{j}) \Gamma^{2}(s)} 
\int_{\Gamma \backslash \mathfrak{h}^{3}} |u_{j}(v)|^{2} E(v, 2s) dv,
\end{equation*}
where $E(z, s)$ is the Eisenstein series. It follows from properties of $E(z, s)$ that $L(s, u_{j} \otimes u_{j})$ inherits meromorphic continuation and the functional equation under the change $s \leftrightarrow 1-s$. In the half-plane $\Re(s) \geq 1/2$, the Rankin--Selberg $L$-function $L(s, u_{j} \otimes u_{j})$ has only a simple pole at $s = 1$ with residue (see~\cite[Lemma 2.15]{Sarnak1983})
\begin{equation*}
\Res_{s = 1} L(s, u_{j} \otimes u_{j}) = \frac{8\pi \sinh(\pi t_{j})}{t_{j}} \Res_{s = 2} E(v, s)
 = \frac{8\pi \sinh(\pi t_{j}) \vol(F_{L})}{t_{j} \vol(\mathcal{M})},
\end{equation*}
where we use Euler's reflection formula
\begin{equation}\label{functional-equation-of-gamma}
\Gamma(1+it) \Gamma(1-it) = \frac{\pi t}{\sinh(\pi t)}.
\end{equation}
Then the harmonic weights are defined by
\begin{equation}\label{harmonic-weights}
\alpha_{j} \coloneqq \frac{|\rho_{j}(1)|^{2} t_{j}}{\sinh(\pi t_{j})}
 = \frac{\zeta_{K}(2)}{L(1, \mathrm{sym}^{2} u_{j})} \frac{32 \vol(F_{L})}{\pi \vol(\mathcal{M})}
 = \frac{16 \pi}{L(1, \mathrm{sym}^{2} u_{j})}.
\end{equation}
We emphasise that $\vol(\mathcal{M}) = 2\pi^{-2} \zeta_{K}(2)$ with $\zeta_{K}(s)$ the Dedekind zeta function of $K = \Q(i)$. It is well known that $\alpha_{j}$ varies mildly with $u_{j}$ (see~\cite[Proposition~3.1]{Koyama2001}) so that
\begin{equation}\label{mild}
t_{j}^{-\epsilon} \ll_{\epsilon} \alpha_{j} \ll_{\epsilon} t_{j}^{\epsilon}.
\end{equation}
It follows from Rankin--Selberg theory that
\begin{align*}
|\rho_{j}(1)|^{2} &\ll_{\epsilon} |\rho_{j}(1)|^{2} \Res_{s = 1} L(s, u_{j} \otimes \widetilde{u_{j}}) (1+|t_{j}|)^{\epsilon}\\
&\ll_{\epsilon} (1+|t_{j}|)^{\epsilon} \Res_{s = 1} \sum_{n \in \mathcal{O}^{\ast}} \frac{|\rho_{j}(n)|^{2}}{N(n)^{s}} 
\ll |\Gamma(1+it_{j})|^{-2} \ll \frac{\sinh(\pi t_{j})}{t_{j}}
\end{align*}
in view of~\eqref{mild} and the three-dimensional analogue of the technique due to Iwaniec~\cite{Iwaniec1990}.

\subsection{Auxiliary Lemmata}\label{auxiliary-lemmata}
We recall the standard classification of elements $\gamma \in \PSL_{2}(\C)$. Following~\cite[Definition 1.3, p.34]{ElstrodtGrunewaldMennicke1998}, an element $\gamma \ne \pm I$ is categorised as \textit{elliptic}, \textit{hyperbolic}, or \textit{parabolic} element for $|\tr(\gamma)| < 2$, $|\tr(\gamma)| > 2$, or $|\tr(\gamma)| = 2$ (with $\tr(\gamma) \in \R$), respectively. We call $\gamma$ \textit{loxodromic} otherwise, namely if $\tr(\gamma) \notin \R$. Every hyperbolic or loxodromic element $T$ is conjugate in $\PSL_{2}(\C)$ to a unique element
\begin{equation*}
D(T) = 
	\begin{pmatrix}
	a(T) & 0\\
	0 & a(T)^{-1}
	\end{pmatrix},
\qquad |a(T)| > 1,
\end{equation*}
which acts on $\mathfrak{h}^{3}$ as $D(T) v = K(T) z+N(T) rj$. Here $K(T) = a(T)^{2}$ is the \textit{multiplier} of~$T$~and $N(T) = |a(T)|^{2}$ is the \textit{norm} of $T$. They depend only on the class of elements conjugate to $T$, which we denote by $\{T \}$. Since $N(T)$ is invariant under conjugation, we define the norm of a conjugacy class to be the norm of any of its representatives. Let $\Gamma \subset \PSL_{2}(\C)$ and~let $C(T)$ be a centraliser of a hyperbolic or loxodromic element $T \in \Gamma$. Then we say $T_{0}$ is \textit{primitive} if it has minimal norm amongst all elements of $C(T)$.

Furthermore, the Weyl law evinces the asymptotic behaviour for the number of the discrete and continuous spectra in an expanding window. It asserts that (see~\cite[\S6,~Theorem~5.4]{ElstrodtGrunewaldMennicke1998})
\begin{equation}\label{Weyl's-law}
N_{\Gamma}(T)+M_{\Gamma}(T) \sim \frac{\vol(\mathcal{M})}{6\pi^{2}} T^{3}
\end{equation}
as $T \to \infty$, where $N_{\Gamma}(T) \coloneqq \# \{j: t_{j} \leq T \}$, while $M_{\Gamma}(T)$ is the winding number that~accounts for the contribution of the continuous spectrum as follows:
\begin{equation*}
M_{\Gamma}(T) \coloneqq \frac{1}{4\pi} \int_{-T}^{T} -\frac{\varphi^{\prime}}{\varphi}(1+it) dt.
\end{equation*}
Here we have denoted by $\varphi$ the determinant of the scattering matrix of $\Gamma$. If $\Gamma$ is cocompact, $M_{\Gamma}(T)$ vanishes and the formula~\eqref{Weyl's-law} reduces to the asymptotic behaviour of $N_{\Gamma}(T)$. Such a formula for $\Gamma$ cocompact may be deduced from some geometrical argument (without recourse the Selberg trace formula). For our purpose, there is a need to control the size of the spectrum in windows of unit length. To this end, we appeal to the following lemma.
\begin{lemma}[Bonthonneau~{\cite{Bonthonneau2017}}]\label{Bonthonneau}
Let $\Gamma$ be a cofinite Kleinian group. Then we have that
\begin{equation*}
N_{\Gamma}(T)+M_{\Gamma}(T) = \frac{\vol(\mathcal{M})}{6\pi^{2}} T^{3}+O \left(\frac{T^{2}}{\log T} \right).
\end{equation*}
\end{lemma}

The Maa{\ss}--Selberg relation~\cite[Theorem 3.6]{ElstrodtGrunewaldMennicke1998} leads to the bound
\begin{equation}\label{unit-intervals}
\# \{j: T \leq t_{j} \leq T+1 \}+\int_{T \leq |t| \leq T+1} \left|\frac{\varphi^{\prime}}{\varphi}(1+it) \right| dt \ll T^{2},
\end{equation}
which is one of the main inputs in the proof of Theorem~\ref{second-moment}.

The following explicit formula also plays an important r\^{o}le in the proof of Theorem~\ref{second-moment}.
\begin{lemma}[Nakasuji~\cite{Nakasuji2000,Nakasuji2001}]\label{Nakasuji}
Let $\Gamma = \PSL_{2}(\mathcal{O}_{K})$ where $\mathcal{O}_{K}$ is the ring of integers of an imaginary quadratic field $K$ of class number 1. For $X \gg 1$ and $1 \leq T \leq \sqrt{X}$, we have
\begin{equation}\label{explicit-formula}
\Psi_{\Gamma}(X) = \frac{X^{2}}{2}+\sum_{|t_{j}| \leq T} \frac{X^{s_{j}}}{s_{j}}+O \left(\frac{X^{2}}{T} \log X \right),
\end{equation}
where $s_{j} = 1+it_{j}$ runs over Laplacian eigenvalues for $L^{2}(\Gamma \backslash \mathfrak{h}^{3})$ counted with multiplicities.
\end{lemma}

Lemma~\ref{Nakasuji} follows from properties of the Selberg zeta function. In the proof of Nakasuji, there are superfluous terms $O(XT \log T+T^{2})$ that are absorbed into the error term in~\eqref{explicit-formula} if $T \leq \sqrt{X}$. Since this error term is at least $O(X^{3/2} \log X)$, we are unable to go beyond~the $3/2$-barrier. If we choose not to consider any cancellation in the spectral sum, then we have $O(T^{2} X+T^{-1} X^{2} \log X)$. A standard optimisation argument recovers the work of Sarnak~\cite{Sarnak1983}. As shall be seen in due course, this formidable situation necessitates a smooth explicit formula to relax the restriction on $T$. In the two-dimensional setting, Iwaniec~\cite[p.139]{Iwaniec1984} put forth~a heuristic for the bound $E_{\Gamma}(X) \ll_{\epsilon} X^{1/2+\epsilon}$ building on the twisted Linnik--Selberg conjecture (see~\cite{BalkanovaFrolenkovRisager2022}). An analogous conjecture over $\Q(i)$ might lead to a stronger estimate than what is currently known, but we leave such pursuits for our future occasion.
\begin{remark}
It is well known that $\lambda_{1}(\Gamma) \geq 975/1024$ for any congruence subgroup $\Gamma$ of $\PSL_{2}(\mathcal{O}_{K})$ (see~\cite{Koyama2004,Nakasuji2012}), where $\lambda_{1}(\Gamma)$ stands for the smallest eigenvalue of the hyperbolic Laplacian (the Ramanujan--Selberg conjecture asserts that $\lambda_{1}(\Gamma) \geq 1$). In this regard, we refer the reader~to work of Grunewald--Mennicke, Stramm~\cite{Stramm1994} and Szmidt~\cite{Szmidt1983}. Intensive numerical calculations in three dimensions were first executed by Smotrov--Golovchansky~\cite{SmotrovGolovchansky1991} and subsequently by many other researchers as in~\cite{AurichSteinerThen2012,GrunewaldHuntebrinker1996,Huntebrinker1991,Huntebrinker1995,Matthies1995,Steil1999}. The first few eigenvalues for the Picard manifold are given by $0$, $44.85 \cdots$, $74.19 \cdots$, $104.64 \cdots$, $124.40 \cdots$. Aurich--Steiner--Then~\cite{AurichSteinerThen2012} have computed 13950 consecutive eigenvalues that play an important r\^{o}le in~\S\ref{numerical-visualisations}.
\end{remark}

\subsection{Kuznetsov Formula}\label{Kuznetsov-formula}
Our experience in the context of $\PSL_{2}(\Z) \backslash \mathfrak{h}^{2}$ shows that it seems reasonable to deem that the Kuznetsov formula in three dimensions has ample applications to various problems over Gaussian integers. The Kuznetsov formula allows one to obtain~better results than those inferred from the Selberg trace formula, and expresses Fourier coefficients of automorphic forms in terms of sums of Kloosterman sums. Note that Kloosterman sums indirectly encode arithmetic information of the group, and sensational revelations were made by Iwaniec in the two-dimensional setting (see, for instance,~\cite{Iwaniec1980,Iwaniec1982}).

For $m, n, c \in \mathcal{O}$ with $c \ne 0$, we define Kloosterman sums over Gaussian integers by
\begin{equation*}
\mathcal{S}(m, n; c) \coloneqq \sum_{a \in (\mathcal{O}/(c))^{\times}} e(\langle m, a/c \rangle) e(\langle n, a^{\ast}/c \rangle),
\end{equation*}
where $a^{\ast}$ signifies the multiplicative inverse of $a$ modulo the ideal $(c)$, i.e. $aa^{\ast} \equiv 1 \tpmod c$. Then the Weil bound for Gaussian Kloosterman sums states that (see~\cite[Equation (3.5)]{Motohashi1997})
\begin{equation}\label{Weil-bound}
|\mathcal{S}(m, n; c)| \leq \sqrt{N(c)} |(m, n, c)| d(c),
\end{equation}
where $d(c)$ denotes the number of divisors of $c$. We also need the Dedekind zeta function~$\zeta_{K}(s)$ over $K = \Q(i)$ as well as the divisor function $\sigma_{\xi}(n) = \sum_{d \mid n} N(d)^{\xi}$. The following Kuznetsov formula for $\Gamma = \PSL_{2}(\mathcal{O})$ was first announced in the seminal work of Motohashi:
\begin{theorem}[Motohashi~\cite{Motohashi1996,Motohashi1997}]\label{Kuznetsov}
Assume that the function $h(t)$ is even and holomorphic in the horizontal strip $|\Im(t)| < 1/2+\epsilon$ for an arbitrary $\epsilon > 0$ and satisfies $h(t) \ll (1+|t|)^{-3-\epsilon}$ in that strip. For $m, n \in \mathcal{O}^{\ast}$, we then have that
\begin{equation*}
D+C = U+S,
\end{equation*}
where
\begin{align*}
D &= \sum_{j \geq 1} \frac{\rho_{j}(n) \overline{\rho_{j}(m)}}{\sinh(\pi t_{j})} t_{j} h(t_{j}),\\
C &= 2\pi \int_{-\infty}^{\infty} \frac{\sigma_{it}(n) \sigma_{it}(m)}{|mn|^{it} |\zeta_{K}(1+it)|^{2}} h(t) dt,\\
U &= \pi^{-2}(\delta_{m = n}+\delta_{m = -n}) \int_{-\infty}^{\infty} t^{2} h(t) dt,\\
S &= \sum_{c \in \mathcal{O}^{\ast}} \frac{\mathcal{S}(m, n; c)}{N(c)} \psi(2\pi \varsigma)
\end{align*}
with $\varsigma = \sqrt{\overline{mn}}/c$. Here $\delta_{m = n}$ is the Kronecker symbol and
\begin{equation}\label{omega}
\psi(z) = \int_{-\infty}^{\infty} \frac{it^{2}}{\sinh(\pi t)} h(t) \mathcal{J}_{it}(z) dt,
\end{equation}
\begin{equation}\label{formulation}
\mathcal{J}_{\nu}(z) = \left(\frac{|z|}{2} \right)^{2\nu} J_{\nu}^{\ast}(z) J_{\nu}^{\ast}(\overline{z}),
\end{equation}
where $J_{\nu}^{\ast}(z) = J_{\nu}(z) (z/2)^{-\nu}$ with $J_{\nu}(z)$ being the $J$-Bessel function of order $\nu$.
\end{theorem}

The choice of the sign of $\varsigma$ is immaterial as $J_{\nu}^{\ast}(z)$ is a function of $z^{2}$. We could put~$J_{\nu}(z) J_{\nu}(\overline{z})$ instead of the kernel $\mathcal{J}_{\nu}(z)$, but the formulation~\eqref{formulation} is to circumvent the possible ambiguity pertaining to the branching of the value of $J_{\nu}(z)$. When $\Gamma = \PSL_{2}(\Z)$, the Kloosterman~sum is expressed in terms of a finite Mellin transform of the square of the normalised Gau{\ss}~sum (which functions as a $p$-adic version of the gamma function). Therefore, the Kloosterman~sum would be thought of as a finite analogue of the Bessel function, or conversely the latter is an archimedean analogue of the former, as shall be seen from the formula (see~\cite[8.432.7]{GradshteynRyzhik2007})
\begin{equation*}
K_{0} \left(\frac{\sqrt{nm}}{c} \right) = \frac{1}{2} \int_{-\infty}^{\infty} 
\exp \left(-\frac{1}{c} \left(nx+\frac{m}{x} \right) \right) \frac{dx}{x}.
\end{equation*}
In the context of $\Gamma = \PSL_{2}(\mathcal{O})$, an analogous phenomenon occurs, explaining the reason~for the occurrence of the Bessel functions in the integral transform $\psi(z)$ in~\eqref{omega}.

In order to discuss some applications, one needs to excogitate an idea to invert the integral transform~\eqref{omega} and to generalise the Kloosterman summation formula that expresses sums of Kloosterman sums in terms of Fourier coefficients of automorphic forms (see~\cite[Lemma~5]{Iwaniec1984} and~\cite[Theorem~2.3]{Motohashi1997-2}). Such an issue coincides with the problem (1) due to Motohashi~\cite[p.277]{Motohashi2001}. Furthermore, Kuznetsov formul{\ae} associated to Hecke Gr\"{o}{\ss}encharakters would assist in developing problems over arbitrary imaginary quadratic number fields.

\section{Proof of Theorem~\ref{second-moment}}\label{proof-of-Theorem-1.1}
This section sheds a spotlight on the evaluation of the second moment of $E_{\Gamma}(X)$, namely
\begin{equation}\label{square-mean}
\frac{1}{\Delta} \int_{V}^{V+\Delta} \left|\Psi_{\Gamma}(X)-\frac{X^{2}}{2} \right|^{2} dX
 = \frac{1}{\Delta} \int_{V}^{V+\Delta} |E_{\Gamma}(X)|^{2} dX
\end{equation}
for some $V$ and $1 \ll \Delta \leq V$. Obtaining a favourable bound for~\eqref{square-mean} requires the utilisation of the Kuznetsov formula with a sufficiently well-behaved test function, the Hardy--Littlewood--P\'{o}lya inequality, and the spectral second moment bound for symmetric square $L$-functions. The pursuits in this section are motivated by Balog et al.~\cite{BalogBiroHarcosMaga2019} and Cherubini--Guerreiro~\cite{CherubiniGuerreiro2018}.

\subsection{Admissible Test Functions}\label{admissible-test-functions}
We use the Kuznetsov formula for which we borrow~the test function from the work of Deshouillers--Iwaniec~\cite{DeshouillersIwaniec1986}. In fact, for $X, T \geq 1$, we choose~a test function whose integral transform behaves like $X^{it}$ with an exponential weight. We~define
\begin{equation*}
\varphi(x) \coloneqq \frac{\sinh \beta}{\pi} x \exp(ix \cosh \beta), \qquad 2\beta = \log X+\frac{i}{T}.
\end{equation*}
The Bessel--Kuznetsov transform\footnote{There is a typographical error in the definition of $\hat{\varphi}(t)$ in the antecedent work such as~\cite[p.68, line~$-1$]{DeshouillersIwaniec1986} and~\cite[p.233, line~$-3$]{LuoSarnak1995}, where the imaginary unit $i$ was placed in the denominator instead of the numerator. The correct version has to be~\eqref{Bessel-Kuznetsov-transform}. This explains the sign change in the definition of $\varphi(x)$.}
\begin{equation}\label{Bessel-Kuznetsov-transform}
\hat{\varphi}(t) \coloneqq \frac{\pi i}{2 \sinh(\pi t)} \int_{0}^{\infty} (J_{2it}(x)-J_{-2it}(x)) \varphi(x) \frac{dx}{x}
\end{equation}
satisfies for $t > 0$ that
\begin{equation}\label{Bessel-Kuznetsov}
\hat{\varphi}(t) = \frac{\sinh((\pi+2i \beta)t)}{\sinh(\pi t)} = X^{it} e^{-\frac{t}{T}}+O(e^{-\pi t}).
\end{equation}
For bounded $t$ away from zero, the transform $\hat{\varphi}(t)$ is approximated by replacing $\sinh$ with~$\cosh$, namely
\begin{equation}\label{approximated-cosh}
\hat{\varphi}(t) = \frac{\cosh((\pi+2i\beta)t)}{\cosh(\pi t)}+O(e^{-\frac{\pi t}{2}}).
\end{equation}
This asymptotic formula is used in our estimation of the second moment of sums of Kloosterman sums over Gaussian integers. The contribution of the term $D$ in Theorem~\ref{Kuznetsov} becomes
\begin{equation}\label{D}
\sum_{j} |\nu_{j}(n)|^{2} \hat{\varphi}(t_{j}) = \sum_{j} \alpha_{j} \hat{\varphi}(t_{j}) |\lambda_{j}(n)|^{2}.
\end{equation}
It suffices to analyse the remaining terms in the Kuznetsov formula to reduce the problem to the evaluation of the second moment of~\eqref{D}. In the following, we compute the contributions of $C$, $U$, and $S$, separately. For the above test function $\varphi$, a simple consideration shows that
\begin{equation}\label{C-U}
|C| \ll T(\log T)^{2} d(n)^{2} \qquad \text{and} \qquad |U| \ll \left|\int_{0}^{T} X^{it} t^{2} dt \right| \ll T^{2}.
\end{equation}
If the sum of Kloosterman sums is denoted by
\begin{equation*}
\mathcal{S}_{n}(h) \coloneqq \sum_{c \in \mathcal{O}^{\ast}} 
\frac{\mathcal{S}(n, n; c)}{N(c)} h \left(\frac{2\pi \overline{n}}{c} \right),
\end{equation*}
we derive
\begin{equation}\label{S}
S = \sum_{c \in \mathcal{O}^{\ast}} \frac{\mathcal{S}(n, n; c)}{N(c)} 
\int_{-\infty}^{\infty} \frac{it^{2}}{\sinh(\pi t)} \hat{\varphi}(t) \mathcal{J}_{it} \left(\frac{2\pi \overline{n}}{c} \right) dt
 = \mathcal{S}_{n}(\psi), \qquad \psi(z) = \int_{-\infty}^{\infty} \mathcal{K}_{it}(z) t^{2} \hat{\varphi}(t) dt,
\end{equation}
where we define $\mathcal{K}_{\nu}(z) \coloneqq (\mathcal{J}_{-\nu}(z)-\mathcal{J}_{\nu}(z))/\sin \pi \nu$. If $\vartheta = \arg z$ so that $z = |z| e^{i\vartheta}$, then for $|\Re(\nu)| < 1/2$, we find the integral representation (see~\cite[Equation (2.10)]{Motohashi1997})
\begin{equation*}
\mathcal{K}_{\nu}(z) = \frac{8\cos \pi \nu}{\pi^{2}} \int_{0}^{\pi/2} \cos(2|z| \cos \vartheta \sin \tau) K_{2\nu}(2|z| \cos \tau) d\tau.
\end{equation*}
This is indicative of\footnote{There is an error in the integral representation of $\psi(z)$ in~\cite[p.12, line 14]{BalkanovaChatzakosCherubiniFrolenkovLaaksonen2019} where the coefficient is written as $4i/\pi^{2}$ instead of $8/\pi^{2}$. We obtain the correct factor by making a change of variables $\tau \mapsto \arcsin(\cos \tau)$ in the formula~\cite[Equation (2.10)]{Motohashi1997}.}
\begin{equation}\label{integral-representation-for-psi}
\psi(z) = \frac{8}{\pi^{2}} \int_{0}^{\pi/2} \cos(2|z| \cos \vartheta \sin \tau) I(2|z| \cos \tau) d\tau,
\end{equation}
where
\begin{equation}\label{I(x)}
I(x) \coloneqq \int_{-\infty}^{\infty} t^{2} \hat{\varphi}(t) \cosh(\pi t) K_{2it}(x) dt.
\end{equation}
When $t$ is bounded, the trivial bound for $\psi(z)$ is derived via the asymptotic (see~\cite{Iwaniec2002})
\begin{equation*}
K_{2it}(y) = \sqrt{\frac{\pi}{2y}} e^{-y} \left(1+O \left(\frac{1}{y} \right) \right)
\end{equation*}
for all $t \in \R$. This gives the bound
\begin{equation}\label{trivial-bound}
\psi(z) \ll |z|^{-\frac{1}{2}}.
\end{equation}

\subsection{Sums of Kloosterman Sums}\label{sums-of-Kloosterman-sums}
Our interest lies in the second moment of the spectral-arithmetic average of the shape
\begin{equation}\label{second-moment-of-the-spectral-arithmetic-average}
\frac{1}{\Delta} \int_{V}^{V+\Delta} \left|\sum_{j} \alpha_{j} \hat{\varphi}(t_{j}) |\lambda_{j}(n)|^{2} \right|^{2} dX.
\end{equation}
The Kuznetsov formula (Theorem~\ref{Kuznetsov}) reduces the problem to the estimation of the second moment of sums of Kloosterman sums, namely
\begin{equation}\label{second-moment-for-sums-of-Kloosterman-sums}
\frac{1}{\Delta} \int_{V}^{V+\Delta} |\mathcal{S}_{n}(\psi)|^{2} dX
 = \sum_{c_{1}, c_{2} \in \mathcal{O}^{\ast}} \frac{\mathcal{S}(n, n; c_{1}) \overline{\mathcal{S}(n, n; c_{2})}}{N(c_{1} c_{2})} 
\frac{1}{\Delta} \int_{V}^{V+\Delta} \psi \left(\frac{2\pi \overline{n}}{c_{1}} \right) 
\overline{\psi \left(\frac{2\pi \overline{n}}{c_{2}} \right)} dX.
\end{equation}
We now aim to replace $\mathcal{S}_{n}(\psi)$ with a finite sum involving the $K$-Bessel function of order~zero
\begin{equation}\label{K-Bessel-order-zero}
K_{0}(z) = \int_{0}^{\infty} \cos(z \sinh \varphi) d\varphi.
\end{equation}
Throughout this subsection, we shall assume that $n \in \mathcal{O}$ satisfies $N(n) \sim N$, which~indicates $N \leq N(n) \leq 2N$ for $N > 1$, and that $T$, $X$, $V$, and $\Delta$ are real numbers such that
\begin{equation}\label{assumption}
1 \ll \Delta \leq V \leq X \leq V+\Delta \qquad \text{and} \qquad V^{\epsilon} \ll T \leq \sqrt{V}.
\end{equation}
In what follows, we assume that $N \ll (TX)^{A}$ for some fixed $A > 0$. For $N$ and $V$ satisfying the hypothesis~\eqref{assumption}, we often use the abbreviation
\begin{equation}\label{abbreviation-1}
X \preccurlyeq Y \qquad \stackrel{\mathrm{def}}{\Longleftrightarrow} \qquad X \ll_{\epsilon} (NV)^{\epsilon} Y
\end{equation}
for two free parameters $X$ and $Y$. This notation is in force in various situations.
\begin{remark}\label{remark-1}
The author had been addressing the problem of generalising the work of Balog et al.~\cite{BalogBiroHarcosMaga2019} to the three-dimensional setting, obtaining Theorem~\ref{second-moment-of-sums-of-Kloosterman-sums} and Lemmata~\ref{truncated}--\ref{Cherubini-Guerreiro-analogue}. Soon after I finished writing up the first draft, I noticed the article~\cite{ChatzakosCherubiniLaaksonen2022}, wherein their results end up overlapping with ours. For completeness, we record our calculations pursued independently. For Lemmata~\ref{truncated-2} and~\ref{Cherubini-Guerreiro-analogue}, we offer rough proofs with an adjustment of our notation to theirs wherever possible. Nevertheless, we stress that our resulting bounds for the second moment of the spectral exponential sum are stronger than those in~\cite{ChatzakosCherubiniLaaksonen2022}.
\end{remark}

In order to facilitate the forthcoming analysis, we prove the following.
\begin{theorem}\label{second-moment-of-sums-of-Kloosterman-sums}
Keep the notation and assumptions as above. Then we have that
\begin{equation*}
\frac{1}{\Delta} \int_{V}^{V+\Delta} |\mathcal{S}_{n}(\psi)|^{2} dX \preccurlyeq \Delta^{-1} NV^{2}+T^{3},
\end{equation*}
where the implicit constant in $\preccurlyeq$ is absolute.
\end{theorem}

This is an analogue of~\cite[Lemma 4.1]{CherubiniGuerreiro2018} when $V = \Delta$ and allows us to establish Theorem~\ref{second-moment}. To prove Theorem~\ref{second-moment-of-sums-of-Kloosterman-sums}, we first need to simplify expressions involving $\psi(z)$ via a power series expansion of $J$-Bessel functions as in the work of Koyama~\cite{Koyama2001}.
\begin{lemma}\label{truncated}
Let
\begin{equation}\label{omega-tilde}
\tilde{\psi}(z) = \int_{-\infty}^{\infty} \frac{it^{2} \hat{\varphi}(t)}{\sinh(\pi t)} \left|\frac{z}{2} \right|^{2it} \Gamma(1+it)^{-2} dt.
\end{equation}
Then we have that
\begin{equation}\label{S-flat}
\mathcal{S}_{n}(\psi) = \mathcal{S}_{n}^{\flat}(\tilde{\psi})+O_{\epsilon}(N^{\frac{1}{2}+\epsilon} T^{1+\epsilon}),
\end{equation}
where
\begin{equation*}
\mathcal{S}_{n}^{\flat}(\tilde{\psi})
 = \sum_{N(c) > 4\pi^{2} N(n)} \frac{\mathcal{S}(n, n; c)}{N(c)} \tilde{\psi} \left(\frac{2\pi \overline{n}}{c} \right).
\end{equation*}
\end{lemma}

Therefore, after the truncation of the initial range $N(c) \leq 4\pi^{2} N(n)$ in $\mathcal{S}_{n}(\psi)$, the problem reduces to bounding the second moment of $\mathcal{S}_{n}^{\flat}(\tilde{\psi})$. It behoves us to mention that the condition $N(c) > 4\pi^{2} N(n)$ in the truncated sum $\mathcal{S}_{n}^{\flat}(\tilde{\psi})$ is equivalent to $|z| < 1$ with $z = 2\pi \overline{n}/c$.

\begin{proof}
We start with the treatment of the tail range $N(c) > 4\pi^{2} N(n)$ and show that the~sum over this range gives rise to $\mathcal{S}_{n}^{\flat}(\tilde{\psi})$. As in~\cite[p.790]{Koyama2001}, the Bessel function can be expressed as
\begin{equation*}
\mathcal{J}_{it}(z) = \left|\frac{z}{2} \right|^{2it} \sum_{l = 0}^{\infty} \sum_{m = 0}^{\infty} 
\frac{(-1)^{l+m}}{l! m!} \Gamma(1+l+it)^{-1} \Gamma(1+m+it)^{-1} 
\left(\frac{z}{2} \right)^{2l} \left(\frac{\overline{z}}{2} \right)^{2m},
\end{equation*}
which follows from a power series expansion of the $J$-Bessel functions. Hence, it holds that
\begin{equation}\label{omega-2}
\psi(z) = \sum_{l = 0}^{\infty} \sum_{m = 0}^{\infty} \frac{(-1)^{l+m}}{l! m!} 
\left(\frac{z}{2} \right)^{2l} \left(\frac{\overline{z}}{2} \right)^{2m} \int_{-\infty}^{\infty} \frac{it^{2} \hat{\varphi}(t)}{\sinh(\pi t)} 
\left|\frac{z}{2} \right|^{2it} \Gamma(1+l+it)^{-1} \Gamma(1+m+it)^{-1} dt.
\end{equation}
The first term associated to $l = m = 0$ in the double sum would make the biggest contribution to $\mathcal{S}_{n}(\psi)$. In fact, the replacement of $\psi(z)$ with $\tilde{\psi}(z)$ is ensured by Stirling's formula, namely
\begin{equation*}
\Gamma(1+l+it) \Gamma(1+m+it) = 2\pi |t|^{1+l+m} e^{-\pi|t|+2it(\log|t|-1)}(1+O(|t|^{-1}))
\end{equation*}
for fixed $l, m \geq 0$. Bounding the integrand in~\eqref{omega-2} with absolute value, we thus find that~the sum over $l+m \geq 1$ is estimated as
\begin{equation*}
\ll \sum_{l+m \geq 1} \frac{1}{l! m!} \left(\frac{z}{2} \right)^{2l} \left(\frac{\overline{z}}{2} \right)^{2m} 
\int_{0}^{\infty} t^{1-l-m} e^{-\pi t} \frac{e^{-\frac{t}{T}}}{\sinh(\pi t)} dt
\ll |z|^{2} \sum_{l+m \geq 1} \frac{4^{-l-m}}{l! m!} \int_{0}^{\infty} e^{-\frac{t}{T}} dt 
\ll T|z|^{2}.
\end{equation*}
This implies that
\begin{equation*}
\sum_{N(c) > 4\pi^{2} N(n)} \frac{\mathcal{S}(n, n; c)}{N(c)} \psi \left(\frac{2\pi \overline{n}}{c} \right)
 = \mathcal{S}_{n}^{\flat}(\tilde{\psi})+O_{\epsilon}(N^{\frac{1}{2}+\epsilon} T),
\end{equation*}
where we use the Weil bound~\eqref{Weil-bound} for Gaussian Kloosterman sums.

We next handle the initial range $N(c) \leq 4\pi^{2} N(n)$. It suffices to confirm that the sum over this range is bounded by $O_{\epsilon}(N^{1/2+\epsilon} T^{1+\epsilon})$. For an arbitrarily large $\ell > 0$, the goal is to~prove
\begin{equation}\label{omega-bound}
\psi(z) \ll |z|^{-\frac{1}{2}}+T^{2+\epsilon} X^{-\frac{1}{2}} |z|^{-1}+T^{-\ell} X|z|^{2}.
\end{equation}
Hereafter, we use the symbol $\ell$ in the same manner unless otherwise specified. The integral representation~\eqref{integral-representation-for-psi}, the definition~\eqref{I(x)}, and the approximation~\eqref{approximated-cosh} then conduce to
\begin{equation}\label{omega-3}
\psi(z) = \frac{8}{\pi^{2}} \int_{0}^{\pi/2} \cos(2|z| \cos \vartheta \sin \tau) 
\int_{-\infty}^{\infty} t^{2} \cosh((\pi+2i\beta)t) K_{2it}(2|z| \cos \tau) dt d\tau+O(|z|^{-\frac{1}{2}}).
\end{equation}
Following Gradshteyn--Ryzhik~\cite[6.795.1]{GradshteynRyzhik2007}, we calculate the inner integral over $t$ as
\begin{align*}
&\int_{-\infty}^{\infty} t^{2} \cosh((\pi+2i\beta)t) K_{2it}(2|z| \cos \tau) dt\\
& = -\frac{\pi}{8} \frac{\partial^{2}}{\partial k^{2}} \exp(-2|z| \cos \tau \cosh k)\\
& = -\frac{\pi}{4}(2(|z| \cos \tau \sinh k)^{2}-|z| \cos \tau \cosh k) \exp(-2|z| \cos \tau \cosh k),
\end{align*}
where $k = \beta-\pi i/2$. Note that
\begin{equation*}
\cosh k = -i \sinh \beta, \qquad \sinh k = -i \cosh \beta
\end{equation*}
so that
\begin{equation}\label{sinh-cosh-asymptotics}
|\cosh k|^{2} \sim |\sinh k|^{2} \asymp X, \qquad \Re(\cosh k) \asymp T^{-1} \sqrt{X},
\end{equation}
where $2\beta = \log X+i/T$. Hence, the first term on the right-hand side of~\eqref{omega-3} boils down~to
\begin{equation}
 -\frac{2}{\pi} \int_{0}^{\pi/2} (2|z|^{2} \cos^{2} \tau \sinh^{2} k-|z| \cos \tau \cosh k) 
 \cos(2|z| \cos \vartheta \sin \tau) \exp(-2|z| \cos \tau \cosh k) d\tau.
\end{equation}
We decompose the $\tau$-integral into two parts $\cos \tau > \Re(\cosh k)^{-1} |z|^{-1}(\log T)^{A}$ and otherwise. The integral in the first range can be evaluated by bounding the integrand in absolute value:
\begin{equation}\label{bounding-in-absolute-value}
|z|^{2} \cos^{2} \tau \sinh^{2} k \exp(-2|z| \cos \tau \cosh k) \ll T^{-\ell} X|z|^{2},
\end{equation}
where we use~\eqref{sinh-cosh-asymptotics} and the fact that the exponential is $O(T^{-\ell})$. The bound~\eqref{bounding-in-absolute-value} coincides with the third term on the right-hand side of~\eqref{omega-bound}. When $\cos \tau \leq \Re(\cosh k)^{-1} |z|^{-1}(\log T)^{A}$, integration by parts\footnote{The exponential is piecewise monotonic so that the derivative is piecewise of~constant sign.} once in $\tau$ leads to the second term $O_{\epsilon}(T^{2+\epsilon} X^{-1/2} |z|^{-1})$. These bounds in tandem with~\eqref{trivial-bound} leads one to the desired claim~\eqref{omega-bound}. We eventually sum up the terms $\mathcal{S}(n, n; c)/N(c)$ over $N(z) \geq 1$, choose $\ell$ large, and use the Weil bound~\eqref{Weil-bound} again,~obtaining
\begin{equation*}
\sum_{N(c) \leq 4\pi^{2} N(n)} \frac{\mathcal{S}(n, n; c)}{N(c)} \psi \left(\frac{2\pi \overline{n}}{c} \right) 
\preccurlyeq \sqrt{N}+\sqrt{N} T^{2} X^{-\frac{1}{2}}+NXT^{-\ell} \ll \sqrt{N} T.
\end{equation*}
The final bound follows from~\eqref{assumption}. This concludes the proof of Theorem~\ref{truncated}.
\end{proof}

We undertake to represent $\tilde{\psi}(z)$ in terms of the $K$-Bessel function of order zero. This~allows one to replace $\mathcal{S}_{n}^{\flat}(\tilde{\psi})$ with a finite sum of such $K$-Bessel functions up to an admissible error term. The following lemma clarifies an approximation of the weight function $\tilde{\psi}(z)$ for small~$z$.
\begin{lemma}\label{truncated-2}
Keep the notation and assumptions as above. Then we have that
\begin{equation}\label{omega-tilde-expressed-by-K-Bessel-function-of-order-zero}
\tilde{\psi}(z) = \frac{iM^{2} X|z|^{2}}{2\pi} K_{0}(M \sqrt{X}|z|)+O(X^{-\frac{1}{4}} |z|^{\frac{3}{2}}),
\end{equation}
where $M = \exp(-i(\pi-1/T)/2)$. This implies that
\begin{equation}\label{S-sharp}
\mathcal{S}_{n}(\psi) = \mathcal{S}_{n}^{\sharp}(K_{0})+O((NX)^{\frac{1}{2}+\epsilon}),
\end{equation}
where
\begin{equation*}
\mathcal{S}_{n}^{\sharp}(K_{0})
 = 2\pi iM^{2} N(n) X \sum_{C_{1} \leq N(c) \leq C_{2}} 
\frac{\mathcal{S}(n, n; c)}{N(c)^{2}} K_{0} \left(\frac{2\pi M |n| \sqrt{X}}{|c|} \right)
\end{equation*}
with $C_{1} = N(n)V(T \log T)^{-2}$ and $C_{2} = N(n) V$.
\end{lemma}

\begin{proof}
The formula~\eqref{functional-equation-of-gamma} yields
\begin{align}\label{omega-tilde-2}
\begin{split}
\tilde{\psi}(z) &= \frac{M^{2} X|z|^{2}}{8\pi^{2}} \int_{(1)} \Gamma(s)^{2} \left(\frac{M\sqrt{X} |z|}{2} \right)^{-2s} ds
 - \frac{M^{-2} X^{-1} |z|^{2}}{8\pi^{2}} \int_{(1)} \Gamma(s)^{2} \left(\frac{M^{-1} X^{-\frac{1}{2}} |z|}{2} \right)^{-2s} ds\\
& = \frac{iM^{2} X|z|^{2}}{2\pi} K_{0}(M\sqrt{X} |z|)-\frac{iM^{-2} X^{-1} |z|^{2}}{2\pi} K_{0}(M^{-1} X^{-\frac{1}{2}} |z|),
\end{split}
\end{align}
where we use~\cite[17.43.32]{GradshteynRyzhik2007}. From the inequality $|K_{0}(z)| \leq 2|z|^{-1/2} \exp(-\Re(z))$~for $\Re(z) \geq 0$ and $|z|$ bounded away from zero, it follows that
\begin{equation*}
|K_{0}(M^{-1} X^{-\frac{1}{2}} |z|)| \ll X^{\frac{3}{4}} |z|^{-\frac{1}{2}}
\end{equation*}
so that the second term in~\eqref{omega-tilde-2} is $O(X^{-1/4} |z|^{3/2})$. Summing this over $N(c) > 4\pi^{2} N(n)$~leads to $O_{\epsilon}(N^{1/2+\epsilon} X^{-1/4})$, which is absorbed into the error term in~\eqref{S-sharp}. Thus it remains to verify that the approximation~\eqref{omega-tilde-expressed-by-K-Bessel-function-of-order-zero} yields~\eqref{S-sharp}. To this end, we employ the bound
\begin{equation*}
X|z|^{2} |K_{0}(M\sqrt{X} |z|)| \ll X^{\frac{3}{4}} |z|^{\frac{3}{2}} \exp \left(-\frac{\sqrt{X} |z|}{100T} \right).
\end{equation*}
Similarly, summing this over $4\pi^{2} N(n) < N(c) < C_{1}$ and $N(c) > C_{2}$ gives rise to the error term $O_{\epsilon}((NX)^{1/2+\epsilon})$. The utilisation of Lemma~\ref{truncated} completes the proof of Lemma~\ref{truncated-2}.
\end{proof}

Lemma~\ref{truncated-2} reduces the problem to the analysis of the second moment of $\mathcal{S}_{n}^{\sharp}(K_{0})$, and one encounters an oscillatory integral of the shape
\begin{equation}\label{oscillatory-integral-K-Bessel}
\frac{1}{\Delta} \int_{V}^{V+\Delta} K_{0}(M \sqrt{X} z_{1}) \overline{K_{0}(M \sqrt{X} z_{2})} dX,
\end{equation}
where $z_{j} = 2\pi|n|/|c_{j}|$ for $j = 1, 2$. This is an analogue of the integral in~\cite[Lemma~4.6]{CherubiniGuerreiro2018}.~The trivial bound for the integral~\eqref{oscillatory-integral-K-Bessel} is $O((Vz_{1} z_{2})^{-1/2})$, which follows from the inequality
\begin{equation*}
|K_{0}(z)| \leq 2|z|^{-\frac{1}{2}} \exp(-\Re(z))
\end{equation*}
for $\Re(z) \geq 0$ and $|z|$ bounded away from zero. In the light of the range $C_{1} \leq N(c_{j}) \leq C_{2}$~for $j = 1, 2$, we impose the restriction $V^{-1/2} \leq z_{j} \ll 1$. On the other hand, the inequality
\begin{equation}\label{K-Bessel-order-zero-bound}
\left|\frac{d}{dz} e^{z} K_{0}(z) \right| \leq |z|^{-\frac{3}{2}}(1+|z|^{-1})
\end{equation}
for $\Re(z) \geq 0$ along with integration by parts leads to $O((z_{1} z_{2})^{-1/2} \Delta^{-1}|z_{1}-z_{2}|^{-1})$. Hence, we have shown the following bound for the oscillatory integral~\eqref{oscillatory-integral-K-Bessel}.
\begin{lemma}\label{Cherubini-Guerreiro-analogue}
Keep the notation and assumptions as above. Then we have for $z_{1}, z_{2} > 0$ with $V^{-1/2} \leq z_{j} \ll 1$ for $j = 1, 2$ that
\begin{equation}\label{oscillatory-integral-K-Bessel-2}
\frac{1}{\Delta} \int_{V}^{V+\Delta} K_{0}(M \sqrt{X} z_{1}) \overline{K_{0}(M \sqrt{X} z_{2})} dX 
\ll (z_{1} z_{2})^{-\frac{1}{2}} \min(V^{-\frac{1}{2}}, \Delta^{-1} |z_{1}-z_{2}|^{-1}).
\end{equation}
In particular, when $z_{j} = 2\pi|n|/|c_{j}|$, the right-hand side of~\eqref{oscillatory-integral-K-Bessel-2} is bounded by
\begin{equation}\label{special-case}
\sqrt{\frac{|c_{1} c_{2}|}{N(n)}} \min \left(V^{-\frac{1}{2}}, \frac{|c_{1} c_{2}| \Delta^{-1}}{|n|||c_{1}|-|c_{2}||} \right).
\end{equation}
\end{lemma}

Lemma~\ref{Cherubini-Guerreiro-analogue} demonstrates that the weight function $K_{0}(M \sqrt{X} |z|)$ carries some oscillation in $X$ when we integrate over the interval $V \leq X \leq V+\Delta$.

\begin{remark}\label{Chatzakos-Cherubini-Laaksonen}
The author has learned from the work of Chatzakos--Cherubini--Laaksonen~\cite{ChatzakosCherubiniLaaksonen2022} that if the bound~\eqref{K-Bessel-order-zero-bound} is also employed in the range $C_{1} \leq N(c) \leq C_{2}$, we would deduce that
\begin{equation}\label{Chatzakos-Cherubini-Laaksonen-remark}
\mathcal{S}_{n}^{\sharp}(K_{0}) \preccurlyeq \sqrt{NTX}.
\end{equation}
Therefore, it follows that
\begin{equation}\label{Chatzakos-Cherubini-Laaksonen-remark-2}
\mathcal{S}_{n}(\psi) \preccurlyeq \sqrt{NTX}.
\end{equation}
This recovers a pointwise bound due to Koyama~\cite[p.792]{Koyama2001}. Via suitable optimisation~of~the parameters $N$ and $T$, we obtain the bound $E_{\Gamma}(X) \ll_{\epsilon} X^{11/7+\epsilon}$ conditionally under the mean-Lindel\"{o}f hypothesis. Unconditionally, a bound of the form $E_{\Gamma}(X) \ll_{\epsilon} X^{(11+4\eta)/(7+2\eta)+\epsilon}$ would follow immediately. Koyama has informed the author that he was able to deal with the case where $\eta \ne 0$, despite the \textit{ad hoc} assumption $\eta = 0$ in his work~\cite{Koyama2001}. The author would like~to appreciate his illuminating comments on the estimates~\eqref{Chatzakos-Cherubini-Laaksonen-remark} and~\eqref{Chatzakos-Cherubini-Laaksonen-remark-2}.
\end{remark}

\begin{proof}[Proof of Theorem~\ref{second-moment-of-sums-of-Kloosterman-sums}]
Our method dates back to the work of Cherubini--Guerreiro~\cite[\S4.1]{CherubiniGuerreiro2018}. Upon invoking the assumption~\eqref{assumption} and using Lemma~\ref{truncated-2}, the second moment of $\mathcal{S}_{n}(\psi)$ is approximated by the second moment of $\mathcal{S}_{n}^{\sharp}(K_{0})$ with an admissible error term, namely
\begin{equation}\label{truncated-3}
\frac{1}{\Delta} \int_{V}^{V+\Delta} |\mathcal{S}_{n}(\psi)|^{2} dX 
= \frac{1}{\Delta} \int_{V}^{V+\Delta} |\mathcal{S}_{n}^{\sharp}(K_{0})|^{2} dX+O((NV)^{1+\epsilon}).
\end{equation}
From the definition of $\mathcal{S}_{n}^{\sharp}(K_{0})$, it follows that
\begin{equation}\label{dyadic-decomposition}
\frac{1}{\Delta} \int_{V}^{V+\Delta} |\mathcal{S}_{n}^{\sharp}(K_{0})|^{2} dX
\preccurlyeq (NV)^{2} \sup_{C_{1} \leq R \leq C_{2}} \frac{1}{\Delta} \int_{V}^{V+\Delta} \left|\sum_{N(c) \sim R} 
\frac{\mathcal{S}(n, n; c)}{N(c)^{2}} K_{0} \left(\frac{2\pi M |n| \sqrt{X}}{|c|} \right) \right|^{2} dX.
\end{equation}
Following Cherubini--Guerreiro~\cite{CherubiniGuerreiro2018}, we square out the sum over $c$ to bound~\eqref{dyadic-decomposition}~by
\begin{equation}\label{diagonal-off-diagonal}
(NV)^{2} \sup_{C_{1} \leq R \leq C_{2}} \sum_{\substack{N(c_{1}) \sim R \\ N(c_{2}) \sim R}} 
\frac{\mathcal{S}(n, n; c_{1}) \overline{\mathcal{S}(n, n; c_{2})}}{N(c_{1} c_{2})^{2}} 
\frac{1}{\Delta} \int_{V}^{V+\Delta} K_{0} \left(\frac{2\pi M |n| \sqrt{X}}{|c_{1}|} \right) 
\overline{K_{0} \left(\frac{2\pi M |n| \sqrt{X}}{|c_{2}|} \right)} dX.
\end{equation}
We decompose the double sum in~\eqref{diagonal-off-diagonal} into two parts $\Sigma_{d}$ and $\Sigma_{od}$, where $\Sigma_{d}$ is the diagonal contribution and $\Sigma_{od}$ is the off-diagonal contribution. The Weil bound~\eqref{Weil-bound} plays a crucial r\^{o}le in the estimations of both the contributions $\Sigma_{d}$ and $\Sigma_{od}$. We also appeal to the bound
\begin{equation*}
\sum_{N(c) \leq x} |(n, c)|^{2} d^{2}(c) \ll x(\log x)^{3} d(n).
\end{equation*}
The evaluation of $\Sigma_{d}$ relies on the first bound in the minimum in~\eqref{special-case}, getting
\begin{equation}\label{diagonal}
\Sigma_{d} \preccurlyeq (NV)^{\frac{3}{2}} \sum_{C_{1} \leq N(c) \leq C_{2}} 
\frac{|(n, c)|^{2} d^{2}(c)}{N(c)^{\frac{5}{2}}} \preccurlyeq T^{3},
\end{equation}
where the assumption $V^{-1/2} \leq z_{j} \ll 1$ in Lemma~\ref{Cherubini-Guerreiro-analogue} is satisfied. Regarding the off-diagonal contribution $\Sigma_{od}$, the interpolation of the two bounds in Lemma~\ref{Cherubini-Guerreiro-analogue} with exponents $(1-\lambda, \lambda)$ for some $0 < \lambda < 1$ leads to
\begin{equation}\label{interpolated}
\min(V^{-\frac{1}{2}}, \Delta^{-1} |z_{1}-z_{2}|^{-1}) 
\ll_{\lambda} \frac{\Delta^{-\lambda} V^{\frac{\lambda-1}{2}}}{|z_{1}-z_{2}|^{\lambda}}.
\end{equation}
Therefore, the substitution of $z_{j} = 2\pi|n|/|c_{j}|$ into~\eqref{interpolated} gives
\begin{equation}\label{substitute}
\begin{split}
\Sigma_{od} &\preccurlyeq (NV)^{2} \sup_{C_{1} \leq R \leq C_{2}} 
\sum_{\substack{N(c_{1}) \ne N(c_{2}) \\ R < N(c_{1}), N(c_{2}) \leq 2R}}
\frac{|\mathcal{S}(n, n; c_{1}) \mathcal{S}(n, n; c_{2})|}{N(c_{1} c_{2})^{2}} 
\left(\frac{|c_{1} c_{2}|}{NV} \right)^{\frac{1-\lambda}{2}} 
\left(\frac{|c_{1} c_{2}|^{\frac{3}{2}} (\Delta N)^{-1}}{||c_{1}|-|c_{2}||} \right)^{\lambda}\\
& \preccurlyeq_{\lambda} \Delta^{-\lambda} N^{\frac{3-\lambda}{2}} V^{\frac{3+\lambda}{2}} 
\sum_{\sqrt{R} < k_{1} \ne k_{2} \leq \sqrt{2R}} \frac{a_{k_{1}} a_{k_{2}}}{|k_{1}-k_{2}|^{\lambda}},
\end{split}
\end{equation}
where $k_{j} = |c_{j}|$ and
\begin{equation*}
a_{k} \coloneqq \sum_{|c| = k} \frac{|\mathcal{S}(n, n; c)|}{|c|^{\frac{7}{2}-\lambda}}.
\end{equation*}
We apply the Hardy--Littlewood--P\'{o}lya inequality~\cite[Theorem 381, p.288]{HardyLittlewoodPolya1934} in a special case, which asserts that if $0 < \lambda < 1$, $\lambda = 2(1-p^{-1})$ and $\{a_{r} \}$ is a sequence of nonnegative~numbers, then we have that\footnote{It is feasible to strengthen this bound. One can use a version of~\eqref{Hardy-Littlewood-Polya-inequality} with an explicit implicit constant proven by Carneiro--Vaaler~\cite[Corollary~7.2, (7.20)]{CarneiroVaaler2010}, or the extremal case of~\eqref{Hardy-Littlewood-Polya-inequality} with $(\lambda, p) = (1, 2)$, and a logarithmic correction proven by Li--Villavert~\cite{LiVillavert2011}.}
\begin{equation}\label{Hardy-Littlewood-Polya-inequality}
\sum_{k_{1} \ne k_{2}} \frac{a_{k_{1}} a_{k_{2}}}{|k_{1}-k_{2}|^{\lambda}} 
\ll_{\lambda} \left(\sum_{k} a_{k}^{p} \right)^{\frac{2}{p}}.
\end{equation}
Using the inequality~\eqref{Hardy-Littlewood-Polya-inequality} coupled with the Weil bound~\eqref{Weil-bound},~we~derive
\begin{align}\label{off-diagonal}
\begin{split}
\Sigma_{od} &\preccurlyeq \Delta^{-\lambda} N^{\frac{3-\lambda}{2}} V^{\frac{3+\lambda}{2}} 
\left(\sum_{k \sim \sqrt{R}} a_{k}^{\frac{2}{2-\lambda}} \right)^{2-\lambda}\\
&\ll \Delta^{-\lambda} N^{\frac{3-\lambda}{2}} V^{\frac{3+\lambda}{2}} 
\left(\sum_{c \ne 0} \frac{|(n, c)|^{2} d^{2}(c)}{N(c)^{1+\frac{1}{2(2-\lambda)}}} \right)^{2-\lambda}
\preccurlyeq \Delta^{-\lambda} N^{\frac{3-\lambda}{2}} V^{\frac{3+\lambda}{2}}
\end{split}
\end{align}
with the implicit constant dependent only on $\epsilon$ and $\lambda$. Since the error term in~\eqref{truncated-3} is~smaller than the bound in~\eqref{off-diagonal}, we eventually combine~\eqref{diagonal} and~\eqref{off-diagonal}, obtaining
\begin{equation*}
\frac{1}{\Delta} \int_{V}^{V+\Delta} |\mathcal{S}_{n}(\psi)|^{2} dX 
\preccurlyeq_{\lambda} \Delta^{-\lambda} N^{\frac{3-\lambda}{2}} V^{\frac{3+\lambda}{2}}+T^{3}.
\end{equation*}
Hence, balancing the right-hand side with $\lambda = 1-\epsilon$ concludes the proof of Theorem~\ref{second-moment-of-sums-of-Kloosterman-sums}.
\end{proof}

Immediate corollaries of Theorem~\ref{second-moment-for-sums-of-Kloosterman-sums} include the following.
\begin{corollary}\label{second-moment-of-spectral-side}
Keep the notation and assumptions as above. Then we have that
\begin{equation}\label{second-moment-of-spectral-side-2}
\frac{1}{\Delta} \int_{V}^{V+\Delta} \left|\sum_{j} \alpha_{j} \hat{\varphi}(t_{j}) |\lambda_{j}(n)|^{2} \right|^{2} dX
\preccurlyeq \Delta^{-1} NV^{2}+T^{4}.
\end{equation}
\end{corollary}

\begin{proof}
The Kuznetsov formula (Theorem~\ref{Kuznetsov}) together with the bound~\eqref{C-U} leads to
\begin{equation}\label{second-moment-of-spectral-side-3}
\frac{1}{\Delta} \int_{V}^{V+\Delta} \left|\sum_{j} \alpha_{j} \hat{\varphi}(t_{j}) |\lambda_{j}(n)|^{2} \right|^{2} dX 
\ll \frac{1}{\Delta} \int_{V}^{V+\Delta} |\mathcal{S}_{n}(\psi)|^{2} dX+T^{4}.
\end{equation}
The claim follows from the incorporation of~\eqref{second-moment-of-spectral-side-3} with Theorem~\ref{second-moment-of-sums-of-Kloosterman-sums}.
\end{proof}


\subsection{Reduction to Moments}\label{reduction-to-moments}
Let $h: (0, \infty) \to \R$ be a smooth compactly supported~test function with holomorphic Mellin transform $\tilde{h}: \C \to \C$. In our case, we choose $h$ such~that it is supported in some dyadic box $[\sqrt{N}, \sqrt{2N}]$ whose derivatives satisfy
\begin{equation}\label{support}
|h^{(\ell)}(\xi)| \ll N^{-\frac{\ell}{2}}, \qquad \ell = 0, 1, 2, \dots
\end{equation}
and whose mean value is
\begin{equation*}
\int_{-\infty}^{\infty} h(\xi) \xi d\xi = \tilde{h}(2) = N.
\end{equation*}
Integrating by parts $\ell$-times and using~\eqref{support}, one sees that
\begin{equation*}
\tilde{h}(s) = \frac{(-1)^{\ell}}{s(s+1) \cdots (s+\ell-1)} 
\int_{0}^{\infty} h^{(\ell)}(x) x^{s+\ell} \frac{dx}{x} \ll_{\sigma, \ell} N^{\frac{\sigma}{2}}(|s|+1)^{-\ell}
\end{equation*}
with the implicit constant depending continuously on $\sigma$ for $s = \sigma+it$. We remark that~this is meant for $s$ outside the set $\{0, -1, -2, \dots \}$, but it holds even at these exceptional points.

Following~\cite{BalogBiroHarcosMaga2019,LuoSarnak1995} and recalling~\eqref{Rankin--Selberg}, we then contemplate the spectral-arithmetic average
\begin{align*}
\sum_{j} \alpha_{j} \hat{\varphi}(t_{j}) \sum_{n \in \mathcal{O}^{\ast}} h(|n|) |\lambda_{j}(n)|^{2}
& = \sum_{j} \sum_{n \in \mathcal{O}^{\ast}} \hat{\varphi}(t_{j}) h(|n|) |\nu_{j}(n)|^{2} \\
& = \sum_{j} \alpha_{j} \hat{\varphi}(t_{j}) \int_{(3)} 
\tilde{h}(s) \frac{\zeta_{K}(s/2)}{\zeta_{K}(s)} L \left(\frac{s}{2}, \mathrm{sym}^{2} u_{j} \right) \frac{ds}{2\pi i}.
\end{align*}
Making a change of variables and shifting the contour with crossing a simple pole at $s = 1$, we have for some absolute constant $c$ that
\begin{equation*}
\sum_{j} \alpha_{j} \hat{\varphi}(t_{j}) \sum_{n \in \mathcal{O}^{\ast}} h(|n|) |\lambda_{j}(n)|^{2}
 = cN \sum_{j} \hat{\varphi}(t_{j})+\sum_{j} \alpha_{j} \hat{\varphi}(t_{j}) \int_{(1/2)} 
\tilde{h}(2s) \frac{\zeta_{K}(s)}{\zeta_{K}(2s)} L(s, \mathrm{sym}^{2} u_{j}) \frac{ds}{\pi i},
\end{equation*}
where we use~\eqref{harmonic-weights}. Therefore, the smoothed spectral exponential sum $\sum_{j} X^{it_{j}} e^{-t_{j}/T}$ equals
\begin{multline}\label{spectral-arithmetic-average}
\frac{1}{cN} \sum_{n \in \mathcal{O}^{\ast}} h(|n|) \sum_{j} \alpha_{j} \hat{\varphi}(t_{j}) |\lambda_{j}(n)|^{2}
 - \frac{1}{cN} \int_{(1/2)} \tilde{h}(2s) \frac{\zeta_{K}(s)}{\zeta_{K}(2s)} 
\sum_{j} \alpha_{j} \hat{\varphi}(t_{j}) L(s, \mathrm{sym}^{2} u_{j}) \frac{ds}{\pi i}+O(1)\\
 = \frac{1}{cN} \sum_{n \in \mathcal{O}^{\ast}} h(|n|) \mathcal{S}_{n}(\psi)
 - \frac{1}{cN} \int_{(1/2)} \tilde{h}(2s) \sum_{j} \frac{t_{j}}{\sinh(\pi t_{j})} 
\hat{\varphi}(t_{j}) L(s, u_{j} \otimes u_{j}) \frac{ds}{\pi i}+O(T^{2}),
\end{multline}
where the sum over $j$ on the right-hand side represents the first moment of Rankin--Selberg $L$-functions. We stress that the spectral weights $\hat{\varphi}(t_{j})$ depend on the parameters $X, T \geq 1$.

The mean-Linde\"{o}f hypothesis was first formulated in the celebrated work of Koyama~\cite{Koyama2001}.
\begin{conjecture}[Mean-Lindel\"{o}f hypothesis]\label{mean-Lindelof}
Assume that $w = 1/2+i\tau$ with $\tau \in \R$. The symmetric square $L$-function attached to a Hecke--Maa{\ss} cusp form for $\Gamma = \PSL_{2}(\mathcal{O})$ satisfies
\begin{equation*}
\sum_{t_{j} \sim T} |L(w, \mathrm{sym}^{2} u_{j})|^{2} \ll_{\epsilon} |w|^{A} T^{3+\epsilon}
\end{equation*}
for some $A > 0$.
\end{conjecture}
\begin{remark}
The polynomial dependence in $w$ suffices as in Conjecture~\ref{mean-Lindelof}. In the $\tau$-aspect,~one may specify the exponent $A$ to obtain hybrid subconvexity (say, $|w|^{3+\epsilon}$ would be reasonable).
\end{remark}
Let $0 \leq \eta \leq 1$ denote the additional exponent of $T$ for the mean value of Rankin--Selberg $L$-functions on the critical line. More precisely, we define the quantity $\eta$ by
\begin{equation}\label{mean-value-estimate}
\sum_{t_{j} \sim T} \frac{t_{j}}{\sinh(\pi t_{j})} |L(w, u_{j} \otimes u_{j})| \ll_{\epsilon} |w|^{A} T^{3+\eta+\epsilon}
\end{equation}
so that
\begin{equation}\label{mean-value-estimate-2}
\sum_{t_{j} \sim T} |L(w, \mathrm{sym}^{2} u_{j})|^{2} \ll_{\epsilon} |w|^{A} T^{3+2\eta+\epsilon}.
\end{equation}
The second bound follows from the Cauchy--Schwarz inequality, the convex bound $\zeta_{\Q(i)}(w) \ll |w|^{1/2+\epsilon}$ and the Hoffstein--Lockhart bound $|\nu_{j}(1)| \ll t_{j}^{\epsilon}$. The convex bound for the left-hand side of~\eqref{mean-value-estimate} is $O_{\epsilon}(T^{4+\epsilon})$, while the mean-Lindel\"{o}f hypothesis states that $\eta = 0$. The~current record is due to Balkanova--Frolenkov~\cite[Theorem 1.2]{BalkanovaFrolenkov2022} that permits $\eta = 2 \theta$. Using Nelson's subconvex bounds~\cite[Theorem 1.1]{Nelson2020}, one can take $\eta = 1/3$. Koyama~\cite{Koyama2001} has shown that~the mean-Lindel\"{o}f hypothesis leads to $\delta = 11/7$. Note that if the Generalised Lindel\"{o}f Hypothesis
\begin{equation}\label{Lindelof}
|L(w, u_{j} \otimes u_{j})| \ll_{\epsilon} |w t_{j}|^{\epsilon} \frac{\sinh(\pi t_{j})}{t_{j}},
\end{equation}
was assumed, then we would obtain $\delta = 3/2$ via the Weil bound~\eqref{Weil-bound}. This statement is an analogue of the heuristic in~\cite[p.139]{Iwaniec1984} for the conjectural exponent $2/3$ in two dimensions. Refining the work of Koyama~\cite{Koyama2001}, we can make his result unconditional, namely
\begin{equation*}
E_{\Gamma}(X) \ll_{\epsilon} X^{\frac{11+4\eta}{7+2\eta}+\epsilon}
\end{equation*}
and
\begin{equation}\label{freedom}
S(T, X) \ll_{\epsilon} T^{\frac{7+2\eta}{4}+\epsilon} X^{\frac{1}{4}+\epsilon}+T^{2}.
\end{equation}

We are in a position to estimate the smoothed spectral exponential sum in square mean.
\begin{theorem}\label{smoothing}
Keep the notation and assumptions as above. For any $\epsilon > 0$, we have that
\begin{equation}\label{square-mean-integral-1}
\frac{1}{\Delta} \int_{V}^{V+\Delta} \left|\sum_{j} X^{it_{j}} e^{-\frac{t_{j}}{T}} \right|^{2} dX \ll_{\epsilon} V^{\epsilon} \times 
	\begin{cases}
	T^{5} \min(\Delta^{-1} V, T) & \text{if $0 < T \leq V^{\frac{1}{5+2\eta}}$},\\
	\Delta^{-1} V^{\frac{3}{2}} T^{\frac{5}{2}+\eta} & \text{if $V^{\frac{1}{5+2\eta}} < T \leq (\Delta^{-2} V^{3})^{\frac{1}{3-2\eta}}$},\\
	T^{4} & \text{if $(\Delta^{-2} V^{3})^{\frac{1}{3-2\eta}} < T \leq \sqrt{V}$}.
	\end{cases}
\end{equation}
In particular, the left-hand side can uniformly be bounded by $O_{\epsilon}(\Delta^{-1} T^{5/2+\eta} V^{3/2+\epsilon})$.
\end{theorem}

\begin{proof}
We start with the identity~\eqref{spectral-arithmetic-average} where we replace the second occurrence of $\hat{\varphi}(t_{j})$~with $X^{it_{j}} e^{-t_{j}/T}$ up to an admissible error term. In addition, we restrict the integration to~$|\Im(s)| \leq T^{\epsilon}$. Using the support condition on the test function $h$, we obtain the bound
\begin{multline}\label{smoothened}
\left|\sum_{j} X^{it_{j}} e^{-\frac{t_{j}}{T}} \right|^{2} 
\ll_{\epsilon} 1+\frac{1}{N} \sum_{N \leq N(n) \leq 2N} 
\left|\sum_{j} \alpha_{j} \hat{\varphi}(t_{j}) |\lambda_{j}(n)|^{2} \right|^{2}\\
 + \frac{T^{\epsilon}}{N} \int_{-T^{\epsilon}}^{T^{\epsilon}} 
\left|\sum_{j} \alpha_{j} X^{it_{j}} e^{-\frac{t_{j}}{T}} L \left(\frac{1}{2}+i\tau, \mathrm{sym}^{2} u_{j} \right) \right|^{2} d\tau.
\end{multline}
Abbreviating
\begin{equation*}
L_{j}(\tau) \coloneqq L \left(\frac{1}{2}+i\tau, \mathrm{sym}^{2} u_{j} \right),
\end{equation*}
we average the left-hand side of~\eqref{smoothened} over $V \leq X \leq V+\Delta$. Applying Corollary~\ref{second-moment-of-spectral-side} to the sum over $n$ in~\eqref{smoothened}, one sees that
\begin{multline}\label{average-1}
\frac{1}{\Delta} \int_{V}^{V+\Delta} \left|\sum_{j} X^{it_{j}} e^{-\frac{t_{j}}{T}} \right|^{2} dX
\ll_{\epsilon} (\Delta^{-1} NV^{2}+T^{4})(NTV)^{\epsilon}\\
 + \frac{T^{\epsilon}}{\Delta N} \int_{|\tau| \leq T^{\epsilon}} 
 \int_{V}^{V+\Delta} \left|\sum_{j} \alpha_{j} X^{it_{j}} e^{-\frac{t_{j}}{T}} L_{j}(\tau) \right|^{2} dX d\tau.
\end{multline}
We appeal to the Cauchy--Schwarz inequality for the sum over $j$. In particular, one distributes the spectral parameters $t_{j}$ into intervals of length $T$, deducing
\begin{equation*}
\left|\sum_{j} \alpha_{j} X^{it_{j}} e^{-\frac{t_{j}}{T}} L_{j}(\tau) \right|^{2}
\ll \sum_{m = 1}^{\infty} m^{2} \left|\sum_{(m-1) T \leq t_{j} < mT} \alpha_{j} X^{it_{j}} e^{-\frac{t_{j}}{T}} L_{j}(\tau) \right|^{2}.
\end{equation*}
If we define
\begin{equation}\label{J}
J(T, V, \Delta; m, \tau) \coloneqq \frac{1}{\Delta} \int_{V}^{V+\Delta} 
\left|\sum_{(m-1)T \leq t_{j} < mT} \alpha_{j} X^{it_{j}} e^{-\frac{t_{j}}{T}} L_{j}(\tau) \right|^{2} dX,
\end{equation}
then it follows that
\begin{equation}\label{sup}
\frac{1}{\Delta} \int_{V}^{V+\Delta} \left|\sum_{j} X^{it_{j}} e^{-\frac{t_{j}}{T}} \right|^{2} dX 
\ll_{\epsilon} (\Delta^{-1} NV^{2}+T^{4})(NTV)^{\epsilon}
 + \frac{T^{\epsilon}}{N} \sup_{|\tau| \leq T} \sum_{m = 1}^{\infty} m^{2} J(T, V, \Delta; m, \tau).
\end{equation}
The estimation of $J(T, V, \Delta; m, \tau)$ is executed by opening the square, integrating explicitly~in $X$, and using the estimate~\eqref{mild} for the harmonic weights $\alpha_{j}$. This procedure leads to
\begin{align*}
J(T, V, \Delta; m, \tau) &\ll_{\epsilon} \Delta^{-1} T^{\epsilon} V e^{-2m} 
\sum_{(m-1)T \leq t_{j}, t_{k} < mT} \frac{|L_{j}(\tau) L_{k}(\tau)|}{1+|t_{j}-t_{k}|}\\ 
&\leq \frac{\Delta^{-1} T^{\epsilon} V e^{-2m}}{2} 
\sum_{(m-1)T \leq t_{j}, t_{k} < mT} \frac{|L_{j}(\tau)|^{2}+|L_{k}(\tau)|^{2}}{1+|t_{j}-t_{k}|}\\
& = \Delta^{-1} T^{\epsilon} V e^{-2m} \sum_{(m-1)T \leq t_{j} < mT} |L_{j}(\tau)|^{2} 
\sum_{(m-1)T \leq t_{k} < mT} \frac{1}{1+|t_{j}-t_{k}|}.
\end{align*}
The Weyl law~\eqref{unit-intervals} on unit intervals allows one to bound the last sum as
\begin{equation*}
\sum_{(m-1)T \leq t_{k} < mT} \frac{1}{1+|t_{j}-t_{k}|} 
\leq \sum_{\ell = 1}^{[T]} \sum_{\substack{(m-1)T \leq t_{k} < mT \\ \ell-1 \leq |t_{j}-t_{k}| < \ell}} 1 
\ll_{\epsilon} (mT)^{2+\epsilon} \sum_{\ell = 1}^{[T]} \frac{1}{\ell} 
\ll_{\epsilon} (mT)^{2+\epsilon},
\end{equation*}
where $(m-1)T \leq t_{j} < mT$. Therefore, we derive
\begin{equation}\label{J-2}
J(T, V, \Delta; m, \tau) \ll_{\epsilon} \Delta^{-1} (mT)^{2+\epsilon} V e^{-2m} \sum_{(m-1)T \leq t_{j} < mT} |L_{j}(\tau)|^{2}.
\end{equation}
We combine~\eqref{J-2} with the mean value estimate~\eqref{mean-value-estimate-2}, getting
\begin{equation*}
J(T, V, \Delta; m, \tau) \ll_{\epsilon} (mT)^{5+2\eta+\epsilon} e^{-2m} (1+|\tau|)^{A}.
\end{equation*}
Plugging this bound into~\eqref{sup} yields
\begin{align}\label{Balog-Biro-Harcos-Maga-analogue}
\begin{split}
\frac{1}{\Delta} \int_{V}^{V+\Delta} \left|\sum_{j} X^{it_{j}} e^{-\frac{t_{j}}{T}} \right|^{2} dX 
&\ll_{\epsilon} (\Delta^{-1} NV^{2}+T^{4})(NTV)^{\epsilon}+\frac{T^{5+2\eta+\epsilon} V}{\Delta N}\\
&\ll_{\epsilon} (\Delta^{-1} V^{\frac{3}{2}} T^{\frac{5}{2}+\eta}+T^{4})(TV)^{\epsilon},
\end{split}
\end{align}
where we optimise $N = V^{-1/2} T^{5/2+\eta}$. The estimate~\eqref{Balog-Biro-Harcos-Maga-analogue} corresponds to~\cite[Equation~(14)]{BalogBiroHarcosMaga2019} that improves upon the work of Cherubini--Guerreiro~\cite[Lemma 4.4]{CherubiniGuerreiro2018}. From our assumption on $N$, the bound~\eqref{Balog-Biro-Harcos-Maga-analogue} is valid if and only if $T > V^{1/(5+2\eta)}$. In the range $T \leq A^{1/(5+2\eta)}$,~we execute an evaluation of the smoothed second moment~\eqref{smoothened} in a repetitive fashion. In fact, we distribute the spectral parameters $t_{j}$ into intervals of length $T$, apply the Cauchy--Schwarz inequality for the sum over $m$, square out the various $j$-subsums, integrate the double sum over $t_{j}$ and $t_{k}$ in $X$, and use the Weyl law. Consequently, we conclude that
\begin{equation}\label{square-mean-integral-2}
\frac{1}{\Delta} \int_{V}^{V+\Delta} \left|\sum_{j} X^{it_{j}} e^{-\frac{t_{j}}{T}} \right|^{2} dX 
\ll \Delta^{-1} V \sum_{m \geq 1} m^{2} e^{-2m} \sum_{(m-1)T \leq t_{j}, t_{k} \leq mT} \frac{1}{1+|t_{j}-t_{k}|}
\ll_{\epsilon} \Delta^{-1} VT^{5+\epsilon}.
\end{equation}
On the other hand, if we use the trivial bound $S(T, X) \ll T^{3}$, then we find that the left-hand side of~\eqref{square-mean-integral-2} is $O(T^{6})$. Gathering these estimates together finishes the proof.
\end{proof}

It is beneficial to observe how Theorem~\ref{smoothing} improves upon the trivial bound. We consider
\begin{equation*}
\frac{1}{N} \sum_{n \in \mathcal{O}^{\ast}} h(|n|) \sum_{j} \alpha_{j} \hat{\varphi}(t_{j}) |\lambda_{j}(n)|^{2}
 = \frac{1}{N} \sum_{j} \sum_{n} h(|n|) \hat{\varphi}(t_{j}) |\nu_{j}(n)|^{2}
 = c \sum_{j} \hat{\varphi}(t_{j})+O(N^{-\frac{1}{2}} T^{3+\eta+\epsilon}).
\end{equation*}
We perform the square mean integral on both sides in view of Corollary~\ref{second-moment-of-spectral-side} and the Cauchy--Schwarz inequality to deduce under the assumption~\eqref{assumption} that
\begin{equation*}
\frac{1}{\Delta} \int_{V}^{V+\Delta} \left|\sum_{j} X^{it_{j}} e^{-\frac{t_{j}}{T}} \right|^{2} dX 
\ll (N \Delta^{-1} V^{2}+T^{4})(NTV)^{\epsilon}+\frac{T^{6+2\eta+\epsilon}}{N}.
\end{equation*}
We optimise $N = \sqrt{\Delta} T^{3+\eta} V^{-1}$ and hence the trivial bound for the left-hand side of~\eqref{square-mean-integral-1}~is $O_{\epsilon}(\Delta^{-1/2} V T^{3+\eta+\epsilon})$. Nonetheless, this estimate is justified only when $T > (\Delta^{-1/2} V)^{1/(3+\eta)}$.~If $V = \Delta$, then we have $O_{\epsilon}(\sqrt{V} T^{3+\eta+\epsilon})$, which is weaker than~\eqref{square-mean-integral-1}.

We replace the weight function $e^{-t_{j}/T}$ with a smooth characteristic function $\chi(t)$ of $[1, T]$. This yields the following second moment bound for $S(T, X)$.
\begin{corollary}\label{replacement}
Keep the notation and assumptions as above. For any $\epsilon > 0$, we have that
\begin{equation}\label{replacement-2}
\frac{1}{\Delta} \int_{V}^{V+\Delta} |S(T, X)|^{2} dX \ll V^{\epsilon} \times 
	\begin{cases}
	T^{5} \min(\Delta^{-1} V, T) & \text{if $0 < T \leq V^{\frac{1}{5+2\eta}}$},\\
	\Delta^{-1} V^{\frac{3}{2}} T^{\frac{5}{2}+\eta} & \text{if $V^{\frac{1}{5+2\eta}} < T \leq (\Delta^{-2} V^{3})^{\frac{1}{3-2\eta}}$},\\
	T^{4} & \text{if $(\Delta^{-2} V^{3})^{\frac{1}{3-2\eta}} < T \leq \sqrt{V}$}
	\end{cases}
\end{equation}
In particular, the left-hand side can uniformly be bounded by $O_{\epsilon}(\Delta^{-1} T^{5/2+\eta} V^{3/2+\epsilon})$.
\end{corollary}

\begin{proof}
We choose a smooth test function $g(\xi)$ such that $\supp(g) \subseteq [1/2, T+1/2]$,~$0 \leq g(\xi) \leq 1$ and $g(\xi) = 1$ for $1 \leq \xi \leq T$. The Weyl law $\# \{j: U \leq t_{j} \leq U+1 \} \ll U^{2}$ then yields
\begin{equation*}
S(T, X) = \sum_{j} g(t_{j}) X^{it_{j}}+O(T^{2}).
\end{equation*}
We use the less standard notation $\hat{g}(t)$ for the Fourier transform of $g(\xi) e^{\xi/T}$, deducing
\begin{align*}
\int_{|t| \geq 1} \hat{g}(t) e(-ty) dt
& = -\frac{1}{\pi} \int_{-\infty}^{\infty} \frac{d}{d\xi} \left(g(\xi) e^{\frac{\xi}{T}} \right) 
\int_{1}^{\infty} \sin(2\pi t(\xi-y)) \frac{dt d\xi}{t}\\
&\ll \frac{1}{1+|y|}+\frac{1}{1+|T-y|}+\frac{\log(T+|y|)}{T}.
\end{align*}
The Fourier inversion implies that
\begin{equation}\label{inversion}
g(t) e^{t/T} = \int_{|\xi| \leq 1} \hat{g}(\xi) e(-\xi t) d\xi+O \left(\frac{1}{1+|t|}+\frac{1}{1+|T-t|}+\frac{\log(T+|t|)}{T} \right).
\end{equation}
For notational convenience, we set
\begin{equation*}
k(T, X, \xi) \coloneqq \sum_{j} (Xe^{-2\pi \xi})^{it_{j}} e^{-\frac{t_{j}}{T}}.
\end{equation*}
It then follows that
\begin{equation*}
\sum_{j} g(t_{j}) X^{it_{j}} = \int_{|\xi| \leq 1} \hat{g}(\xi) k(T, X, \xi) d\xi+O \left(\sum_{j} \left(\frac{e^{-\frac{t_{j}}{T}}}{t_{j}}
 + \frac{e^{\frac{t_{j}}{T}}}{1+|T-t_{j}|}+\frac{\log(T+t_{j}) e^{-\frac{t_{j}}{T}}}{T} \right) \right),
\end{equation*}
where the error term is $O(T^{2})$ by the Weyl law. We therefore deduce that
\begin{align}\label{split}
\begin{split}
\frac{1}{\Delta} \int_{V}^{V+\Delta} |S(T, X)|^{2} dX 
&\ll \frac{1}{\Delta} \int_{V}^{V+\Delta} \left|\int_{|\xi| \leq 1} \hat{g}(\xi) k(T, X, \xi) d\xi \right|^{2} dX+T^{4}\\
&\ll \frac{1}{\Delta} \int_{V}^{V+\Delta} \left|\int_{|\xi| \leq \delta} \hat{g}(\xi) k(T, X, \xi) d\xi \right|^{2} dX\\
& + \frac{1}{\Delta} \int_{V}^{V+\Delta} \left|\int_{\delta < |\xi| \leq 1} \hat{g}(\xi) k(T, X, \xi) d\xi \right|^{2} dX+T^{4}
\end{split}
\end{align}
with a parameter $\delta > 0$ at our disposal. As for the first (resp. second) term on the right-hand side of~\eqref{split}, we use Theorem~\ref{smoothing}, the Cauchy--Schwarz inequality in $\xi$, and the first (resp. second) bound in $\hat{g}(t) \ll \min(T, |t|^{-1})$. The desired claim follows if we choose $\delta = T^{-2}$.
\end{proof}

\subsection{Completion of the Proofs}\label{completion-of-the-proofs}
We aim to establish Theorem~\ref{second-moment} by bounding~the second moment of $S(T, X)$. It would be beneficial to reproduce an argument pursuant to the work~of Cherubini--Guerreiro~\cite{CherubiniGuerreiro2018}. To start with, we invoke Lemma~\ref{Nakasuji}. By partial summation, the spectral sum over $|t_{j}| \leq T$ in~\eqref{explicit-formula} can be expressed as twice the real part of
\begin{equation*}
\sum_{t_{j} \leq T} \frac{X^{1+it_{j}}}{1+it_{j}}
 = X \frac{S(T, X)}{1+iT}+iX \int_{1}^{T} \frac{S(U, X)}{(1+iU)^{2}} dU.
\end{equation*}
By Corollary~\ref{replacement} together with a repeated application of the Cauchy--Schwarz inequality, the average over $V \leq X \leq V+\Delta$ is bounded as
\begin{align*}
\frac{1}{\Delta} \int_{V}^{V+\Delta} \left|\sum_{t_{j} \leq T} \frac{X^{1+it_{j}}}{1+it_{j}} \right|^{2} dX 
&\ll \frac{V^{2}}{\Delta T^{2}} \int_{V}^{V+\Delta} |S(T, X)|^{2} dX
 + \frac{V^{2} \log T}{\Delta} \int_{1}^{T} \int_{V}^{V+\Delta} |S(U, X)|^{2} \frac{dX dU}{U^{3}}\\
&\ll_{\epsilon} \Delta^{-1} T^{\frac{1}{2}+\eta} V^{\frac{7}{2}+\epsilon}.
\end{align*}
It therefore follows that
\begin{equation}\label{final-bound}
\frac{1}{\Delta} \int_{V}^{V+\Delta} |E_{\Gamma}(X)|^{2} dX 
\ll_{\epsilon} \Delta^{-1} T^{\frac{1}{2}+\eta} V^{\frac{7}{2}+\epsilon}+\frac{V^{4}}{T^{2}}(\log V)^{2},
\end{equation}
where the second term comes from the error term in Lemma~\ref{Nakasuji}. We estimate the right-hand side of~\eqref{final-bound} as
\begin{equation*}
O_{\epsilon}(V^{4-\frac{2}{5+2\eta}+\epsilon} \Delta^{-\frac{4}{5+2\eta}})
\end{equation*}
where we optimise $T = (\Delta^{2} V)^{1/(5+2\eta)}$ with the restriction $1 \ll \Delta \leq V^{(3+2\eta)/4}$ stemming from the initial assumption $V^{\epsilon} \ll T \leq \sqrt{V}$. This concludes the proof of~\eqref{main}. \qed\\

We establish the following sophisticated version instead of the second assertion~\eqref{main-smooth-explicit-formula},~since the substitution of $\eta = \theta = 0$ in Theorem~\ref{second-moment-2} immediately yields the bound~\eqref{main-smooth-explicit-formula}.
\begin{theorem}\label{second-moment-2}
Keep the notation as above. Assume that
\begin{equation*}
V^{\frac{3}{10}-\frac{8\theta}{5}} \leq \Delta \leq V^{\frac{19}{20}+\frac{\eta}{2}-\frac{8\theta}{5}}.
\end{equation*}
For any $\epsilon > 0$, we have that
\begin{equation}\label{second-moment-3}
\frac{1}{\Delta} \int_{V}^{V+\Delta} |E_{\Gamma}(X)|^{2} dX 
\ll_{\epsilon} V^{\frac{44+8\theta+16\eta(2+\theta)}{13+10\eta}+\epsilon} \Delta^{-\frac{8}{13+10\eta}}.
\end{equation}
\end{theorem}

\begin{proof}
We leverage the smooth explicit formula due to the barrier $O_{\epsilon}(X^{3/2+\epsilon})$ in Lemma~\ref{Nakasuji}. The strategy is rooted in the work of Soundararajan--Young~\cite{SoundararajanYoung2013}. Let $X^{1/2+\epsilon} \leq Y \leq X$ be a parameter to be chosen later, and let $k(u)$ be a smooth real-valued function with compact support on $(Y, 2Y)$. We suppose that $\int_{-\infty}^{\infty} k(u) du = \int_{Y}^{2Y} k(u) du = 1$ and that for all $\ell \geq 0$,
\begin{equation*}
\int_{-\infty}^{\infty} |k^{(\ell)}(u)| du \ll_{j} Y^{-\ell}.
\end{equation*}
The smoothed version of our counting function is defined by
\begin{equation*}
\Psi_{\Gamma}(X; k) \coloneqq \int_{Y}^{2Y} \Psi_{\Gamma}(X+u) k(u) du
\end{equation*}
so that
\begin{equation*}
\Psi_{\Gamma}(X) = \Psi_{\Gamma}(X; k)-\int_{Y}^{2Y} (\Psi_{\Gamma}(X+u)-\Psi_{\Gamma}(X)) k(u) du.
\end{equation*}
We now invoke the smooth explicit formula of Balog et al.~\cite{BalogBiroCherubiniLaaksonen2022}. For $T, X, Y \gg 1$ with~$T, Y \leq X$ and $TY > X^{1+\xi}$ for some fixed $\xi > 0$, it asserts that
\begin{equation*}
\Psi_{\Gamma}(X; k) = \int_{Y}^{2Y} \left(\frac{(X+u)^{2}}{2}+\sum_{|t_{j}| \leq T} 
\frac{(X+u)^{s_{j}}}{s_{j}} \right) k(u) du+O_{\epsilon}(X^{2+\epsilon} T^{-1}+X^{2+\epsilon} Y^{-2}+X^{1+\epsilon}).
\end{equation*}
If we choose $T = X$, then an elementary manipulation gives
\begin{equation}\label{manipulation}
\Psi_{\Gamma}(X; k) = \frac{1}{2} \int_{Y}^{2Y} (X+u)^{2} k(u) du+E(X; k)+O_{\epsilon}(X^{1+\epsilon}),
\end{equation}
where we exploit the Weyl law along with integration by parts to find that the contribution~of terms with $|t_{j}| \geq X^{1+\xi}/Y$ is at most $O_{\epsilon}(X^{1+\epsilon})$. Thus the second term in~\eqref{manipulation} becomes
\begin{equation}\label{definition-of-E(X;k)}
E(X; k) = \sum_{|t_{j}| \leq X^{1+\xi}/Y} \frac{1}{s_{j}} \int_{Y}^{2Y} 
(X+u)^{s_{j}} k(u) du, \qquad X^{\frac{1}{2}+\epsilon} \leq Y \leq X,
\end{equation}
which is akin to~\cite[Equation (20)]{SoundararajanYoung2013}. A routine calculation leads to
\begin{equation}\label{approximation}
\Psi_{\Gamma}(X)-\frac{X^{2}}{2} = E_{\Gamma}(X)
 = E(X; k)+O_{\epsilon}(X^{\frac{4\theta+6}{5}+\epsilon} Y^{\frac{2}{5}}+X^{1+\epsilon}).
\end{equation}
The Bykovski\u{\i}-type theorem is valid for $X^{1/3+\epsilon} \leq Y \leq X$ and implies that $O_{\epsilon}(X^{1+\epsilon})$ in~\eqref{approximation} is absorbed into $O_{\epsilon}(X^{(4\theta+6)/5+\epsilon} Y^{2/5})$ regardless of how we optimise the parameter $Y$.

We now handle the second moment bound for $E_{\Gamma}(X)$, which is estimated via~\eqref{approximation} as
\begin{equation}\label{upper-bound}
\frac{1}{\Delta} \int_{V}^{V+\Delta} |E_{\Gamma}(X)|^{2} dX 
\ll_{\epsilon} \frac{1}{\Delta} \int_{V}^{V+\Delta} |E(X; k)|^{2} dX+V^{\frac{8\theta+12}{5}+\epsilon} Y^{\frac{4}{5}}.
\end{equation}
By partial summation, the sum over $|t_{j}| \leq X^{1+\xi}/Y$ in $E(X; k)$ equals twice the real part of
\begin{equation*}
\sum_{t_{j} \leq Z} \frac{(X+u)^{1+it_{j}}}{1+it_{j}}
 = (X+u) \frac{S(Z, X+u)}{1+iZ}+i(X+u) \int_{1}^{Z} \frac{S(U, X+u)}{(1+iU)^{2}} dU
\end{equation*}
with $Z \coloneqq X^{1+\xi} Y^{-1}$. Using Corollary~\ref{replacement} with a repeated application of the Cauchy--Schwarz inequality, the average over $V \leq X \leq V+\Delta$ boils down to
\begin{align*}
&\frac{1}{\Delta} \int_{V}^{V+\Delta} \left|\sum_{t_{j} \leq Z} \frac{(X+u)^{s_{j}}}{s_{j}} \right|^{2} dX\\
&\ll \frac{V^{2}}{\Delta Z_{0}^{2}} \int_{V}^{V+\Delta} |S(Z_{0}, X+u)|^{2} dX
 + \frac{V^{2} \log Z_{0}}{\Delta} \int_{1}^{Z_{0}} \int_{V}^{V+\Delta} |S(U, X+u)|^{2} \frac{dX dU}{U^{3}}\\
&\ll \Delta^{-1} Z_{0}^{\frac{1}{2}+\eta} V^{\frac{7}{2}+\epsilon} 
\ll \Delta^{-1} V^{4+\eta+\epsilon} Y^{-\frac{1}{2}-\eta}
\end{align*}
with $Z_{0} \coloneqq V^{1+\epsilon} Y^{-1}$ and $V^{1/2+\epsilon} \leq Y \leq V$. 
Using~\eqref{definition-of-E(X;k)} and~\eqref{upper-bound}, we thus conclude that
\begin{equation*}
\frac{1}{\Delta} \int_{V}^{V+\Delta} |E_{\Gamma}(X)|^{2} dX 
\ll_{\epsilon} \Delta^{-1} V^{4+\eta+\epsilon} Y^{-\frac{1}{2}-\eta}+V^{\frac{8\theta+12}{5}+\epsilon} Y^{\frac{4}{5}}.
\end{equation*}
Balancing the right-hand side with
\begin{equation}\label{choice}
Y = V^{\frac{10\eta+16(1-\theta)}{13+10\eta}} \Delta^{-\frac{10}{13+10\eta}}
\end{equation}
yields
\begin{equation*}
\frac{1}{\Delta} \int_{V}^{V+\Delta} |E_{\Gamma}(X)|^{2} dX 
\ll_{\epsilon} V^{\frac{44+8\theta+16\eta(2+\theta)}{13+10\eta}+\epsilon} \Delta^{-\frac{8}{13+10\eta}}.
\end{equation*}
This completes the proof of Theorem~\ref{second-moment-2}.
\end{proof}

\section{Corollaries and Theorem~\ref{main-2}}\label{corollaries-and-Theorem-1.3}
This section is devoted to the proof of Theorem~\ref{main-2} and some byproducts of Theorem~\ref{second-moment}.

\subsection{Bounds in Short Intervals}\label{bounds-in-short-intervals}
We discuss a three-dimensional analogue of~\cite[Theorem 2]{BalogBiroHarcosMaga2019}. The second moment bound~\eqref{second-moment-3} leads to the following short interval estimate.
\begin{theorem}\label{improved-short-interval}
Keep the notation as above. Assume that
\begin{equation*}
V^{-\frac{1}{2}} \leq h \leq V^{\frac{3-16\theta}{13+10\eta}} \Delta^{-\frac{10}{13+10\eta}}, \qquad 
V^{\frac{3}{10}-\frac{8\theta}{5}} \leq \Delta \leq V^{\frac{19}{20}+\frac{\eta}{2}-\frac{8\theta}{5}}.
\end{equation*}
Then we have for any $\epsilon$ that
\begin{equation}\label{improved-short-interval-2}
\frac{1}{\Delta} \int_{V}^{V+\Delta} |\Psi_{\Gamma}(X)-\Psi_{\Gamma}(X-hX)-h(1-h/2)X^{2}|^{2} dX 
\ll_{\epsilon} \Delta^{-1} V^{7/2+\epsilon} h^{-1/2-\eta}.
\end{equation}
\end{theorem}

\begin{remark}
We emphasise that the bound~\eqref{improved-short-interval-2} does not contain the additional term $V^{3}(\log V)^{2}$. This enables one to obtain a lower exponent up to the record of the subconvex exponent $\eta$.
\end{remark}

\begin{proof}
The smooth explicit formula leads to
\begin{equation*}
\Psi_{\Gamma}(X)-\Psi_{\Gamma}(X-hX)-h \left(1-\frac{h}{2} \right) X^{2}
 = E(X; k)-E(X(1-h); k)+O_{\epsilon}(X^{\frac{4\theta+6}{5}+\epsilon} Y^{\frac{2}{5}}),
\end{equation*}
where the notation is the same as in the proof of Theorem~\ref{second-moment-2}. The goal is to evaluate~the second moment of this expression over the short interval $V \leq X \leq V+\Delta$. The contribution of spectral parameters up to $h^{-1}$ can be bounded by the Cauchy--Schwarz inequality as
\begin{equation}\label{contribution-of-the-spectral-parameters-up-to-1/h}
\frac{1}{\Delta} \int_{V}^{V+\Delta} \left|\int_{1-h}^{1} X S \left(\frac{1}{h}, u+X\zeta \right) d\zeta \right|^{2} dX
\ll \frac{V^{2}}{\Delta} \left(\int_{1-h}^{1} \frac{d\zeta}{\zeta} \right) 
\left(\int_{1-h}^{1} \int_{V}^{V+\Delta} \left|S \left(\frac{1}{h}, u+X\zeta \right) \right|^{2} dX \frac{d\zeta}{\zeta} \right),
\end{equation}
where we disregard the integral over $u$ in~\eqref{definition-of-E(X;k)} due to the assumption on $k(u)$. In view of the uniform bound in Corollary~\ref{replacement}, the integral in $X$ is $O_{\epsilon}(V^{3/2+\epsilon} h^{-5/2-\eta})$ and the right-hand side of~\eqref{contribution-of-the-spectral-parameters-up-to-1/h} is $O_{\epsilon}(\Delta^{-1} V^{7/2+\epsilon} h^{-1/2-\eta})$. Setting $Z = X^{1+\xi} Y^{-1}$, we estimate the contribution~of spectral parameters $h^{-1} < t_{j} \leq Z$ as
\begin{equation}\label{contribution-of-the-remaining-spectral-parameters}
\frac{1}{\Delta} \int_{V}^{V+\Delta} \left|\sum_{1/h < t_{j} \leq Z} \frac{(X+u)^{1+it_{j}}}{1+it_{j}} \right|^{2} dX
 + \frac{1}{(1-h) \Delta} \int_{(1-h)V}^{(1-h)(V+\Delta)} 
\left|\sum_{1/h < t_{j} \leq Z} \frac{(X+u)^{1+it_{j}}}{1+it_{j}} \right|^{2} dX.
\end{equation}
The first term can be evaluated via partial summation, the Cauchy--Schwarz inequality, and Theorem~\ref{smoothing} as
\begin{align*}
&\ll \frac{h^{2} V^{2}}{\Delta} \int_{V}^{V+\Delta} \left|S \left(\frac{1}{h}, X \right) \right|^{2} dX
 + \frac{V^{2}}{\Delta Z_{0}^{2}} \int_{V}^{V+\Delta} |S(Z_{0}, X)|^{2} dX\\
&\qquad + \frac{V^{2} \log(hZ_{0})}{\Delta} \int_{1/h}^{Z_{0}} \int_{V}^{V+\Delta} |S(U, X)|^{2} \frac{dX dU}{|1+iU|^{3}}\\
&\ll_{\epsilon} \Delta^{-1} V^{\frac{7}{2}+\epsilon} h^{-\frac{1}{2}-\eta}
 + \Delta^{-1} Z_{0}^{\frac{1}{2}+\eta} V^{\frac{7}{2}+\epsilon}
 + \Delta^{-1} V^{\frac{7}{2}+\epsilon} \int_{1/h}^{Z_{0}} U^{-\frac{1}{2}+\eta} dU\\
&\ll_{\epsilon} \Delta^{-1} V^{\frac{7}{2}+\epsilon}(h^{-\frac{1}{2}-\eta}+Z_{0}^{\frac{1}{2}+\eta})
\end{align*}
with $Z_{0} = V^{1+\epsilon} Y^{-1}$. Similarly, the second integral in~\eqref{contribution-of-the-remaining-spectral-parameters} is bounded by the same quantity. We note that the error term in the smooth explicit formula contributes $O_{\epsilon}(V^{(8\theta+12)/5+\epsilon} Y^{4/5})$. Consequently, the optimal choice becomes
\begin{equation*}
Y = V^{\frac{10\eta+16(1-\theta)}{13+10\eta}} \Delta^{-\frac{10}{13+10\eta}}
\end{equation*}
and the term involving $h$ dominates the other contributions. Since we have
\begin{align*}
\frac{3-16\theta}{13+10\eta}-\left(\frac{19}{20}+\frac{\eta}{2}-\frac{8\theta}{5} \right) \frac{10}{13+10\eta} &= -\frac{1}{2},\\
\frac{3-16\theta}{13+10\eta}-\left(\frac{3}{10}-\frac{8\theta}{5} \right) \frac{10}{13+10\eta} &= 0,
\end{align*}
our restrictions on $h$ and $\Delta$ in Theorem~\ref{improved-short-interval} make sense.
\end{proof}

\subsection{Bounds for Almost All $X$}\label{bounds-for-almost-all-X}
Koyama~\cite{Koyama2016} initiated the study of more effective bounds for $E_{\Gamma}(X)$ based on the notion of Gallagher~\cite{Gallagher1970}. His results are of the type
\begin{equation}\label{outside-a-closed-set-with-finite-logarithmic-measure}
E_{\Gamma}(X) \ll_{\epsilon} X^{\delta+\epsilon}
\end{equation}
outside a closed subset $\mathcal{A}$ of $[2, \infty]$ with finite logarithmic measure, namely $\int_{\mathcal{A}} dx/x < \infty$.~In particular, the estimate~\eqref{outside-a-closed-set-with-finite-logarithmic-measure} with $\delta = 7/10$ was shown in the two-dimensional setting for $\Gamma$ cocompact. Avdispahi\'{c}~\cite{Avdispahic2018,Avdispahic2018-2} generalised this idea to three dimensions. For $\Gamma = \PSL_{2}(\mathcal{O}_{K})$, his work demonstrates that $E_{\Gamma}(X) = O_{\epsilon}(X^{21/13}(\log X)^{2/13}(\log \log X)^{2/13+\epsilon})$ for $X \not \in \mathcal{A}$. In this direction, Theorem~\ref{main} improves his work to the extent that $\delta$ can be taken to be any number exceeding 3/2. In other words, we deduce the following corollary.
\begin{corollary}\label{finite-logarithmic-measure}
Keep the notation as above. For any $\delta > 3/2$, there exists a set $A \subseteq [2, \infty)$~of finite logarithmic measure such that $E_{\Gamma}(X) \ll X^{\delta}$ for $X \notin \mathcal{A}$. Indeed, the set $\mathcal{A} = \{X \geq 1: |E_{\Gamma}(X)| \geq X^{3/2}(\log X)^{1+\epsilon} \}$ has finite logarithmic measure.
\end{corollary}

When $\Gamma = \PSL_{2}(\Z)$, Iwaniec~\cite[p.187]{Iwaniec1984-2} has mentioned without a proof that ``One can~also show quite easily that $E_{\Gamma}(X) \ll_{\epsilon} X^{2/3+\epsilon}$ for almost all $X$'' (with the notation adjusted to ours). Corollary~\ref{finite-logarithmic-measure} establishes the three-dimensional analogue of his claim. The antecedent method based on Gallagher's notion is a detour as Corollary~\ref{finite-logarithmic-measure} is an immediate~consequence of our second moment bound. Heuristically, the Gallagherian exponent in the Prime Geodesic Theorem may be taken to be the conjectural exponent (the barrier in the explicit formula)~in pointwise bounds for $E_{\Gamma}(X)$.

\subsection{Averages of Class Numbers}\label{averages-of-class-numbers}
Sarnak~\cite{Sarnak1983} offered a bijective correspondence between norms $N(T)$ and primitive binary quadratic forms $Q(x, y) = ax^{2}+bxy+cy^{2}$ of discriminants $d = b^{2}-4ac > 0$ (``primitive'' means $(a, b, c) = 1$). Denote by $h(d)$ the number of inequivalent classes of such forms of discriminant $d$ and by $\epsilon_{d} = (t+\sqrt{d} u)/2$ the fundamental solution to the Pell-type equation $t^{2}-du^{2} = 4$ for $d \in \Omega$ and $(t, u) \in \mathcal{O} \times \mathcal{O}$, where $\Omega$ signifies the set of discriminants of binary quadratic forms over Gaussian integers, namely
\begin{equation*}
\Omega \coloneqq \{m \in \mathcal{O}: \text{$m \equiv y^{2} \tpmod 4$ for some $y \in \Z[i]$ and $m$ is not a perfect square} \}.
\end{equation*}
The final condition in the definition of $\Omega$ ensures that a form of discriminant $d \in \Omega$ does not factorise over $\mathcal{O}$. Automorphs of $Q(x, y)$ are given by $\pm P(t, u)$, where
\begin{equation*}
P(t, u) = 
	\begin{pmatrix}
	\dfrac{t-bu}{2} & -cu\\
	au & \dfrac{t+bu}{2}
	\end{pmatrix}.
\end{equation*}
For $u \ne 0$, $P(t, u)$ is hyperbolic with norm $(t+\sqrt{d} u)^{2}/4$ and trace $t$. We denote by $(t_{d}, u_{d})$ the fundamental solution to the Pell-type equation so that $P(t_{d}, u_{d})$ is a primitive hyperbolic matrix with norm $\epsilon_{d}^{2}$ and trace $t_{d}$. Therefore, $Q(x, y) \mapsto P(t, u)$ sends the primitive quadratic form $Q(x, y)$ of discriminant $d$ to the primitive hyperbolic element $P(t_{d}, u_{d})$. Hence, for any discriminant $d$, there are $h(d)$, the class number, primitive hyperbolic conjugacy classes and they all have the same norm $\epsilon_{d}^{2}$ and trace $t_{d}$. Sarnak~\cite[Corollary 4.1]{Sarnak1983} has established that
\begin{equation*}
\sum_{\substack{d \in \Omega \\ |\epsilon_{d}| \leq \sqrt{X}}} h(d) = \pi_{\Gamma}(X)
\coloneqq \# \{\text{$T_{0}$ primitive hyperbolic or loxodromic: $N(T_{0}) \leq X$} \}.
\end{equation*}
Consequently, we obtain the following corollary of Corollary~\ref{corollary-of-Theorem-1.1}.
\begin{theorem}\label{corollary}
For $1 \ll \Delta \leq V$ and $\epsilon > 0$, we have that
\begin{equation*}
\frac{1}{\Delta} \int_{V}^{V+\Delta} \left|\sum_{\substack{d \in \Omega \\ |\epsilon_{d}| \leq \sqrt{X}}} h(d)-\Li(X^{2}) \right|^{2} dX 
\ll_{\epsilon} V^{\frac{62}{17}+\epsilon} \Delta^{-\frac{12}{17}}.
\end{equation*}
\end{theorem}

Theorem~\ref{corollary} improves upon the work of Balkanova et al.~\cite[Corollary 1.7]{BalkanovaChatzakosCherubiniFrolenkovLaaksonen2019}\footnote{The right-hand side of~\cite[Equation (1.3)]{BalkanovaChatzakosCherubiniFrolenkovLaaksonen2019} seems incorrect. If we replace $X$ with $\sqrt{X}$ in the integrand of the left-hand side of their bound (1.3), then the correctness of the bound $\Delta^{-2/5} V^{18/5+\epsilon}$ is guaranteed.}.
\begin{remark}
Since
\begin{equation*}
\mathscr{D} \coloneqq \{d > 0: \text{$d \equiv 0, 1 \tpmod 4$ and $d$ is not a perfect square} \} \subset \Omega,
\end{equation*}
the pseudoprime counting function $\pi_{\Gamma}(X)$ with $\Gamma \subset \PSL_{2}(\C)$ counts many more norms than in two dimensions. Since the respective main terms in the asymptotic formul{\ae} for~$\pi_{\Gamma}(X)$ in the two and three dimensional settings are $\Li(X)$ and $\Li(X^{2})$, the contribution of the norms corresponding to loxodromic elements is considerable.
\end{remark}

\subsection{Uniform Pointwise Bounds}\label{uniform-pointwise-bounds}
In the 1990's, Motohashi~\cite{Motohashi1997-2} found an identity relating the smoothed fourth moment of the Riemann zeta function $\zeta(1/2+it)$ to the spectral cubic moment of automorphic $L$-functions attached to the modular group $\SL_{2}(\Z)$. Starting with his ground-breaking work, this kind of identity has been a recurring theme in number theory. For a sufficiently well-behaved test function $w$, his work demonstrates the spectral decomposition
\begin{equation}\label{Motohashi-formula}
\int_{-\infty}^{\infty} \left|\zeta \left(\frac{1}{2}+it \right) \right|^{4} w(t) dt 
\leftrightsquigarrow \sum_{j} L \left(\frac{1}{2}, \varphi_{j} \right)^{3} \check{w}(t_{j})+(\cdots),
\end{equation}
where the sum runs over Hecke--Maa{\ss} cusp forms $\varphi_{j}$ of Laplacian eigenvalue $1/4+t_{j}^{2}$ and $\check{w}$ is an elaborate integral transform of $w$ involving the Gau{\ss} hypergeometric function $_{2}F_{1}$. The formula~\eqref{Motohashi-formula} relates two completely different families of $L$-functions in an exact equality.

We now look over the ambient scenery. Michel--Venkatesh (\cite[\S 4.3.3]{MichelVenkatesh2006},~\cite[\S 4.5.3]{MichelVenkatesh2010}) offered a geometric and spectral manoeuvre to reprove the identity~\eqref{Motohashi-formula}. Following their perspective, Nelson~\cite{Nelson2020} studied the cubic moment of automorphic $L$-functions on $\GL_{2}$ via the utilisation of regularised diagonal periods of Eisenstein series. We refer the reader to the work of Wu~\cite{Wu2022} and Balkanova--Frolenkov--Wu~\cite{BalkanovaFrolenkovWu2022}. We record the result of Nelson~\cite{Nelson2020} as follows.
\begin{theorem}[{Nelson~\cite[Theorem 1]{Nelson2020}}]\label{Nelson}
Let $F$ be a number field with ad\`{e}le ring $\A$, and let $\chi$ be a quadratic character of $\A^{\times}/F^{\times}$. Then we have the Weyl-strength subconvex bound
\begin{equation*}
L \left(\frac{1}{2}+it, \chi \right) \ll_{\epsilon} (1+|t|)^{O(1)} C(\chi)^{\frac{1}{6}+\epsilon}.
\end{equation*}
\end{theorem}

\begin{proof}[Proof of Theorem~\ref{main-2}]
Let $k(u)$ be the same test function as above. It then follows that
\begin{equation*}
E_{\Gamma}(X) = E(X; k)+O_{\epsilon}(X^{\frac{4\theta+6}{5}+\epsilon} Y^{\frac{2}{5}}+X^{1+\epsilon})
\end{equation*}
where $E(X; k)$ is defined in~\eqref{definition-of-E(X;k)}. Invoking the bound~\eqref{freedom}, we deduce
\begin{equation*}
E(X; k) \ll_{\epsilon} X^{2+\frac{\eta}{2}+\epsilon} Y^{-\frac{3+2\eta}{4}}+X^{2+\epsilon} Y^{-1}.
\end{equation*}
Therefore, we obtain
\begin{equation*}
E_{\Gamma}(X) \ll_{\epsilon} X^{2+\frac{\eta}{2}+\epsilon} Y^{-\frac{3+2\eta}{4}}
 + X^{\frac{4\theta+6}{5}+\epsilon} Y^{\frac{2}{5}}+X^{2+\epsilon} Y^{-1}.
\end{equation*}
We optimise $Y = X^{(10\eta+16(1-\theta))/(23+10\eta)}$ to conclude the proof of Theorem~\ref{main-2}.
\end{proof}

\section{Conclusions and Conjectures}\label{conclusions-and-conjectures}
There remain many unfathomed questions concerning the spectral exponential sum. This section aims to formulate a conjecture on the best possible bound for $S(T, X)$ corresponding to that in two dimensions, and to establish Theorem~\ref{Laaksonen-analogue}. Numerical evidence is also given.

\subsection{Conditional Improvements}\label{conditional-improvements}
Petridis--Risager~\cite[Conjecture 2.2]{PetridisRisager2017} have conjectured for $\Gamma = \PSL_{2}(\Z)$ that the spectral exponential sum $S(T, X)$ exhibits square-root cancellation. This subsection aims to speculate what the true order of $E_{\Gamma}(X)$ for $\Gamma = \PSL_{2}(\mathcal{O})$ should be. The bound~$E_{\Gamma}(X) = O_{\epsilon}(X^{1+\epsilon})$ seems reasonable and we are led to the following\footnote{It is beneficial to compare Conjecture~\ref{cancellation} with~\eqref{freedom} to see how strong the bound~\eqref{cancellation-2} is.}.
\begin{conjecture}\label{cancellation}
Keep the notation as above. Then we have for any $\epsilon > 0$ that
\begin{equation}\label{cancellation-2}
S(T, X) \ll_{\epsilon} T^{2+\epsilon} X^{\epsilon}.
\end{equation}
\end{conjecture}

Nonetheless, both the conjectures for $E_{\Gamma}(X)$ and $S(T, X)$ are no longer tractable, and~we reach a discrepancy between them. This phenomenon was bespoken in~\cite[Remarks~1.5~\&~3.1]{BalkanovaChatzakosCherubiniFrolenkovLaaksonen2019}. In order to describe the discrepancy, we spell out asymmetry between $E_{\Gamma}(X)$ and $S(T, X)$. If $T = \sqrt{X}$, then the error term in Lemma~\ref{Nakasuji} is $O_{\epsilon}(X^{3/2+\epsilon})$ and the spectral sum $\sum_{|t_{j}| \leq T} s_{j}^{-1} X^{s_{j}}$ is bounded by the same quantity under Conjecture~\ref{cancellation}. Theorem~\ref{Laaksonen-analogue} yields Conjecture~\ref{cancellation}~for fixed $X \geq 1$ as $T \to \infty$, which bears out the legitimacy of the bound $E_{\Gamma}(X) = O_{\epsilon}(X^{3/2+\epsilon})$. If we provisionally ignore the restriction $1 \leq T \leq \sqrt{X}$ and assume that $S(T, X) = O_{\epsilon}(T^{1+\epsilon} X^{\epsilon})$, then we balance with $T = X$ to obtain $E_{\Gamma}(X) = O_{\epsilon}(X^{1+\epsilon})$. There are reasons suggesting~that this bound may possibly hold. The work of Nakasuji~\cite{Nakasuji2001} is a two-dimensional counterpart~of Hejhal~\cite{Hejhal1976} and reveals that the bound $E_{\Gamma}(X) = O_{\epsilon}(X^{1+\epsilon})$, if true, would be optimal.
\begin{theorem}[Nakasuji~\cite{Nakasuji2001}]\label{omega-result}
Let $\Gamma \subset \PSL_{2}(\C)$ be a cocompact group, or a cofinite subgroup such that $\sum_{\gamma_{j} > 0} X^{\beta_{j}-1}/\gamma_{j}^{2} \ll (\log X)^{-3}$, where $\beta_{j}+i\gamma_{j}$ are poles of the scattering determinant. Then we have that
\begin{equation*}
\Psi_{\Gamma}(X)
 = \frac{X^{2}}{2}+\sum_{1 < s_{j} < 2} \frac{X^{s_{j}}}{s_{j}}+\Omega_{\pm}(X(\log \log X)^{\frac{1}{3}}).
\end{equation*}
\end{theorem}

On the other hand, we conditionally improve upon the exponents shown in Theorem~\ref{main-2}.
\begin{theorem}\label{under-conjecture-5.1}
Keep the notation as above. We assume Conjecture~\ref{cancellation}. Then we have that
\begin{equation*}
E_{\Gamma}(X) \ll_{\epsilon} X^{\frac{2(5+2\theta)}{7}+\epsilon}.
\end{equation*}
In particular, the current record $\theta = 1/6$ gives the exponent $32/21 \approx 1.52381$. The additional assumption of $\theta = 0$ gives the exponent $10/7 \approx 1.42857$.
\end{theorem}

\begin{proof}
We use the smooth explicit formula to deduce~\eqref{approximation} and optimise $Y = X^{4(1-\theta)/7}$.
\end{proof}

\subsection{Spectral Exponential Sum}\label{spectral-exponential-sum}
This section sheds some light on the order of magnitude of $S(T, X)$ for fixed $X \geq 1$ as $T \to \infty$. Any technical difficulty is not foreseen in the spectral aspect and we prove an asymptotic formula for $S(T, X)$ from the Weyl law with relative~ease.

Before embarking on the proof of Theorem~\ref{Laaksonen-analogue}, we shall invoke the classical situation. Let $\zeta(s)$ be the Riemann zeta function and denote its nontrivial zeroes by $\rho = \beta+i\gamma$. Let~$\Lambda(x)$ be the von Mangoldt function extended to $\R$ by letting $\Lambda(x) = 0$ for $x \in \R \backslash \mathbb{N}$. The work~of Landau demonstrates the behaviour of an exponential sum counted with the Riemann zeroes.
\begin{theorem}[{Landau~\cite[Satz~1]{Landau1912}}]\label{Landau}
For fixed $X \geq 1$, we have that
\begin{equation}\label{Landaus-formula}
\sum_{0 < \gamma \leq T} x^{\rho} = -\frac{T}{2\pi} \Lambda(x)+O(\log T)
\end{equation}
as $T \to \infty$.
\end{theorem}

Theorem~\ref{Landau} follows from a fairly elementary contour integration. Theorem~\ref{Landau} reveals~that the right-hand side of~\eqref{Landaus-formula} grows by an order of $T$ whenever $x$ attains to a prime power.~The empirical data of Odlyzko~\cite{Odlyzko2014} visualises the real and imaginary parts of the normalised sum $T^{-1} \sum_{0 < \gamma \leq T} x^{\rho}$ as shall be seen in~\cite{PetridisRisager2017}. The peaks at primes are much larger than at higher prime powers, since the von Mangoldt function $\Lambda(x)$ takes the value $\log p$ at powers of a fixed prime $p$. Hence any peak that is higher than the preceding ones corresponds to a new~prime.

The ensuing focus is on the spectral exponential sums. In the context of the full modular group $\Gamma = \PSL_{2}(\Z)$, Fujii~\cite{Fujii1984} announced the following asymptotic formula without a proof.
\begin{theorem}[{Fujii~\cite[Theorem 1]{Fujii1984}}]\label{Fujii-1984}
Let $\Gamma = \PSL_{2}(\Z)$. For fixed $X \geq 1$, we have that
\begin{equation}\label{Fujii}
\sum_{t_{j} \leq T} X^{it_{j}}
 = \frac{\vol(\Gamma \backslash \mathfrak{h}^{2})}{2\pi i \log X} T X^{iT}
 + \frac{T}{2\pi}(\sqrt{X}-X^{-1/2})^{-1} \Lambda_{\Gamma}(X)
 + \frac{T}{\pi} X^{-1/2} \Lambda(\sqrt{X})+O \left(\frac{T}{\log T} \right)
\end{equation}
as $T \to \infty$, where $\Lambda_{\Gamma}(X) = \log N(P_{0})$ if $X = N(P_{0})^{j}$ for $j \geq 1$ and $\Lambda_{\Gamma}(X) = 0$ otherwise.
\end{theorem}

The spectral exponential sum has the form $\sum e^{i\alpha t_{j}}$ in the work of Fujii; one takes $\alpha = \log X$ to undo it back to $S(T, X)$. An analogous method developed in his another article~\cite{Fujii1982} would also be useful for proving~\eqref{Fujii}. Laaksonen~\cite[Appendix]{PetridisRisager2017} has independently reached the~real part of~\eqref{Fujii} via the Selberg trace formula. His conclusions are derived by doubling the real part of ii) in~\cite[Theorem 1]{Fujii1984}. The real part (resp. imaginary part) of the spectral exponential sum $S(T, X)$ is termed the \textit{cosine kernel} (resp. \textit{sine kernel}). It behoves us to~mention that Theorem~\ref{Fujii-1984} reveals that the oscillations at peak points for each kernel are of different nature.

Now, the Selberg zeta function of $\Gamma \subset \PSL_{2}(\C)$ is built out of the Euler product as
\begin{equation}\label{Selberg-zeta-function}
Z(s) \coloneqq \prod_{\{T_{0} \}} \prod_{(k, \ell)} (1-a(T_{0})^{-2k} \overline{a(T_{0})}^{-2\ell} N(T_{0})^{-s}),
\end{equation}
where the outer product is taken over primitive hyperbolic and loxodromic conjugacy classes of $\Gamma$ and $(k, \ell)$ ranges over all the pairs of nonnegative integers satisfying $k \equiv \ell \tpmod{m(T_{0})}$, where $m(T)$ signifies the order of the torsion of the centraliser of $T$. Moreover, the complex numbers $a(T)$ and $a(T)^{-1}$ are the eigenvalues of $T$ with $|a(T)| > 1$ and $N(T) = |a(T)|^{2}$. The Selberg zeta function $Z(s)$ is known to be absolutely convergent in $\Re(s) > 2$ and extends to a meromorphic function on the whole complex plane $\C$ with a functional equation under $s \leftrightarrow 2-s$. Elstrodt--Grunewald--Mennicke~\cite[Lemma 4.2, p.208]{ElstrodtGrunewaldMennicke1998} have shown that
\begin{equation}\label{Elstrodt-Grunewald-Mennicke-1}
\frac{Z^{\prime}}{Z}(s) = \sum_{\{T \}} \widetilde{\Lambda_{\Gamma}}(T) N(T)^{-s}, \qquad 
\widetilde{\Lambda_{\Gamma}}(T) \coloneqq \frac{N(T) \log N(T_{0})}{m(T)|a(T)-a(T)^{-1}|^{2}},
\end{equation}
where the sum ranges over all hyperbolic and loxodromic conjugacy classes of $\Gamma$ and $T_{0}$ is~a primitive element associated to $T$. The denominator $m(T)|a(T)-a(T)^{-1}|^{2}$ in $\widetilde{\Lambda_{\Gamma}}(T)$ is not important since $\widetilde{\Lambda_{\Gamma}}(T) = \Lambda_{\Gamma}(T)+O(N(T)^{-1+\epsilon})$ by invoking that $m(T) \ne 1$ for finitely many classes. The Selberg zeta function for $\Gamma = \PSL_{2}(\mathcal{O})$ plays an important r\^{o}le in the proof~of Theorem~\ref{Laaksonen-analogue}. Given an analogue of the Riemann Hypothesis for $Z(s)$, we can arithmetically encode the spectral exponential sum. As far as the author knows, any asymptotic formula~for $S(T, X)$ in the spectral aspect in the three dimensional setting has yet to be investigated.

\begin{proof}[Proof of Theorem~\ref{Laaksonen-analogue}]
For notational convenience, we write $N_{\Gamma}(T)$ for $\Gamma = \PSL_{2}(\mathcal{O})$ as~$N(T)$. We invoke the Weyl law in~\cite[Theorem~5.1]{ElstrodtGrunewaldMennicke1998}, namely
\begin{equation}\label{Weyl's-law-in-a-precise-form}
N(T) = \frac{1}{4\pi} \int_{-T}^{T} \frac{\varphi^{\prime}}{\varphi}(1+it) dt
 + \mathcal{S}(T)+\frac{\vol(\mathcal{M})}{6\pi^{2}} T^{3}+b_{2} T \log T+b_{3} T+b_{4}+O \left(\frac{1}{T} \right)
\end{equation}
with explicitly computable constants $b_{j}$ for $j = 2, 3, 4$. The asymptotic~\eqref{Weyl's-law-in-a-precise-form} follows from~the discussion akin to Venkov~\cite{Venkov1979,Venkov1982}. We recall the definition of $M(T) = M_{\Gamma}(T)$, obtaining
\begin{align*}
S(T, X) &= \int_{1}^{T} X^{it} dN(t)\\
& = \frac{\vol(\mathcal{M})}{2\pi^{2}} \int_{1}^{T} X^{it} t^{2} dt
 - \int_{1}^{T} X^{it} dM(t)+\int_{1}^{T} X^{it} d\mathcal{S}(t)+O(\log T)\\
& = Y_{1}+Y_{2}+Y_{3}+O(\log T).
\end{align*}
We start with the analysis of the first term $Y_{1}$. By partial integration, we derive
\begin{align*}
Y_{1} &= \frac{\vol(\mathcal{M})}{2\pi^{2}} \frac{X^{iT}}{i \log X} T^{2}
 -\frac{\vol(\mathcal{M})}{2\pi i \log X} \int_{1}^{T} X^{it} t dt+O(1)\\
& = \frac{\vol(\mathcal{M})}{2\pi^{2} i \log X} X^{iT} T^{2}
 + \frac{\vol(\mathcal{M})}{2\pi(\log X)^{2}} X^{iT} T+O(1).
\end{align*}
The first term matches the main term in Theorem~\ref{Laaksonen-analogue}. As for the second term $Y_{2}$, we have to simplify the scattering determinant. To this end, we define the completed zeta function~by
\begin{equation*}
\xi_{K}(s) \coloneqq \left(\frac{\sqrt{|d_{K}|}}{2\pi} \right)^{s} \Gamma(s) \zeta_{K}(s)
\end{equation*}
with $\zeta_{K}(s)$ the Dedekind zeta function of $K = \Q(i)$. The functional equation is $\xi_{K}(s)$ reads
\begin{equation*}
\xi_{K}(s) = \xi_{K}(1-s).
\end{equation*}
The scattering determinant is given by
\begin{equation*}
\varphi(s) = \frac{\xi_{K}(s-1)}{\xi_{K}(s)} = \frac{2\pi}{\sqrt{|d_{K}|}} \frac{1}{s-1} \frac{\zeta_{K}(s-1)}{\zeta_{K}(s)}
\end{equation*}
so that
\begin{equation*}
-\frac{\varphi^{\prime}}{\varphi}(1+it) = \frac{\zeta_{K}^{\prime}}{\zeta_{K}}(1 \pm it)
 + \frac{\Gamma^{\prime}}{\Gamma}(1-it)+\frac{\Gamma^{\prime}}{\Gamma}(it)-2\log \pi+\frac{1}{it}.
\end{equation*}
We therefore use Stirling's formula to obtain
\begin{equation*}
M(t) = \frac{1}{2\pi i}(\log \zeta_{K}(1+it)-\log \zeta_{K}(1-it))+O(t \log t).
\end{equation*}
From the trivial bound for $\log \zeta_{K}(1 \pm it)$ and partial integration, it follows that
\begin{align*}
Y_{2} &= -\frac{1}{2\pi i} \int_{1}^{T} X^{it} d(\log \zeta_{K}(1+it)-\log \zeta_{K}(1-it))+O(\log T)\\
& = \frac{\log X}{2\pi} \int_{1}^{T} X^{it} \log \zeta_{K}(1+it) dt
 - \frac{\log X}{2\pi} \int_{1}^{T} X^{it} \log \zeta_{K}(1-it) dt+O(\log T).
\end{align*}
If we set $\xi = (\log X)^{-1}$, then the first term equals
\begin{align*}
&\frac{\log X}{2\pi i} \int_{1+i}^{1+iT} X^{t-1} \log \zeta_{K}(t) dt\\
& = \frac{\log X}{2\pi i} \left(\int_{1+\xi+i}^{1+\xi+iT}-\int_{1+iT}^{1+\xi+iT}+\int_{1+i}^{1+\xi+i} \right) 
X^{t-1} \log \zeta_{K}(t) dt\\
& = \frac{\log X}{2\pi i} \int_{1+\xi+i}^{1+\xi+iT} X^{t-1} \log \zeta_{K}(t) dt+O(\log T)\\
& = \frac{X^{\xi} \log X}{2\pi} \sum_{n = 2}^{\infty}
\frac{\Lambda_{K}(n)}{n^{1+\xi} \log n} \int_{1}^{T} \exp(it(\log X-\log n)) dt+O(\log T)\\
& = \frac{T}{2\pi X} \Lambda_{K}(X)+O(\log T).
\end{align*}
The second integral involving $\log \zeta_{K}(1-it)$ is bounded by $O(\log T)$.

To conclude the proof, it remains to handle $Y_{3}$. Integration by parts once again gives
\begin{equation*}
Y_{3} = X^{iT} \mathcal{S}(T)-i\log X \int_{1}^{T} X^{it} \mathcal{S}(t) dt+O(1)
 = X^{iT} \mathcal{S}(T)+Y_{4} \log X-iY_{5} \log X+O(1),
\end{equation*}
where
\begin{equation*}
Y_{4} = \int_{1}^{T} \sin(t \log X) \mathcal{S}(t) dt, \qquad Y_{5} = \int_{1}^{T} \cos(t \log X) \mathcal{S}(t) dt.
\end{equation*}
We now introduce the notation
\begin{equation*}
F(z) \coloneqq \log Z(z) \sin((1-z)i \log X),
\end{equation*}
where we take the principal value of the logarithm and the branch of $\log Z(z)$ is taken such that $\log Z(z)$ is real for $z > 1$. Imitating the treatment of $Y_{2}$, we consider the rectangle with vertices $2+\xi+i, 2+\xi+iT, 1+iT$ and $1+i$ with $\xi = (\log X)^{-1}$. It therefore follows that
\begin{equation*}
Y_{4} = \Im \left(\frac{1}{\pi i} \int_{1+i}^{1+iT} F(z) dz \right)\\
 = \Im \left(\frac{1}{\pi i} 
\left(\int_{2+\xi+i}^{2+\xi+iT}-\int_{1+iT}^{2+\xi+iT}+\int_{1+i}^{2+\xi+i} \right) F(z) dz \right).
\end{equation*}
Note that the third integral is trivially bounded. The first integral can be evaluated as
\begin{align*}
\frac{1}{\pi i} \int_{2+\xi+i}^{2+\xi+iT} F(z) dz &= \frac{1}{\pi} \int_{1}^{T} \log Z(2+\xi+it) \sin((t-(1+\xi)i) \log X) dt\\
& = -\frac{1}{2\pi i} \sum_{\{T \}} \frac{\widetilde{\Lambda_{\Gamma}}(T)}{N(T)^{2+\xi} \log N(T)} 
\int_{1}^{T} \exp(i(t-(1+\xi)i) \log X-it \log N(T)) dt\\
& + \frac{1}{2\pi i} \sum_{\{T \}} \frac{\widetilde{\Lambda_{\Gamma}}(T)}{N(T)^{2+\xi} \log N(T)} 
\int_{1}^{T} \exp(-i(t-(1+\xi)i) \log X-it \log N(T)) dt\\
& = -\frac{X^{1+\xi}}{2\pi i} \sum_{\{T \}} 
\frac{\widetilde{\Lambda_{\Gamma}}(T)}{N(T)^{2+\xi} \log N(T)} \int_{1}^{T} \exp(it(\log X-\log N(T))) dt+O(1)\\
& = -\frac{\widehat{\Lambda_{\Gamma}}(X)}{2\pi iX \log X} T+O(1),
\end{align*}
where we recall~\eqref{Elstrodt-Grunewald-Mennicke-1} and the definition of $\widehat{\Lambda_{\Gamma}}(X)$ in Theorem~\ref{Laaksonen-analogue}. This yields
\begin{equation*}
\Im \left(\frac{1}{\pi i} \int_{2+\xi+i}^{2+\xi+iT} F(z) dz \right)
 = \frac{\widehat{\Lambda_{\Gamma}}(X)}{2\pi X \log X} T+O(1).
\end{equation*}
The second integral can be bounded as
\begin{equation*}
-\frac{1}{\pi i} \int_{1+iT}^{2+\xi+iT} F(z) dz 
 = \frac{1}{\pi} \int_{1}^{2+\xi} \log Z(\sigma+iT) \sinh((\sigma-1+iT) \log X) d\sigma \ll G(T),
\end{equation*}
where $G(T)$ is the same as in the introduction. Here we use the fact that $X > 1$ is fixed.~In a similar fashion, we derive $Y_{5} \ll G(T)$. Notice that the order of $G(T)$ is always much bigger than $\log T$. Collecting the above estimates concludes the proof of Theorem~\ref{Laaksonen-analogue}.
\end{proof}



\subsection{Numerical Visualisations}\label{numerical-visualisations}
It is convenient to separate $N(T)$ into two parts as
\begin{equation*}
N(T) = N_{\mathrm{mean}}(T)+N_{\mathrm{fluct}}(T),
\end{equation*}
where $N_{\mathrm{mean}}(T)$ is a smooth function that describes a mean value of the number of $t_{j} \leq T$, and the fluctuating part $N_{\mathrm{fluct}}(T)$ is a function that oscillates around $0$:
\begin{equation}\label{fluct}
\lim_{T \to \infty} \frac{1}{T} \int_{0}^{T} N_{\mathrm{fluct}}(t) dt = 0.
\end{equation}
The Selberg trace formula yields the asymptotic formula (see Matthies~\cite{Matthies1995})
\begin{equation*}
N_{\mathrm{mean}}(T) = \frac{\vol(\mathcal{M})}{6\pi^{2}} T^{3}+a_{2} T \log T+a_{3} T+a_{4}+O \left(\frac{1}{T} \right)
\end{equation*}
with
\begin{align*}
a_{2} &= -\frac{3}{2\pi},\\
a_{3} &= \frac{1}{\pi} \left(\frac{13}{16} \log 2+\frac{7}{4} \log \pi
 - \log \Gamma \left(\frac{1}{4} \right)+\frac{2}{9} \log(2+\sqrt{3})+\frac{3}{2} \right),\\
a_{4} &= -\frac{1}{2}.
\end{align*}

\begin{figure}
	\centering
	\includegraphics[width=0.9\linewidth]{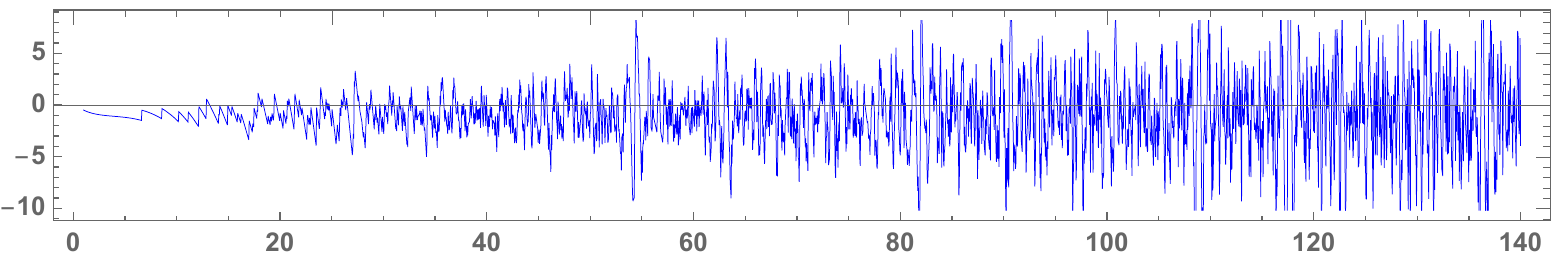}
	\caption{The Weyl remainder $N_{\mathrm{fluct}}(T)$ for $\Gamma = \PSL_{2}(\mathcal{O})$}
	\label{fluctuation-in-the-Weyl-remainder}
\end{figure}

\begin{figure}
	\centering
	\includegraphics[width=0.9\linewidth]{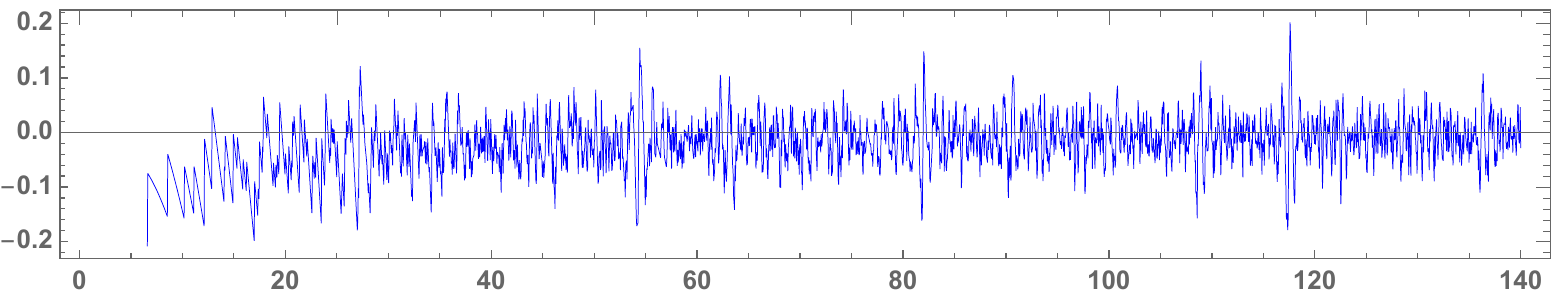}
	\caption{The scaled version $N_{\mathrm{fluct}}(T)/T$ of the Weyl remainder}
	\label{scaled-version}
\end{figure}

On the other hand, the fluctuation $N_{\mathrm{fluct}}(T)$ is visualised in Fig.~\ref{fluctuation-in-the-Weyl-remainder} whose computation~was executed via the utilisation of the first 13950 consecutive eigenvalues for the Picard manifold $\mathcal{M} = \Gamma \backslash \mathfrak{h}^{3}$ due to Then~\cite{AurichSteinerThen2012}. Taking account of Fig.~\ref{fluctuation-in-the-Weyl-remainder}, one sees that the Weyl remainder~grows with respect to $T$. We also plot the scaled version $N_{\mathrm{fluct}}(T)/T$ as shown in Fig.~\ref{scaled-version}. We easily observe that such a scaled Weyl remainder is bounded in the range $T \leq 140$ in our calculation; however we cannot even decide that the Weyl remainder in Fig.~\ref{scaled-version} decreases in $T$.~Therefore, one numerically confirms that
\begin{equation*}
\lim_{t \to \infty} |N_{\mathrm{fluct}}(t)| t^{-\sigma} = 0.
\end{equation*}
for any fixed $\sigma > 1$. Since the functions $N_{\mathrm{fluct}}(T)$ and $\mathcal{S}(T)$ are bounded by the same~quantity, we expect that $N_{\mathrm{fluct}}(T) = o(T)$ so that the lower order terms in Theorem~\ref{Laaksonen-analogue} do not absorb the main terms. Similarly to Hejhal's deduction in~\cite[Theorem 2.29]{Hejhal1983}, it follows that
\begin{equation}\label{Hejhal-analogue}
\mathcal{S}(T) = O \left(\frac{T^{2}}{\log T} \right).
\end{equation}
The Selberg zeta function has a higher density of zeroes, which in turn makes $\mathcal{S}(T)$ bigger. At the moment, the progress on the evaluation of $\mathcal{S}(T)$ remains elusive, and the bound~\eqref{Hejhal-analogue} is all we know to date. Note that the conjectural bound is $\mathcal{S}(T) = o(T)$. On the other hand, the analysis of $\int_{0}^{T} N_{\mathrm{fluct}}(t) dt$ was done for the modular surface by Booker--Platt~\cite{BookerPlatt2019} by means of Turing's method. Their reasoning would be applicable for proving the corresponding result in three dimensions. We now state the following immediate corollary of Theorem~\ref{Laaksonen-analogue}.
\begin{corollary}\label{cos-sin}
For fixed $X \geq 1$, we have that
\begin{align*}
\sum_{t_{j} \leq T} \cos(t_{j} \log X)
& = \frac{\vol(\mathcal{M})}{2\pi^{2}} \frac{\sin(T \log X)}{\log X} T^{2}+O \left(\frac{T^{2}}{\log T} \right),\\
\sum_{t_{j} \leq T} \sin(t_{j} \log X)
& = -\frac{\vol(\mathcal{M})}{2\pi^{2}} \frac{\cos(T \log X)}{\log X} T^{2}+O \left(\frac{T^{2}}{\log T} \right)
\end{align*}
as $T \to \infty$.
\end{corollary}

We visualise the real and imaginary parts of $S(T, X)$ in terms of $T$ and $X$. To this end,~let
\begin{equation*}
R(T, X) \coloneqq 2\sum_{t_{j} \leq T} \cos(t_{j} \log X), \qquad Q(T, X) \coloneqq 2\sum_{t_{j} \leq T} \sin(t_{j} \log X).
\end{equation*}
We remark that $R(T, X)$ can be expressed as twice the real part of $S(T, X)$. By Corollary~\ref{cos-sin}, we rescale
\begin{align*}
\Pi_{1}(T, X) &\coloneqq R(T, X)-\frac{\vol(\mathcal{M})}{\pi^{2}} \frac{\sin(T \log X)}{\log X} T^{2},\\
\Pi_{2}(T, X) &\coloneqq Q(T, X)+\frac{\vol(\mathcal{M})}{\pi^{2}} \frac{\cos(T \log X)}{\log X} T^{2}.
\end{align*}
The programs for our plots adapt the 13950 consecutive Laplacian eigenvalues associated to $\Gamma = \PSL_{2}(\mathcal{O})$\footnote{Laaksonen~\cite[Appendix]{PetridisRisager2017} clarifies that computations of spectral exponential sums are robust. In other words, the number of eigenvalues and their precision have no significant impact in two dimensions.}. The corresponding spectral parameters satisfy $t_{j} \leq 140$. For $\Gamma = \PSL_{2}(\Z)$, we could apply the 53000 consecutive eigenvalues, while in our case the number of eigenvalues hitherto computed is much less.

We start with considering the order of magnitude of $\Pi_{j}(T, X) \ (j = 1, 2)$ in terms of $T$. In Fig.~\ref{two-different-normalisations}, we plot two different normalisations of $\Pi_{j}(T, X)$, where we fix $X = 46.97$ as $T$~tends to infinity (the value 46.97 is one of the norms of primitive hyperbolic and loxodromic elements). Figure~\ref{two-different-normalisations} demonstrates that $\Pi_{j}(T, X)$ would have the order of magnitude of $T$.
\begin{figure}
	\centering
	\begin{subfigure}{0.49\columnwidth}
		\centering
		\includegraphics[width=0.99\linewidth]{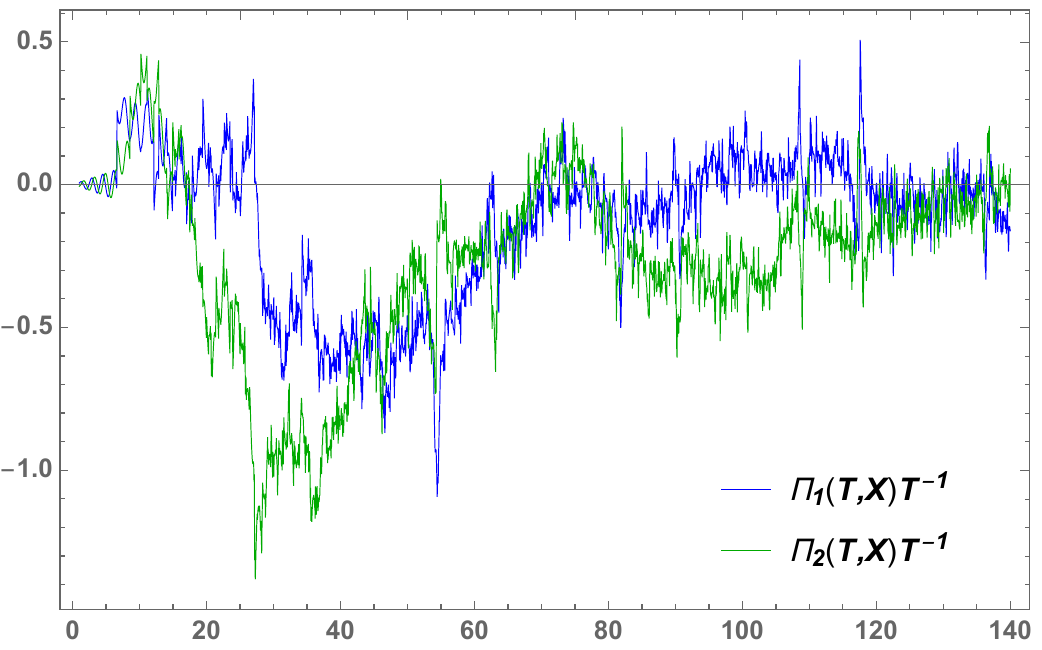}
		\caption{$\Pi_{1}(T, X) T^{-1}$ and $\Pi_{2}(T, X) T^{-1}$}
		\label{A}
	\end{subfigure}
	\begin{subfigure}{0.49\columnwidth}
		\centering
		\includegraphics[width=0.99\linewidth]{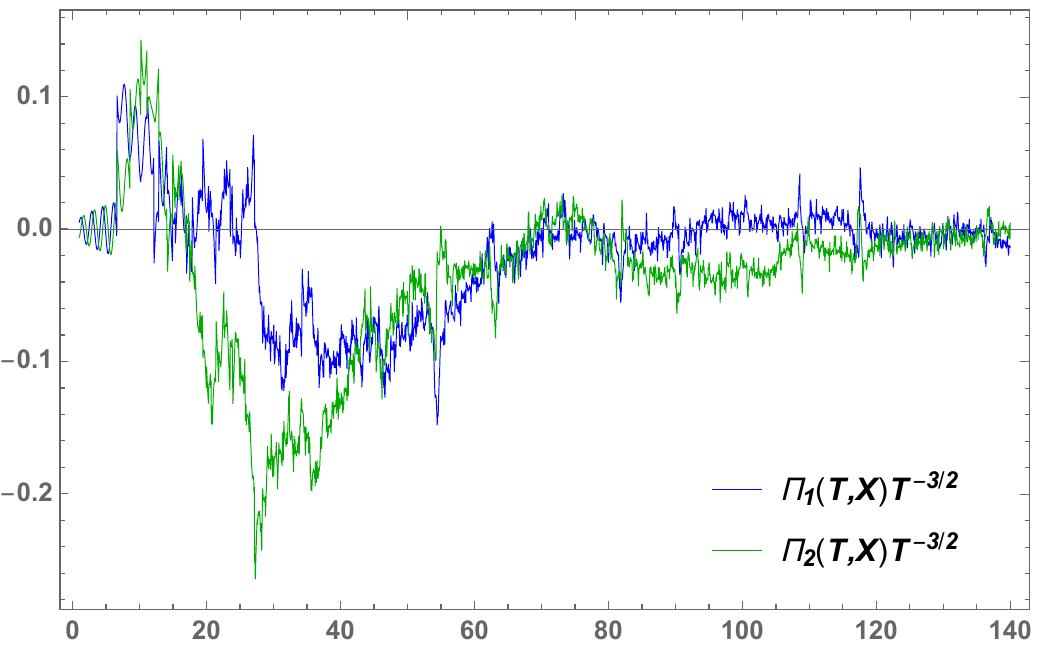}
		\caption{$\Pi_{1}(T, X) T^{-3/2}$ and $\Pi_{2}(T, X) T^{-3/2}$}
		\label{B}
	\end{subfigure}
\caption{Two different normalisations at $X = 46.97$}
\label{two-different-normalisations}
\end{figure}


Next we handle $\Sigma_{1}(T, X) = R(T, X) T^{-2}$ and $\Sigma_{2}(T, X) = Q(T, X) T^{-2}$. We plot these~sums in Fig.~\ref{amplitude} in the range $X \in [3, 20]$ with $T = 140$.
\begin{figure}
	\centering
	\includegraphics[width=0.8\linewidth]{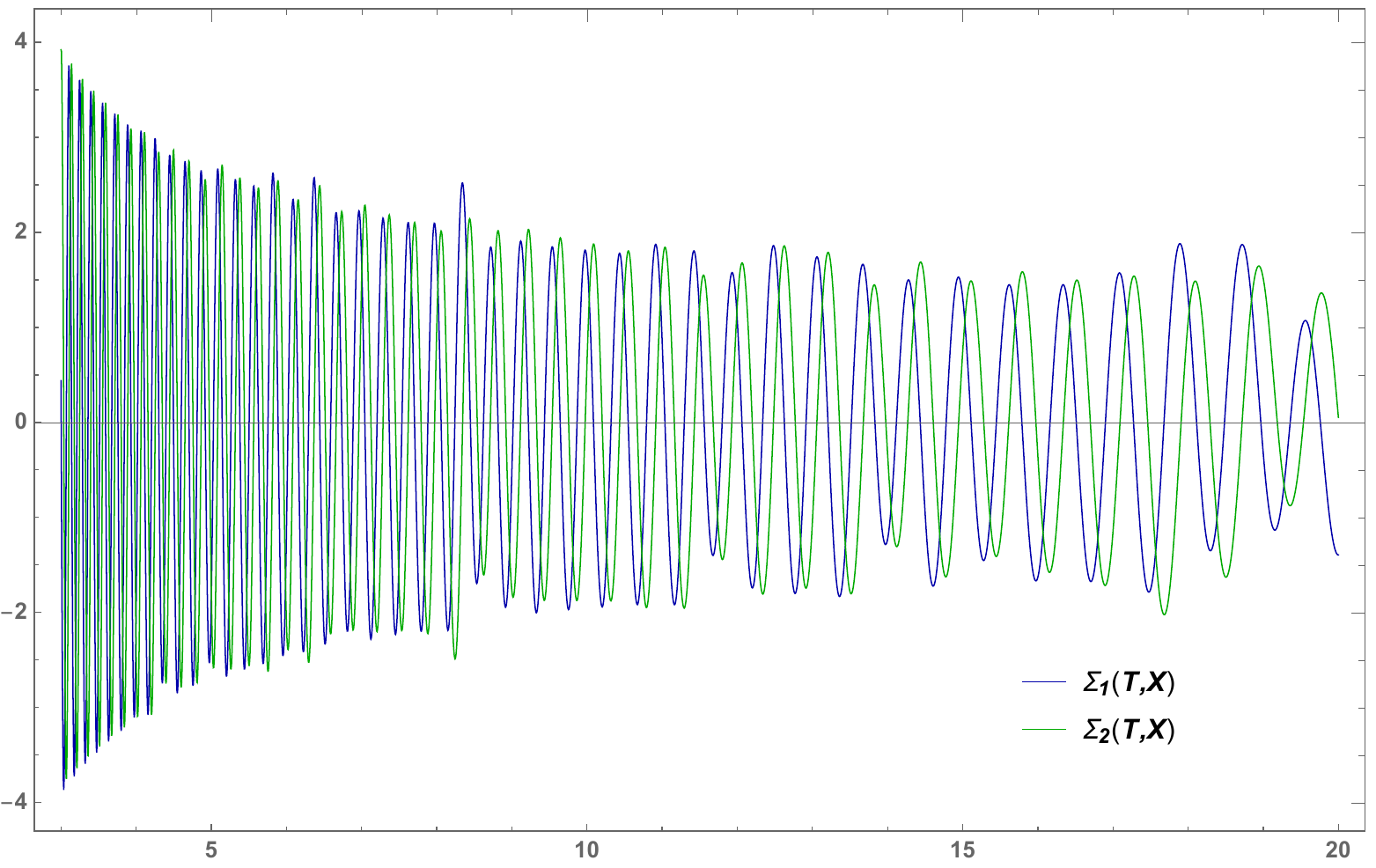}
	\caption{The normalised sums in terms of $X \in [3, 20]$ and $T = 140$}
	\label{amplitude}
\end{figure}
The oscillations in $R(T, X)$ and $Q(T, X)$ is much more conspicuous than in the case of the Riemann zeta function and this phenomenon leads to the belief that the main term in each sum should involve an oscillatory component. By Theorem~\ref{Laaksonen-analogue}, we know that such a component is expressed by the sine or cosine. We also observe that the oscillations in $\Sigma_{1}(T, X)$ and $\Sigma_{2}(T, X)$ are slightly out of synchronisation.
\begin{figure}
	\centering
	\includegraphics[width=0.8\linewidth]{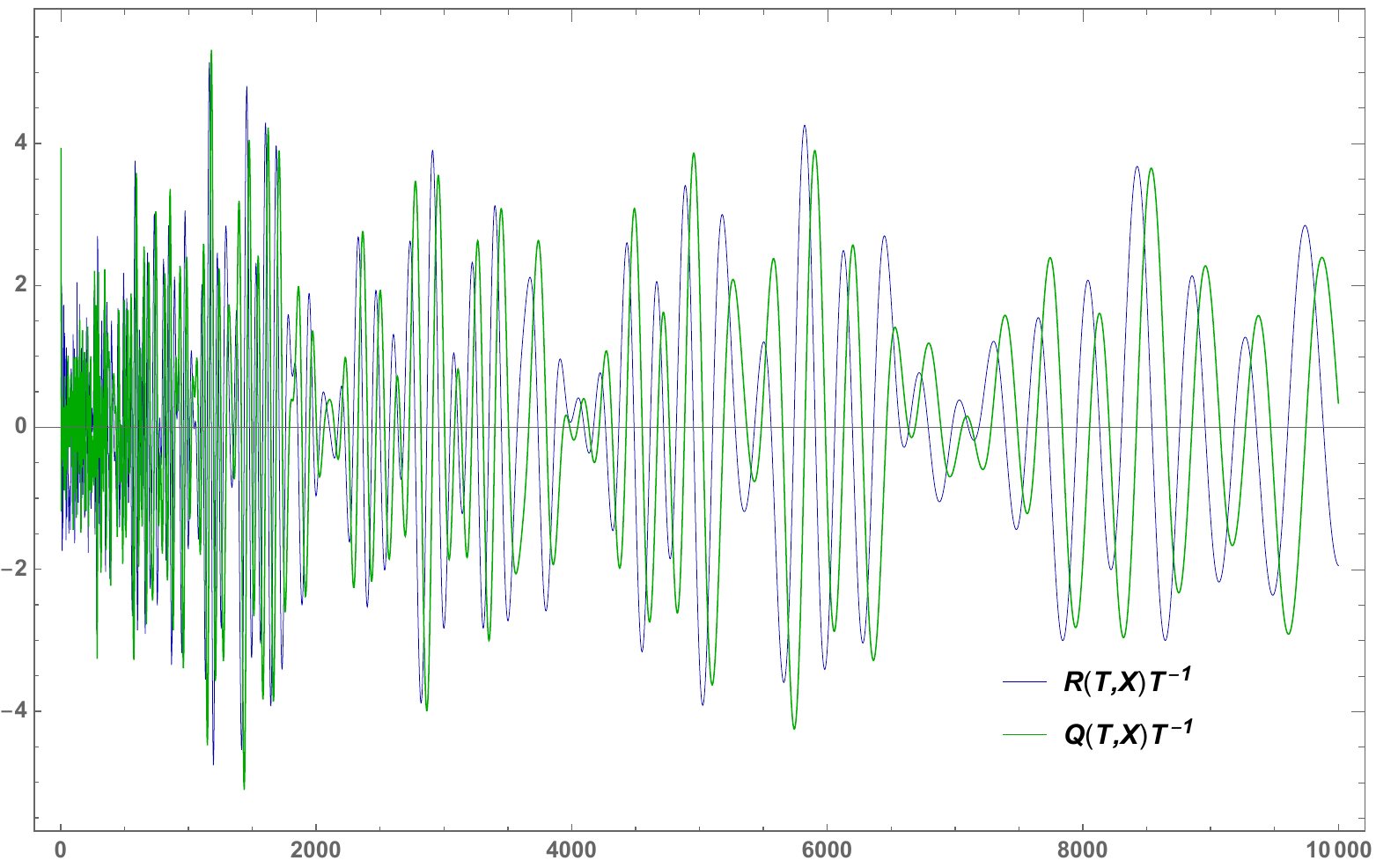}
	\caption{The normalised sums in terms of $X \in [3, 10000]$ and $T = 140$}
	\label{spectral-exponential-sums-for-the-Picard-group}
\end{figure}
Figure~\ref{spectral-exponential-sums-for-the-Picard-group} plots $R(T, X) T^{-1}$ and $Q(T, X) T^{-1}$ in the $X$-aspect with $T = 140$. Figure~\ref{spectral-exponential-sums-for-the-Picard-group} suggests that the bound $S(T, X) = O_{\epsilon}(T^{1+\epsilon} X^{\epsilon})$ may possibly hold, although we cannot dispose~of~$X^{\epsilon}$.

\section*{Acknowledgements}\label{acknowledgements}
The author wishes to thank Peter Sarnak for his valuable suggestions and Holger Then~for his list of the first 13950 eigenvalues for $\PSL_{2}(\mathcal{O})$. He also thanks Akio Fujii for teaching~his proof of results in~\cite{Fujii1984}, and Giacomo Cherubini, Dimitrios Chatzakos, and Niko Laaksonen for discussions. The author expresses his gratitude to Eren Mehmet K{\i}ral, Shin-ya Koyama, Maki Nakasuji, Hiroyuki Ochiai, Mikhail Smotrov, and Masao Tsuzuki for their feedback on earlier drafts. Special thanks are owed to the anonymous referees for their thorough review that helped improve the readability and rigor of the article.


\begin{thebibliography}{10}

\bibitem{AurichSteinerThen2012}
R.~Aurich, F.~Steiner, and H.~Then.
\newblock Numerical computation of {M}aass waveforms and an application to
  cosmology.
\newblock In J.~Bolte and F.~Steiner, editors, {\em Arithmetic Geometry and
  Applications in Quantum Chaos and Cosmology}, volume 307 of {\em London Math.
  Soc. Lecture Note Ser.}, pages 229--269. Cambridge Univ. Press, Cambridge,
  2012.

\bibitem{Avdispahic2018-2}
M.~Avdispahi{\'{c}}.
\newblock Errata and addendum to ``{O}n the prime geodesic theorem for
  hyperbolic 3-manifolds'' {M}ath. {N}achr. \textbf{291} (2018), no. 14--15,
  2160--2167.
\newblock {\em Math. Nachr.}, pages 1--3, 2018.

\bibitem{Avdispahic2018}
M.~Avdispahi{\'{c}}.
\newblock On the prime geodesic theorem for hyperbolic 3-manifolds.
\newblock {\em Math. Nachr.}, 291(14--15):2160--2167, 2018.

\bibitem{BalkanovaChatzakosCherubiniFrolenkovLaaksonen2019}
O.~Balkanova, D.~Chatzakos, G.~Cherubini, D.~Frolenkov, and N.~Laaksonen.
\newblock Prime geodesic theorem in the 3-dimensional hyperbolic space.
\newblock {\em Trans. Amer. Math. Soc.}, 372(8):5355--5374, 2019.

\bibitem{BalkanovaFrolenkov2019-2}
O.~Balkanova and D.~Frolenkov.
\newblock Bounds for a spectral exponential sum.
\newblock {\em J. London Math. Soc. (2)}, 99(2):249--272, 2019.

\bibitem{BalkanovaFrolenkov2022}
O.~Balkanova and D.~Frolenkov.
\newblock The second moment of symmetric square {$L$}-functions over {G}aussian
  integers.
\newblock {\em Proc. Roy. Soc. Edinburgh Sect. A}, 152(1):54--80, 2022.

\bibitem{BalkanovaFrolenkovRisager2022}
O.~Balkanova, D.~Frolenkov, and M.~S. Risager.
\newblock Prime geodesics and averages of the {Z}agier {$L$}-series.
\newblock {\em Math. Proc. Cambridge Philos. Soc.}, 172(3):705--728, 2022.

\bibitem{BalkanovaFrolenkovWu2022}
O.~Balkanova, D.~Frolenkov, and H.~Wu.
\newblock {W}eyl's subconvex bound for cube-free {H}ecke characters: {T}otally
  real case.
\newblock {\em arXiv e-prints}, 73 pages, 2022.
\newblock \url{https://arxiv.org/abs/2108.12283}.

\bibitem{BalogBiroCherubiniLaaksonen2022}
A.~Balog, A.~Bir{\'{o}}, G.~Cherubini, and N.~Laaksonen.
\newblock {B}ykovskii-type theorem for the {P}icard manifold.
\newblock {\em Int. Math. Res. Not. IMRN}, 2022(3):1893--1921, 2022.

\bibitem{BalogBiroHarcosMaga2019}
A.~Balog, A.~Bir{\'{o}}, G.~Harcos, and P.~Maga.
\newblock The prime geodesic theorem in square mean.
\newblock {\em J. Number Theory}, 198:239--249, 2019.

\bibitem{Bonthonneau2017}
Y.~Bonthonneau.
\newblock Weyl laws for manifolds with hyperbolic cusps.
\newblock {\em arXiv e-prints}, 38 pages, 2017.
\newblock \url{https://arxiv.org/abs/1512.05794}.

\bibitem{BookerPlatt2019}
A.~R. Booker and D.~J. Platt.
\newblock {T}uring's method for the {S}elberg zeta-function.
\newblock {\em Commun. Math. Phys.}, 365(1):295--328, 2019.

\bibitem{Bykovskii1994}
V.~A. Bykovski{\u{\i}}.
\newblock Density theorems and the mean value of arithmetic functions in short
  intervals ({R}ussian).
\newblock {\em Zap. Nauchn. Sem. S.-Peterburg. Otdel. Mat. Inst. Steklov.
  (POMI)}, 212, 1994.
\newblock Anal. Teor. Chisel i Teor. Funktsi{\u{\i}}. 12, 196:56--70;
  translation in J. Math. Sci. (N.Y.), 83(6):720--730, 1997.

\bibitem{Cai2002}
Y.~Cai.
\newblock Prime geodesic theorem.
\newblock {\em J. Th{\'{e}}or. Nombres Bordeaux}, 14(2):59--72, 2002.

\bibitem{CarneiroVaaler2010}
E.~Carneiro and J.~Vaaler.
\newblock Some extremal functions in {F}ourier analysis. {I\hspace{-.1mm}I}.
\newblock {\em Trans. Amer. Math. Soc.}, 362(11):5803--5843, 2010.

\bibitem{ChatzakosCherubiniLaaksonen2022}
D.~Chatzakos, G.~Cherubini, and N.~Laaksonen.
\newblock Second moment of the prime geodesic theorem for {$\mathrm{PSL}(2,
  \mathbb{Z}[i])$}.
\newblock {\em Math. Z.}, 300(1):791--806, 2022.

\bibitem{CherubiniGuerreiro2018}
G.~Cherubini and J.~Guerreiro.
\newblock Mean square in the prime geodesic theorem.
\newblock {\em Algebra \& Number Theory}, 12(3):571--597, 2018.

\bibitem{ConreyIwaniec2000}
J.~B. Conrey and H.~Iwaniec.
\newblock The cubic moment of central values of automorphic {$L$}-functions.
\newblock {\em Ann. of Math. (2)}, 151(3):1175--1216, 2000.

\bibitem{DeshouillersIwaniec1986}
J.-M. Deshouillers and H.~Iwaniec.
\newblock The non-vanishing of {R}ankin-{S}elberg zeta-functions at special
  points.
\newblock In D.~Hejhal, P.~Sarnak, and A.~Terras, editors, {\em The Selberg
  Trace Formula and Related Topics, Proceedings of a Summer Research Conference
  held July 22--28 ({B}runswick, {M}aine, 1984)}, volume~53 of {\em Contemp.
  Math.}, pages 51--95. Amer. Math. Soc., Providence, RI, 1986.

\bibitem{ElstrodtGrunewaldMennicke1998}
J.~Elstrodt, F.~Grunewald, and J.~Mennicke.
\newblock {\em Groups acting on hyperbolic space: {H}armonic analysis and
  number theory}.
\newblock Springer Monographs in Mathematics. Springer-Verlag, Berlin,
  Heidelberg, 1998.

\bibitem{Fujii1982}
A.~Fujii.
\newblock On the uniformity of the distribution of the zeros of the {R}iemann
  zeta function. ({I\hspace{-.1mm}I}).
\newblock {\em Comment. Math. Univ. St. Paul.}, 31(1):99--113, 1982.

\bibitem{Fujii1984}
A.~Fujii.
\newblock Zeros, eigenvalues and arithmetic.
\newblock {\em Proc. Japan Acad. Ser. A Math. Sci.}, 60(1):22--25, 1984.

\bibitem{Gallagher1970}
P.~X. Gallagher.
\newblock A large sieve density estimate near $\sigma = 1$.
\newblock {\em Invent. Math.}, 11(4):329--339, 1970.

\bibitem{GradshteynRyzhik2007}
I.~S. Gradshteyn and I.~M. Ryzhik.
\newblock {\em Table of integrals, series, and products}.
\newblock Elsevier/Academic Press, Amsterdam, 7 edition, 2007.
\newblock Translated from the Russian, Translation edited and with a preface by
  Alan Jeffrey and Daniel Zwillinger.

\bibitem{GrunewaldHuntebrinker1996}
F.~Grunewald and W.~Huntebrinker.
\newblock A numerical study of eigenvalues of the hyperbolic {L}aplacian for
  polyhedra with one cusp.
\newblock {\em Experiment. Math.}, 5(1):57--80, 1996.

\bibitem{HardyLittlewoodPolya1934}
G.~Hardy, J.~E. Littlewood, and G.~P{\'{o}}lya.
\newblock {\em Inequalities}.
\newblock Cambridge Mathematical Library. Cambridge Univ. Press, New York,
  1934.

\bibitem{Hejhal1976-2}
D.~Hejhal.
\newblock The {S}elberg trace formula and the {R}iemann zeta function.
\newblock {\em Duke Math. J.}, 43(3):441--482, 1976.

\bibitem{Hejhal1976}
D.~Hejhal.
\newblock {\em The {S}elberg Trace Formula for {$\mathrm{PSL}(2, \mathbb{R})$
  I}}, volume 548 of {\em Lecture Notes in Mathematics}.
\newblock Springer-Verlag, Berlin Heidelberg, 1976.

\bibitem{Hejhal1983}
D.~Hejhal.
\newblock {\em The {S}elberg trace formula for {$\mathrm{PSL}(2, \mathbb{R})$
  I\hspace{-.1mm}I}}, volume 1001 of {\em Lecture Notes in Mathematics}.
\newblock Springer-Verlag, Berlin Heidelberg, 1983.

\bibitem{Huber1961}
H.~Huber.
\newblock {Z}ur analytischen {T}heorie hyperbolischer {R}aumformen und
  {B}ewegungsgruppen. {I\hspace{-.1mm}I}.
\newblock {\em Math. Ann.}, 142(4):385--398, 1961.

\bibitem{Huber1961-2}
H.~Huber.
\newblock {Z}ur analytischen {T}heorie hyperbolischer {R}aumformen und
  {B}ewegungsgruppen. {I\hspace{-.1mm}I}. {N}achtrag zu {M}ath. {A}nn.
  {\textbf{142}}, 385--398 (1961).
\newblock {\em Math. Ann.}, 143(5):463--464, 1961.

\bibitem{Huntebrinker1991}
W.~Huntebrinker.
\newblock Numerische {B}estimmung von {E}igenwerten des {L}aplace-{O}perators
  auf hyperbolischen {R}{\"{a}}umen mit adaptiven
  {F}inite-{E}lement-{M}ethoden.
\newblock {\em Bonner Math. Schiften}, 225:1--34, 1991.

\bibitem{Huntebrinker1995}
W.~Huntebrinker.
\newblock {\em Numerische {B}estimmung von {E}igenwerten des
  {L}aplace-{B}eltrami-{O}perators auf dreidimensionalen hyperbolischen
  {R}{\"{a}}umen mit {F}inite-{E}lement-{M}ethoden}.
\newblock PhD thesis, Univ. D{\"{u}}sseldorf, D{\"{u}}sseldorf, Germany, May
  1995.

\bibitem{Iwaniec1980}
H.~Iwaniec.
\newblock {F}ourier coefficients of cusp forms and the {R}iemann zeta-function.
\newblock {\em S\'{e}minaire de Th\'{e}orie des Nombres (Ann\'{e}e
  1979--1980)}, pages 1--36, 1980.
\newblock Exp. No. 18, Univ. Bordeaux I, Talence.

\bibitem{Iwaniec1982}
H.~Iwaniec.
\newblock Mean values for {F}ourier coefficients of cusp forms and sums of
  {K}loosterman sums.
\newblock In J.~V. Armitage, editor, {\em Journ\'{e}es Arithm\'{e}tiques},
  volume~56 of {\em London Math. Soc. Lecture Note Ser.}, pages 306--321,
  Exeter, England, 1982. Cambridge Univ. Press, Cambridge.

\bibitem{Iwaniec1984-2}
H.~Iwaniec.
\newblock Non-holomorphic modular forms and their applications.
\newblock In R.~A. Rankin, editor, {\em Modular forms (Durham, 1983)}, Ellis
  Horwood Ser. Math. Appl., pages 157--196. Statist. Oper. Res., Horwood,
  Chichester, 1984.

\bibitem{Iwaniec1984}
H.~Iwaniec.
\newblock Prime geodesic theorem.
\newblock {\em J. Reine Angew. Math.}, 349:136--159, 1984.

\bibitem{Iwaniec1990}
H.~Iwaniec.
\newblock Small eigenvalues of {L}aplacian for {$\Gamma_{0}(N)$}.
\newblock {\em Acta Arith.}, 56(1):65--82, 1990.

\bibitem{Iwaniec2002}
H.~Iwaniec.
\newblock {\em Spectral methods of automorphic forms}, volume~53 of {\em
  Graduate Studies in Mathematics}.
\newblock Amer. Math. Soc., Providence, RI; Revista Matem{\'{a}}tica
  Iberoamericana, Madrid, 2 edition, 2002.

\bibitem{Kaneko2020}
I.~Kaneko.
\newblock The second moment for counting prime geodesics.
\newblock {\em Proc. Japan Acad. Ser. A Math. Sci.}, 96(1):7--12, 2020.

\bibitem{Kaneko2022-2}
I.~Kaneko.
\newblock Spectral exponential sums on hyperbolic surfaces.
\newblock {\em to appear in Ramanujan J.}, 10 pages, 2022.
\newblock \url{https://arxiv.org/abs/1905.00681}.

\bibitem{KanekoKoyama2022}
I.~Kaneko and S.~Koyama.
\newblock Euler products of {S}elberg zeta functions in the critical strip.
\newblock {\em Ramanujan J.}, 59(2):437--458, 2022.

\bibitem{Koyama2001}
S.~Koyama.
\newblock Prime geodesic theorem for the {P}icard manifold under the
  mean-{L}indel{\"{o}}f hypothesis.
\newblock {\em Forum Math.}, 13(6):781--793, 2001.

\bibitem{Koyama2004}
S.~Koyama.
\newblock The first eigenvalue problem and tensor products of zeta functions.
\newblock {\em Proc. Japan Acad. Ser. A Math. Sci.}, 80(5):35--39, 2004.

\bibitem{Koyama2016}
S.~Koyama.
\newblock Refinement of prime geodesic theorem.
\newblock {\em Proc. Japan Acad. Ser. A Math. Sci.}, 92:77--81, 2016.

\bibitem{Kuznetsov1978}
N.~V. Kuznetsov.
\newblock The arithmetic form of {S}elberg's trace formula and the distribution
  of norms of the primitive hyperbolic classes of the modular group.
\newblock Preprint (Khabarovsk), 1978.

\bibitem{Landau1912}
E.~Landau.
\newblock {\"{U}}ber die {N}ullstellen der {Z}etafunktion.
\newblock {\em Math. Ann.}, 71(4):548--564, 1912.

\bibitem{LiVillavert2011}
C.~Li and J.~Villavert.
\newblock An extension of the {H}ardy-{L}ittlewood-{P\'{o}}lya inequality.
\newblock {\em Acta Math. Sci. Ser. B Engl. Ed.}, 31(6):2285--2288, 2011.

\bibitem{LuoSarnak1995}
W.~Luo and P.~Sarnak.
\newblock Quantum ergodicity of eigenfunctions on
  {$\mathrm{PSL}_{2}(\mathbf{Z}) \backslash \mathbf{H}^{2}$}.
\newblock {\em Publ. Math. Inst. Hautes {\'{E}}tudes Sci.}, 81:207--237, 1995.

\bibitem{Matthies1995}
C.~Matthies.
\newblock {\em {P}icards {B}illard. {E}in {M}odell f{\"{u}}r {A}rithmetisches
  {Q}uantenchaos in drei {D}imensionen}.
\newblock PhD thesis, des Fachbereiches Physik der Universit{\"{a}}t Hamburg,
  Hamburg, Germany, June 1995.

\bibitem{MichelVenkatesh2006}
P.~Michel and A.~Venkatesh.
\newblock Equidistribution, {$L$}-functions and ergodic theory: {O}n some
  problems of {Y}u. {L}innik.
\newblock In {\em International Congress of Mathematicians. Vol.
  {I\hspace{-.1mm}I}}, pages 421--457. Eur. Math. Soc., Z{\"{u}}rich, 2006.

\bibitem{MichelVenkatesh2010}
P.~Michel and A.~Venkatesh.
\newblock The subconvexity problem for {$\mathrm{GL}_{2}$}.
\newblock {\em Publ. Math. Inst. Hautes {\'{E}}tudes Sci.}, 111:171--271, 2010.

\bibitem{Motohashi1996}
Y.~Motohashi.
\newblock A trace formula for the {P}icard group. {I}.
\newblock {\em Proc. Japan Acad. Ser. A Math. Sci.}, 72(8):183--186, 1996.

\bibitem{Motohashi1997-2}
Y.~Motohashi.
\newblock {\em Spectral theory of the {R}iemann zeta-function}, volume 127 of
  {\em Cambridge Tracts in Math.}
\newblock Cambridge Univ. Press, Cambridge, 1997.

\bibitem{Motohashi1997}
Y.~Motohashi.
\newblock Trace formula over the hyperbolic upper half space.
\newblock In Y.~Motohashi, editor, {\em Analytic Number Theory}, volume 247 of
  {\em London Math. Soc. Lecture Note Ser.}, pages 265--286. Cambridge Univ.
  Press, Cambridge, 1997.

\bibitem{Motohashi2001}
Y.~Motohashi.
\newblock New analytic problems over imaginary quadratic number fields.
\newblock In M.~Jutila and T.~Mets{\"{a}}nkyl{\"{a}}, editors, {\em Number
  Theory: Proceedings of the Turku Symposium on Number Theory in Memory of
  Kustaa Inkeri, May 31--June 4, 1999}, De Gruyter Proceedings in Mathematics,
  pages 255--279, Turku, Finland, 2001. de Gruyter, Berlin.

\bibitem{Nakasuji2000}
M.~Nakasuji.
\newblock Prime geodesic theorem via the explicit formula of {$\Psi$} for
  hyperbolic 3-manifolds.
\newblock Research Report KSTS/RR-00/005, Department of Mathematics, Keio
  University, May 2000.
\newblock 38 pages, available at
  \url{http://www.math.keio.ac.jp/academic/research_pdf/report/2000/00005.pdf}.

\bibitem{Nakasuji2001}
M.~Nakasuji.
\newblock Prime geodesic theorem via the explicit formula of {$\Psi$} for
  hyperbolic 3-manifolds.
\newblock {\em Proc. Japan Acad. Ser. A Math. Sci.}, 77(7):130--133, 2001.

\bibitem{Nakasuji2012}
M.~Nakasuji.
\newblock Generalized {R}amanujan conjecture over general imaginary quadratic
  fields.
\newblock {\em Forum Math.}, 24(1):85--98, 2012.

\bibitem{Nelson2020}
P.~D. Nelson.
\newblock {E}isenstein series and the cubic moment for {$\mathrm{PGL}_{2}$}.
\newblock {\em arXiv e-prints}, 77 pages, 2020.
\newblock \url{https://arxiv.org/abs/1911.06310}.

\bibitem{Odlyzko2014}
A.~Odlyzko.
\newblock Tables of zeros of the {R}iemann zeta function.
\newblock \url{http://www.dtc.umn.edu/~odlyzko/zeta_tables/}, 2014.

\bibitem{PetridisRisager2017}
Y.~N. Petridis and M.~S. Risager.
\newblock Local average in hyperbolic lattice point counting, with an appendix
  by {N}iko {L}aaksonen.
\newblock {\em Math. Z.}, 285(3--4):1319--1344, 2017.

\bibitem{PetrowYoung2020}
I.~Petrow and M.~P. Young.
\newblock The {W}eyl bound for {D}irichlet {$L$}-functions of cube-free
  conductor.
\newblock {\em Ann. of Math. (2)}, 192(2):437--486, 2020.

\bibitem{PetrowYoung2022}
I.~Petrow and M.~P. Young.
\newblock The fourth moment of {D}irichlet {$L$}-functions along a coset and
  the {W}eyl bound.
\newblock {\em to appear in Duke Math. J.}, 57 pages, 2022.
\newblock \url{https://arxiv.org/abs/1908.10346}.

\bibitem{Sarnak1983}
P.~Sarnak.
\newblock The arithmetic and geometry of some hyperbolic three manifolds.
\newblock {\em Acta Math.}, 151(3--4):253--295, 1983.

\bibitem{Selberg1989}
A.~Selberg.
\newblock {\em Collected Papers}, volume~1 of {\em Springer Collected Works in
  Mathematics}.
\newblock Springer-Verlag Berlin Heidelberg, 1989.

\bibitem{Shimura1994}
G.~Shimura.
\newblock {\em Introduction to the arithmetic theory of automorphic functions}.
\newblock Princeton Univ. Press, Princeton, NJ, 1994.
\newblock Reprint of the 1971 original.

\bibitem{SmotrovGolovchansky1991}
M.~N. Smotrov and V.~V. Golovchansky.
\newblock Small eigenvalues of the {L}aplacian on {$\Gamma \backslash H_{3}$}
  for {$\Gamma = \mathrm{PSL}_{2}(\mathbb{Z}[i])$}.
\newblock Preprint 91--040, Universit{\"{a}}t Bielefeld, SBF 343, 1991.

\bibitem{SoundararajanYoung2013}
K.~Soundararajan and M.~P. Young.
\newblock The prime geodesic theorem.
\newblock {\em J. Reine Angew. Math.}, 676:105--120, 2013.

\bibitem{Steil1999}
G.~Steil.
\newblock Eigenvalues of the {L}aplacian for {B}ianchi groups.
\newblock In D.~A. Hejhal, J.~Friedman, M.~C. Gutzwiller, and A.~M. Odlyzko,
  editors, {\em Emerging applications of number theory (Minneapolis, MN,
  1996)}, volume 109 of {\em IMA Vol. Math. Appl.}, pages 617--641, New York,
  NY, 1999. Springer.

\bibitem{Stramm1994}
K.~Stramm.
\newblock Kleine {E}igenwerte des {L}aplace-{O}perators zu {K}ongruenzgruppen.
\newblock {\em Schriftenreihe Math. Inst. Univ. M{\"{u}}nster, Ser. 3}, 11:92
  pages, 1994.

\bibitem{Szmidt1983}
J.~Szmidt.
\newblock The {S}elberg trace formula for the {P}icard group {$\mathrm{SL}(2,
  \mathbb{Z}[i])$}.
\newblock {\em Acta Arith.}, 42(4):391--424, 1983.

\bibitem{Venkov1979}
A.~B. Venkov.
\newblock Remainder term in the {W}eyl-{S}elberg asymptotic formula
  ({R}ussian).
\newblock {\em Zap. Nauchn. Sem. Leningrad. Otdel. Mat. Inst. Steklov. (LOMI)},
  91:5--24, 1979.

\bibitem{Venkov1982}
A.~B. Venkov.
\newblock {\em Spectral Theory of Automorphic Functions}, volume 153 of {\em
  Trudy Math. Inst. Steklov}.
\newblock Amer. Math. Soc., 1982.

\bibitem{Venkov1990}
A.~B. Venkov.
\newblock {\em Spectral theory of automorphic functions and its applications},
  volume~51 of {\em Mathematics and its Applications (Soviet Series)}.
\newblock Kluwer Academic Publishers Group, 1990.
\newblock Translated from the Russian by N. B. Lebedinskaya.

\bibitem{Wu2022}
H.~Wu.
\newblock On {M}otohashi's formula.
\newblock {\em to appear in Trans. Amer. Math. Soc.}, 49 pages, 2022.
\newblock \url{https://arxiv.org/abs/2001.09733}.

\end{thebibliography}

\end{document}